\documentclass[reqno,12pt,letterpaper]{amsart}
\usepackage{amsmath,amssymb,amsthm,graphicx,mathrsfs,url}
\usepackage{enumitem}

\usepackage[russian, french, english]{babel}
\usepackage{graphicx}
\usepackage{amsmath, amscd}
\usepackage[usenames,dvipsnames]{color}
\usepackage[colorlinks=true, linkcolor=Red, citecolor=Green]{hyperref}
\usepackage[displaymath,mathlines]{lineno}

\usepackage[margin=1.2in]{geometry}

\newtheorem{theo}{Theorem}
\newtheorem{prop}{Proposition}[section]

\newtheorem{lemm}[prop]{Lemma}
\theoremstyle{definition}
\newtheorem{corr}[theo]{Corollary}

\newtheorem{rem}[prop]{Remark}

\numberwithin{equation}{section}

\newcommand{\R}{\mathbb{R}}
\newcommand{\N}{\mathbb{N}}
\def\C{\mathbb {C}}
\newcommand{\Z}{\mathbb{Z}}

\newcommand{\om}{\omega}

\newcommand{\dd}{\mathrm{d}}
\newcommand{\Lie}{\mathcal{L}}

\newcommand{\e}{\mathrm{e}}

\newcommand{\wt}{\widetilde}
\newcommand{\cal}{\mathcal}
\newcommand{\Pbf}{\mathbf{P}}
\newcommand{\Qbf}{\mathbf{Q}}
\newcommand{\Rbf}{\mathbf{R}}
\newcommand{\Abf}{\mathbf{A}}
\newcommand{\Tbf}{\mathbf{T}}
\newcommand{\End}{\mathrm{End}}
\newcommand{\ubf}{\mathbf{u}}
\newcommand{\clos}{\mathrm{clos}}

\let\Im=\Imag

\let\Re=\Real

\DeclareMathOperator{\supp}{supp}

\DeclareMathOperator{\WF}{WF}

\DeclareMathOperator{\tr}{tr}
\DeclareMathOperator{\id}{Id}

\def\bc{{\mathcal B}}

\title{Dynamical zeta functions for billiards}
\author[Y.~Chaubet]{Yann Chaubet}
\address{Université de Nantes, Laboratoire de mathématiques Jean Leray (UMR CNRS 6629)
Département de mathématiques,
2 rue de la Houssinière, 44322 Nantes Cédex 3, France}
\email{yann.chaubet@univ-nantes.fr}
\author[V.~Petkov]{Vesselin Petkov}
\address{Universit\'e de Bordeaux, Institut de Math\'ematiques de Bordeaux, 351, Cours de la Lib\'eration, 33405 Talence, France}
\email{petkov@math.u-bordeaux.fr}

\begin{document}

\maketitle


\begin{abstract} Let $D \subset \R^d,\: d \geq 2,$ be the union of a finite collection of pairwise disjoint strictly convex compact obstacles. Let $\mu_j \in \C,\: \Im \mu_j > 0$ be the resonances of the Laplacian in the  exterior of  $D$ with Neumann or Dirichlet boundary condition on $\partial D$. For $d$ odd, $u(t) = \sum_j  \e^{i |t| \mu_j}$ is a distribution in $ \mathcal{D}'(\R \setminus \{0\})$  and the Laplace transforms of the leading singularities of $u(t)$ yield the dynamical zeta functions $\eta_{\mathrm N},\: \eta_{\mathrm D}$ for Neumann and Dirichlet boundary conditions, respectively. These zeta functions play a crucial role in the analysis of the distribution of the resonances. Under a non-eclipse condition, for $d \geqslant 2$ we show that $\eta_{\mathrm N}$ and  $\eta_\mathrm D$ admit a meromorphic continuation to the whole complex plane. In the particular case when the boundary $\partial D$ is real analytic, by using a result of Fried \cite{fried1995meromorphic}, we prove that the function $\eta_\mathrm{D}$ cannot be entire. Following  Ikawa \cite{ikawa1988zeta}, this implies the existence of a strip $\{z \in \C: \: 0 < \Im z \leqslant\alpha\}$ containing an infinite number of resonances $\mu_j$ for the Dirichlet problem. Moreover, for $\alpha \gg 1$ we obtain a lower bound for the resonances lying in this strip.
\end{abstract}

\section{Introduction}

Let $D_1, \dots, D_r \subset \R^d,\: {r \geqslant 3},\: d \geqslant 2,$ be compact strictly convex disjoint obstacles with smooth boundary and let $D = \bigcup_{j= 1}^r D_j.$  We assume that every $D_j$ has non-empty interior and throughout this paper we suppose the following non-eclipse condition
\begin{equation}\label{eq:1.1}
D_k \cap {\rm convex}\: {\rm hull} \: ( D_i \cup D_j) = \emptyset, 
\end{equation} 
for any $1 \leqslant i, j, k \leqslant r$ such that $i \neq k$ and $j \neq k$.
Under this condition all periodic trajectories for the billiard flow in $\Omega  = \R^d \setminus \mathring{D}$ are ordinary reflecting ones without tangential intersections to the boundary of $D$. Notice that if (\ref{eq:1.1}) is not satisfied, for generic perturbations of $\partial D$  all periodic reflecting trajectories in $\Omega$ have no tangential intersections to $\partial D$ (see Theorem 6.3.1 in \cite{petkov2017geometry}). We consider the (non-grazing) billiard flow {$\varphi_t$} (see {\S\ref{subsec:smooth}} for a precise definition).  In this paper the periodic trajectories will be called periodic rays and we refer to Chapter 2 in \cite{petkov2017geometry} for basic definitions. For any periodic trajectory $\gamma$, denote by  $\tau(\gamma) > 0$ its period, by $\tau^\sharp(\gamma) > 0$ its primitive period, and by $m(\gamma)$ the number of reflections of $\gamma$ at the obstacles. Denote by  $P_\gamma$ the associated linearized Poincar\'e map (see section 2.3 in \cite{petkov2017geometry} and Appendix A for the definition). Let $\mathcal{P}$ be the set of all periodic rays. The counting function of the lengths of periodic rays satisfies 
\begin{equation} \label{eq:per}
\sharp\{\gamma \in \mathcal{P}:\: \tau^\sharp(\gamma) \leqslant x\} \sim \frac{\e^{a x}}{a x} , \quad x \to + \infty,
\end{equation}
for some $a > 0$ (see for instance, \cite[Theorem 6.5] {Parry1990} for weak-mixing suspension symbolic flow and \cite{ikawa1990poles}, \cite{morita1991symbolic}). In contrast to the case $r = 2$, for $r \geqslant 3$ there exists an infinite number of primitive periodic trajectories and we have (see Corollary 2.2.5 in \cite{petkov2017geometry}) the estimate
\[ \sharp\{\gamma \in \mathcal P:\: \tau(\gamma) \leqslant x\} \leqslant \e^{ a_1 x}, \: x > 0\]
with $a_1 > a.$
Moreover, for some positive constants $C_1, b_1, b_2$ we have (see for instance \cite[Appendix] {petkov1999zeta})
\[
C_1 \e^{b_1 \tau(\gamma)} \leqslant |\det(\mathrm{Id} - P_{\gamma})| \leqslant \e^{b_2 \tau(\gamma)}, \quad \gamma \in \mathcal{P}.
\]
By using these estimates, we define for $\Re(s) \gg 1$ two Dirichlet series
\[
\eta_\mathrm{N}(s) = \sum_{\gamma \in \mathcal{P}} \frac{\tau^\sharp(\gamma)\e^{-s\tau(\gamma)}}{|\det(\mathrm{Id}-P_\gamma)|^{1/2} }, \quad \eta_\mathrm{D}(s) = \sum_{\gamma \in \mathcal{P}} (-1)^{m(\gamma)} \frac{ \tau^\sharp(\gamma) \e^{-s\tau(\gamma)}}{|\det(\mathrm{Id}-P_\gamma)|^{1/2} },
\]
 Here the sums run over all oriented periodic rays. Notice that some periodic rays have only one orientation, while others admits two (see \S\ref{subsec:orientation}). On the other hand, the length $\tau^{\sharp}(\gamma)$, the period $\tau(\gamma)$ and $|\det(\mathrm{Id} - P_{\gamma})|^{1/2}$ are independent of the orientation of $\gamma.$ 

The series $\eta_\mathrm N(s), \: \eta_\mathrm D(s)$ are related to the  resonances of the self-adjoint operators $-\Delta_\mathrm b, \:\mathrm b = \mathrm N, \mathrm D,$ acting on domains ${\mathcal D}_\mathrm b \subset{\mathcal H} = L^2(\R^d \setminus D),$ with Neumann and Dirichlet boundary conditions on $\partial D$, respectively. To explain this relation,  consider the  resolvents  
\[{\mathcal R}_\mathrm b(\mu) = \left(-\Delta_\mathrm b - \mu^2\right)^{-1}: \: {\mathcal H} \rightarrow {\mathcal D}_{\mathrm b},\] 
which are analytic in $\{\mu \in \C:\: \Im \mu < 0\}$. Then ${\mathcal R}_{\mathrm b}(\mu): {\mathcal H}_{\mathrm{comp}} \to {\mathcal D}_{\mathrm{b,loc}}$ has a meromorphic continuation to $\mu \in \C$ if $d$ is odd, and to the logarithmic covering of $\C\setminus \{0\}$ if $d$ is even (see  \cite[Chapter 5]{lax1989scattering} for $d$ odd and  \cite[Chapter 4]{dyatlov2019mathematical}). These resolvents have poles in $\{z \in \C: \: \Im z > 0\}$ and the poles $\mu_j$ are called {\it resonances}. Introduce the distribution $u \in {\mathcal D}'(\R)$ by the formula 
\[
\langle u, \varphi \rangle = 2\, {\rm tr}_{L^2(\R^d)} \:  \int_\R \Bigl(\cos (t \sqrt{- \Delta_{\mathrm{b}}})\oplus 0 - \cos(t \sqrt{-\Delta_0})\Bigr) \varphi(t) \dd t, \quad \varphi \in C^\infty_c(\R).
 \]
Here $\Delta_0$ is the free Laplacian in $\R^d$ and writing $L^2(\R^d) = L^2(\R^d \setminus D) \oplus L^2(D),$ the operator $\cos(t \sqrt{-\Delta_{\mathrm{b}}})\oplus 0$ acts as 0 on $L^2(D)$.
Then for $d$ odd, Melrose \cite{melrose1982scattering}  (see also \cite{bardos1982relation} for a slightly weaker result) proved that $u\vert_{\R \setminus \{0\}}$ is a distribution in ${\mathcal D}'(\R\setminus \{0\})$ having the representation
\[
u(t) = \sum_j m(\mu_j) \e^{i |t| \mu_j},
\] 
where $m(\mu_j)$ is the multiplicity of $\mu_j.$
In the notation we omitted the dependence on the boundary conditions. The above series converges in the sense of distributions since we have a bound $\sharp\{ \mu_j: \: |\mu_j| \leqslant r\} \leqslant Cr^d$ for all $r > 0$ (see {Section 4.3} in \cite{dyatlov2019mathematical}) and we may express the action $\langle u, \varphi \rangle$ on functions $\varphi \in C_0^{\infty}(\R^+)$ by the derivatives of $\varphi$ (see Lemma B.1 in Appendix B). The reader may consult \cite{zworski1997poisson} and \cite{dyatlov2019mathematical} for the form of the singularity of $u(t)$ at $t=0$, though it is not important for our exposition.

 For $d$ even, the situation is more complicated since the resonances are defined in a logarithmic covering $ \exp^{-1} (\C \setminus \{0\})$ of $\C \setminus \{0\}$ and the arguments of the resonances are not bounded (see \cite{vodev1994}, \cite{vodev1994even}).  Let $\Lambda = \C \setminus \e^{i \frac{\pi}{2} \overline{{\R}^+}}$ and for $0 < \omega < \frac{\pi}{2}$ introduce
\[
{\Lambda_{\omega} = \{ \mu \in \Lambda:\: 0 < \Im \mu \leqslant \omega |\Re \mu|, \: 0 < \arg \mu <  \pi\}}.
\]
 Choose a function $\psi$ in $C_c^{\infty}(\R; [0, 1])$ equal to 1 in a neighborhood of 0 and denote by $\sigma_{\mathrm{b}}(\lambda): = \frac{i}{2 \pi} \log \: {\rm det}\: S_{\mathrm{b}}(\lambda)$ the scattering phase related to $-\Delta_{\mathrm{b}}$, where $S_{\mathrm{b}}(\lambda)$ is the scattering matrix (see Definition 4.25 in \cite{zworski1998poisson} for $S_{\mathrm{b}}(\lambda)$). Following the work of Zworski {(Theorem 1 in \cite{zworski1998poisson})}, there exists a function $v_{\omega, \psi} \in C^{\infty} (\R \setminus \{0\})$ such that {for even dimension $d$} one has in the sense of distributions ${\mathcal D}' (\R \setminus \{0\})$
\begin{equation}\label{eq:tracezworski}
\begin{aligned}
u(t) =  &\sum_{\mu_j \in \Lambda_\omega} m(\mu_j) \e^{i \mu_j |t|} + m(0) \\
&\qquad + 2 \int_0^{\infty} \psi(\lambda) \frac{\dd \sigma_{\mathrm{b}}}{\dd \lambda}(\lambda) \cos(t \lambda) \dd\lambda + v_{\omega, \psi}(t),
\end{aligned}
\end{equation}
where $m(0)$ is a constant and
\[ \partial_t^k v_{\omega, \psi}(t) = {\mathcal O}(|t|^{-N}),\: \forall k, \: \forall N,\: |t| \to \infty.\]
 The reader may consult  \cite{sjostrand1997trace} for a local trace formula involving the resonances.
Concerning the singularities of the distribution $u(t) \in {\mathcal D}'(\R \setminus \{0\})$,  from \cite{bardos1982relation} it follows that
  \[\text{sing supp} \: u \subset \{ \pm\tau(\gamma):\: \gamma \in {\mathcal P}\}.\]
Under the condition (\ref{eq:1.1}), every periodic trajectory $\gamma$ with period $T = \tau(\gamma)$ is an ordinary reflecting ray and the singularity of $u$ at {$t=T$} was described by Guillemin and Melrose \cite{guillemin1979poisson}. More precisely, the singularity at {$T$} has the form
\begin{equation}\label{eq:sing}
{\sum_{\gamma \in {\mathcal P}, \tau(\gamma) = T} (-1)^{m(\gamma)} \tau^{\sharp}(\gamma) |\det(\mathrm{Id} - P_{\gamma})|^{-1/2}\delta(t - T) + L^1_{\mathrm {loc}}(\R)}
\end{equation} 
(see for instance, Corollary 4.3.4 in \cite{petkov2017geometry}), where for the Neumann problem the factor $(-1)^{m(\gamma)}$ must be omitted.
Taking the sum of the Laplace transforms of the singularities of $u(t)\vert_{\R^+}$ at  $\tau(\gamma), \:\gamma \in {\mathcal P}$, we obtain the Dirichlet series $\eta_\mathrm N(s), \: \eta_\mathrm D(s)$. 

The poles of $\eta_\mathrm N(s)$ and $\eta_\mathrm D(s)$ are important for the analysis of the distribution of the resonances (see \cite{ikawa1988analysis, ikawa1990poles,ikawa1990zeta,ikawa1992addentum,stoyanov2009poles,petkov2008zeta}  and the papers cited there). By using the Ruelle transfer operator and symbolic dynamics  (see \cite{ikawa1990zeta,petkov1999zeta,stoyanov2009poles,morita1991symbolic}), a meromorphic continuation of $s \mapsto \eta_\mathrm N(s), \eta_\mathrm D(s)$ has been proved in a domain {$s_0 - \epsilon \leqslant \Re s$} with a suitable {$\epsilon > 0,$} where $s_0$ is the abscissa of absolute convergence of the Dirichlet series $\eta_{\mathrm N}(s),\:\eta_\mathrm D(s).$ In particular,  these results imply the asymptotic (\ref{eq:per}). Recently, a meromorphic continuation to $\C$ of the series
\begin{equation} \label{eq:zeta,Poin}
\sum_{\gamma \in \mathcal{P}} \frac{\tau^\sharp(\gamma)\e^{-s\tau(\gamma)}}{|\det(\mathrm{Id}-P_\gamma)| }, \quad \Re(s) \gg 1,
\end{equation} 
has been proved by Delarue--Sch\"{u}tte-Weich (see Theorem 5.8, \cite{delarue2022resonances}). We refer also to  \cite{Weich2023}  for results concerning weighted zeta functions. On the other hand, a meromorphic continuation to the whole complex {plane} of the semi-classical zeta function for contact Anosov flows was established by Faure--Tsujii {\cite{faure2017semiclassical}. Their zeta function is similar to the function $\zeta_{\mathrm N}(s)$ defined in \eqref{eq:zetab} below. The meromorphic continuation of the Ruelle zeta function {$\prod_{\gamma \in \mathcal P} ( 1 - \e^{-s \tau(\gamma)})^{-1}$} for general Anosov flows was established by Giulietti--Liverani--Pollicott \cite{giulietti2013anosov} (see also the work of Dyatlov--Zworski \cite{dyatlov2016zetafunction} for another proof based on microlocal analysis). In this paper the series $\eta_\mathrm N(s), \: \eta_\mathrm D(s)$ are simply called dynamical zeta functions following previous works \cite{petkov1999zeta, petkov2008zeta} and we refer to the book of Baladi \cite{baladi2018dynamical} for more references concerning zeta functions for hyperbolic dynamical systems. 

Our main result is the following
\begin{theo}\label{thm:main}
Let {$d \geqslant 2$ and let the obstacles $D_j, j = 1,...,r,$ satisfy the condition $(\ref{eq:1.1})$}. Then the series $\eta_{\mathrm{N}}(s)$ and $\eta_{\mathrm{D}}(s)$ admit a meromorphic continuation to the whole complex plane with simple poles and integer residues.
\end{theo}

One may also consider the zeta functions $\zeta_\mathrm{b}(s)$ associated to the  boundary conditions $\mathrm{b} = \mathrm D, \mathrm N,$  defined  for $\Re s$ large enough by
\begin{equation}\label{eq:zetab}
\zeta_\mathrm{b}(s) = \exp \left(-\sum_{\gamma \in \mathcal{P}} (-1)^{m(\gamma) \varepsilon(\mathrm b)} \frac{\e^{-s\tau(\gamma)}}{\mu(\gamma) |\det(\mathrm{Id}-P_\gamma)|^{1/2} }\right), 
\end{equation}
where $\varepsilon(\mathrm D) = 1,\: \varepsilon(\mathrm N) = 0$ and $\tau(\gamma) = \mu(\gamma) \tau^{\sharp}(\gamma);$ {$\mu(\gamma) \in \N$ is the repetition number}.
Notice that we have
\begin{equation}\label{eq:linketazeta}
\frac{\zeta_\mathrm{b}'(s)}{\zeta_\mathrm{b}(s)} = \eta_\mathrm b(s), \quad \mathrm{b} = \mathrm{D, N},\quad \Re s \gg 1.
\end{equation}
In particular, since by the above theorem $\eta_\mathrm b(s)$ has simple poles with integer residues, it follows by a classical argument of complex analysis that we have the following
\begin{corr}
{Under the assumptions of Theorem 1} for $\mathrm b = \mathrm{D, N},$ the function $s\mapsto \zeta_\mathrm b(s)$ extends meromorphically to the whole complex plane.
\end{corr}

In fact, we will prove a slightly more general result. For $q \in \mathbb N,\: q \geqslant 2$, consider the Dirichlet series
\[
\eta_q(s)  = \sum_{{\gamma \in {\mathcal P}},\:m(\gamma) \in q \mathbb N} \frac{\tau^\sharp(\gamma) \e^{-s\tau(\gamma)}}{{|\det(\mathrm{Id}-P_\gamma)|}^{1/2}}, \quad \Re(s) \gg 1,
\]
where the sum runs over all periodic rays $\gamma$ with $m(\gamma) \in q \N$. We will show that $\eta_q(s)$ admits a meromorphic continuation to the whole complex plane, with simple poles and residues valued in $\Z / q$ (see Theorem \ref{thm:dirichletq}). In particular, considering the function $\zeta_q(s)$ defined by 
\[
\zeta_q(s) := \exp \left(-\sum_{\gamma \in \mathcal{P},\: m(\gamma) \in q \N}  \frac{\e^{-s\tau(\gamma)}}{\mu(\gamma) |\det(\mathrm{Id}-P_\gamma)|^{1/2} }\right), \quad \Re s \gg 1,
\]
one gets $q \zeta_q' / \zeta_q = q\eta_q$. Thus the function $s \mapsto \zeta_q(s)^q$ extends meromorphically to the whole complex plane since its logarithmic derivative is $q\eta_q$ and by Theorem 4 the function $q\eta_q$ has simple poles with integer residues. One reason for which it is interesting to study these functions is the relation
\begin{equation} \label{eq:1.4} 
\eta_{\mathrm D}(s) =  \frac{\dd}{\dd s} \log \frac{\zeta_2(s)^2}{\zeta_\mathrm N(s)} = 2 \eta_2(s) - \eta_\mathrm N(s),
\end{equation} 
showing that  $\eta_\mathrm D(s)$ for $\Re s \gg 1$ is expressed as the difference of two Dirichlet series with positive coefficients. In particular, to show that $\eta_\mathrm D(s)$ has a meromorphic extension to $\C,$ it is sufficient  to prove that both series $\eta_\mathrm N(s)$ and $\eta_2(s)$ have this property. 

The distribution of the resonances $\mu_j$ in $\C$ depends on the geometry of the obstacles and for trapping obstacles it was conjectured by Lax and Phillips  \cite[page 158] {lax1989scattering} that there exists a sequence of resonances $\mu_j$ with $\Im \mu_j \searrow 0.$  For two disjoint strictly convex obstacles this conjecture is false since there exists a strip $\{z \in \C: \: 0 < \Im z \leqslant a\}$ without resonances (see \cite{Ikawa1983two}).  Ikawa  \cite[page 212]{ikawa1990poles}  conjectured that for {\it trapping obstacles} and $d$ odd there exists $\alpha > 0$ such that  \begin{equation}\label{eq:1.5}
{N_{0, \alpha} = \sharp \{ \mu_j \in \C~:~ 0 < \Im \mu_j \leqslant\alpha\} = \infty}.
\end{equation} 
For $d$ even we must consider
\begin{equation} \label{eq:1.5even}
{N_{0, \alpha} = \sharp \{ \mu_j \in \exp^{-1} (\C \setminus \{0\})~:~ 0 < \Im \mu_j \leqslant\alpha,\: 0 < {\rm arg}\: \mu_j < \pi\}}
\end{equation}
since a meromorphic extension of ${\mathcal R}_\mathrm D(\mu)$ is possible to the covering ${\exp^{-1} (\C \setminus \{0\})}$ of $\mathbb C \setminus \{0\}$ (see \cite{vodev1994}, \cite{vodev1994even} for the counting function of the number of resonances $\mu_j$ when $|\mu_j| \leqslant r$ and $|{\rm arg}\: \mu_j| \to \infty$).
Ikawa called this conjecture modified Lax-Phillips conjecture (MLPC).  In this direction, for $d$ odd, Ikawa \cite{ikawa1988zeta, ikawa1990poles} proved  for strictly convex disjoint obstacles satisfying (\ref{eq:1.1}) that if $\eta_{\mathrm N}(s)$ or $\eta_{\mathrm D}(s)$ cannot be prolonged as {\it entire functions} to $\C$, then there exists {$\alpha > 0$} for which (\ref{eq:1.5}) holds for the Neumann or Dirichlet boundary problem. {Notice that the value $\alpha > 0$ in \cite{ikawa1990poles} is related to the singularity of $\eta_{\mathrm D}(s)$ and to some dynamical characteristics}. The proof in \cite{ikawa1990poles} can be modified to cover also the case $d$ even, applying the trace formula of Zworski \eqref{eq:tracezworski} and the results of Vodev \cite{vodev1994}, \cite{vodev1994even} ({see Appendix B}). It is important to note that the meromorphic continuation of $\eta_{\mathrm D}(s)$ to $\C$ was not established previously and to apply the result of Ikawa we need to show that some (analytic) singularity exists.
The existence of a such singularity is trivial for the Neumann problem since $\eta_{\mathrm N}(s)$  is a Dirichlet series with positive coefficients, and by the theorem of Landau (see for instance, \cite[Théorème 1, Chapitre IV]{bernstein1933}), $\eta_{\mathrm N}(s)$ must have a singularity at $s_0 \in \R$, where $s_0$ is the abscissa of absolute convergence of $\eta_{\mathrm N}(s)$. Moreover, for $d$ odd it was proved (see \cite{petkov2002poles}) that there are constants $c_0 > 0,\: \varepsilon_0> 0$ such that for every $0< \varepsilon \leqslant \varepsilon_0$ we have a lower bound
\[\sharp \left\{ \mu_j \in \C: \: 0 < \Im \mu_j \leqslant \frac{c_0}{\varepsilon},\: |\mu_j| \leqslant r\right\} \geqslant C_{\varepsilon} r^{1 - \varepsilon}, \: r \to \infty.\]
\vspace{0.08cm}
 The situation for the Dirichlet problem is more complicated since $\eta_{\mathrm D}(s)$ is analytic for $\Re s \geqslant s_0$, $s_0$ being the abscissa of absolute convergence \cite{petkov1999zeta}.   Moreover, for $d = 2$  \cite{stoyanov2001spectrum} and for $d\geqslant 3$ under some conditions  \cite{stoyanov2012non} Stoyanov proved that there exists $\varepsilon > 0$ such that $\eta_{\mathrm D}(s)$ is analytic for $\Re s \geqslant s_0 - \varepsilon.$ The reason of this cancellation of singularities is related to the change of signs in the  Dirichlet series defining $\eta_{\mathrm D}(s)$, as it is emphasised by the relation \eqref{eq:1.4}. Despite many works in the physical literature concerning the $n$-disk problem (see for example \cite{cvitanovic1997quantum, wirzba1999quantum,lin2002quantum,potzuweit2012weyl,barkhofen2013experimental} and the references cited there), a rigorous proof of the (MLPC) was established only for sufficiently small balls \cite{ikawa1990zeta} and for obstacles with sufficiently small diameters \cite{stoyanov2009poles}. 

 In this direction we prove the following
\begin{theo}\label{thm:mlpc}
{Assume the boundary $\partial D$ real analytic.   Under the assumptions of Theorem $1$}, the function $\eta_{\mathrm{D}}$ has at least one pole  and the $(\mathrm{MLPC})$ is {satisfied} {for the Dirichlet problem.} Moreover, for every $0 < \delta < 1$ there exists $\alpha_{\delta} > 0$ such that for $\alpha > \alpha_{\delta}$ and $d$ odd we have
\begin{equation} \label{eq:1.10o}
\sharp \{ \mu_j \in \C: \: 0 < \Im \mu_j \leq  {\alpha}, \: |\mu_j | \leqslant r \} \neq  {\mathcal O}(r^{\delta}),
\end{equation} 
while for $d$ even we have
\begin{equation} \label{eq:1.11e}
\sharp \{ \mu_j \in \Lambda_{\omega}: \: 0 < \Im \mu_j  \leq {\alpha}, \; 0 < {\rm arg}\: \mu_j < \pi,\: |\mu_j | \leqslant r\} \neq {\mathcal O} (r^{\delta}).
\end{equation}
\end{theo}
Thus for the resonances of Dirichlet problem we obtain the analog of the result concerning the Neumann problem mentioned above. More precisely, in Appendix \ref{appendix:b} (see Proposition \ref{prop:2.3ika} and Theorem \ref{thm:5}) we show that there exists $a > 0$ depending on the singularity of $\eta_D$ and the dynamical characteristics of $D$ such that for any $0 < \delta < 1$, if we choose
\[\alpha = \frac{a}{1 - \delta},\]
then for $d$ odd and any constant $0 < C < \infty$ the estimate
\[\sharp \{ \mu_j \in \C: \: 0 < \Im \mu_j  \leq {\alpha}, \: |\mu_j | \leqslant r \} \leqslant C r^{\delta},\quad r \geqslant 1,\]
does not hold. For similar results and reference concerning the Pollicott-Ruelle resonances we refer to \cite[Theorem 2]{zworski2017local} and \cite[Theorem 4.1] {jin2023resonances}.

Our paper relies heavily on the works \cite{dyatlov2016pollicott,delarue2022resonances} and we provide specific references in the text. For convenience of the reader we explain briefly the general idea of the proofs of Theorems \ref{thm:main} and \ref{thm:mlpc}. First, in \S\ref{sec:intro} we make some geometric preparations. 
The non-grazing billiard flow {$\varphi_t$} acts on $M = B/\sim,$ where
\[B = S\R^d \setminus (\pi^{-1} (\mathring{D}) \cup {\mathcal D}_g),\]
 $\pi: S\R^d \rightarrow \R^d$ is the natural projection, 
 $\cal D_\mathrm g = \pi^{-1}(\partial D) \cap T(\partial D)$ is the grazing part and 
 $(x, v) \sim (y, w)$ if and only if $(x,v) = (y,w)$ or $x = y \in \partial D$ and $w$ is equal to the reflected direction of $v$ at $x \in \partial D$. 
 By using this equivalence relation, the flow {$\varphi_t$} is continuous in $M$. However, to apply the Dyatlov--Guillarmou theory \cite{dyatlov2016pollicott} in order  to study the spectral properties of {$\varphi_t$} which are  related to the dynamical zeta functions, we need to work with a {\it smooth flow.} 
 For this reason we use a special {\it smooth structure} on $M$  defined by flow-coordinates introduced in the recent work of Delarue-Sch\"utte-Weich \cite{delarue2022resonances} (see \S\ref{subsec:smooth}). In this smooth model, the flow $\varphi_t$ is smooth, and it is uniformly hyperbolic when restricted to the compact trapped set $K$ of $\varphi_t$ (see \S\ref{subsec:anosov}). 
The periodic points are dense in $K$ and for any $z \in K$ the tangent space $T_z M$ has the decomposition $T_zM= \R X(z) \oplus E_u(z) \oplus E_s(z)$ with unstable and stable spaces $E_u(z),\:E_s(z)$, where $X$ is the generator of $\varphi_t$. 
A meromorphic continuation of the cut-off resolvent $\chi (X+ s)^{-1} \chi$ with $\chi \in C_c^{\infty}(M)$ supported near $K$ has been established in \cite{dyatlov2016pollicott} in a general setting. As in \cite{dyatlov2016zetafunction} and \cite{dyatlov2016pollicott}, the estimates on the wavefront set of the resolvent $\chi(X + s)^{-1}\chi$ allow to define its flat trace which is related to the series (\ref{eq:zeta,Poin}). 
This implies a meromorphic continuation of this series to $\C$ (see \cite{delarue2022resonances}). 

To prove a meromorphic continuation of the series $\eta_\mathrm N(s)$ which involves factors $|\det(\mathrm{Id} - P_{\gamma})|^{-1/2}$ instead of $|\det(\mathrm{Id}-P_\gamma)|^{-1}$, a natural approach would consist to study the Lie derivative $\Lie_X$ acting on sections of the unstable bundle $E_u$ (see for example \cite[pp. 6--8]{faure2017semiclassical}). However, in general, $E_u(z)$ is not smooth with respect to $z$, but only H\"older continuous. Thus we are led to change the geometrical setting as in the work of Faure--Tsujii \cite{faure2017semiclassical} (notice that the Grassmannian bundle introduced below also appears in \cite{bowen1975ergodic} and \cite{gouezel2008compact}). Consider the Grassmannian bundle $\pi_G: G  \to V$ over a neighborhood $V$ of $K$; for every $z \in V$ the fiber $\pi_G^{-1}(z)$ is formed by all $(d-1)$-dimensional planes of $T_z V.$ Define the trapped set $\wt K_u = \{(z, E_u(z)):\: z \in K\} \subset G$ and introduce the natural lifted smooth flow $\wt \varphi_t$ on $G$ {(see \S\ref{subsec:lifting})}. Then according to \cite[Lemma A.3]{bowen1975ergodic}, the set $\wt K_u$ is hyperbolic for $\wt \varphi_t$. We introduce the tautological bundle $\mathcal{E} \to G$ by setting
\[\mathcal{E} = \{(\omega, v) \in \pi_G^*(TV)~:~{\omega \in G},\:v \in [\omega]\},\]
where $[\omega]$ denotes the subspace of {$T_{\pi_G(\omega)}V$} that $\omega \in G$ represents, and $\pi_G^*(TV)$ is the pull-back of the tangent bundle $TV \to V$ by $\pi_G$. Next, we define the vector bundle $\mathcal{F} \to G$ by
\[
\cal{F} = \{(\omega,W) \in TG~:~\dd \pi_G(w) \cdot W = 0 \}
\]
which is the {\og vertical subbundle\fg} of the bundle $TG \to G$. Finally, set
\[
\cal{E}_{k, \ell} = \wedge^k \cal{E}^* \otimes \wedge^\ell \cal{F}, \quad 0 \leqslant k \leqslant d - 1, \quad 0 \leqslant \ell \leqslant d^2- d,
\]
{where $\cal E^*$ is the dual bundle of $\cal E$}. We define a suitable flow $\Phi_t^{k, \ell}: \cal{E}_{k, \ell} \rightarrow \cal{E}_{k, \ell}$ as well as a transfer operator (see \S \ref{subsec:vectorbundle} {for the notations})
\[
{\Phi^{k, \ell, *}_{-t} u(\omega) = \Phi_t^{k, \ell} [{\bf u}(\tilde{\varphi}_{-t}(\omega)], \quad {\bf u} \in C^\infty(G, \cal E_{k, \ell}).}
\]
For a periodic orbit $\gamma(t)$ of $\varphi_t$, this geometrical setting allows to express the term $|\det(\mathrm{Id} - P_{\gamma})|^{-1/2} $ as a finite sum involving the traces $\tr (\alpha^{k, \ell}_{\wt \gamma })$ related to the periodic orbit $\wt \gamma = \{(\gamma(t), E_u(\gamma(t)):\:t \in [0, \tau(\gamma)]\}$ of the flow $\wt \varphi_t$ (see \S\ref{subsec:flattraceresolv} for the notation $\alpha^{k, \ell}_{\wt \gamma}$ and Lemma \ref{lem:alternatedsum}). {This crucial argument explains the introduction of the bundles ${\cal E}_{k, \ell}$ and the related geometrical technical complications.} In this context we may apply the Dyatlov--Guillarmou theory {(see Theorem 1 in \cite{dyatlov2016pollicott})} for the generators 
\[\Pbf_{k,\ell} \ubf  = \left.\frac{\dd}{\dd t} \Bigl(\Phi^{k, \ell,*}_{-t}\ubf\Bigr)\right|_{t=0}, \quad \ubf \in C^\infty(G, \cal E_{k, \ell})\]
of the transfer operators $\Phi^{k, \ell, *}_{-t}$
{(in fact, by using a smooth connexion, we introduce a new operator $\Qbf_{k, \ell}$ which coincides with $\Pbf_{k, \ell}$ near $\tilde{K}_u$ (see \S\ref{subsec:dyatlovguillarmou})).} {This leads to a meromophic continuation of the the cut-off resolvent $\wt\chi  (\Qbf_{k, \ell} + s)^{-1} \wt\chi,$ where $\wt \chi \in C_c^{\infty}(\wt V_u)$ is equal to 1 on $\wt K_u$ (see \ref{subsec:dyatlovguillarmou} for the notations).} {By applying the Guillemin flat trace formula \cite{guillemin1977lectures} (see \cite[Appendix B]{dyatlov2016zetafunction} and Section 3 in \cite{Weich2023})}, concerning 
{\[{\rm tr}^\flat \Bigl(\int_0^{\infty} \varrho(t) \wt \chi( \e^{-t \Qbf_{k, \ell}} \ubf ) \wt \chi \dd t \Bigr),\quad \varrho \in C_c^{\infty} (0, \infty),\]}we obtain the meromorphic continuation of $\eta_\mathrm N$. Finally, the meromorphic continuation of $\eta_q$ is obtained in a similar way, by considering in addition a certain \textit{$q$-reflection bundle} $\cal R_q \to G$ to which the flow $\wt \varphi_t$ can be lifted (see \S\ref{subsec:qreflexion}).

The strategy to prove Theorem \ref{thm:mlpc} is the following. First, the representation (\ref{eq:1.4}) tells us that, if $\eta_\mathrm D(s)$ can be extended to an entire function, then the function $\zeta_2^2/\zeta_\mathrm N$ has neither zeros nor poles on the whole complex plane. For obstacles with real analytic boundary we may use real analytic charts near $\partial D$ to define a real analytic structure on $M$ which makes $\varphi_t$ a real analytic flow. In this setting we may apply a result of Fried \cite{fried1995meromorphic} to the non-grazing  flow $\varphi_t$  lifted to the Grassmannian bundle, and show that the entire functions $\zeta_2$ and $\zeta_\mathrm N$ have {\it finite order.} This crucial point implies that the meromorphic function  $\zeta_2^2/\zeta_\mathrm N$ has also finite order. Finally, by using Hadamard's factorisation theorem, one concludes that we may write $\zeta_2(s)^2/\zeta_\mathrm N(s) = \e^{Q(s)}$ for some polynomial $Q(s)$.  This leads to $\eta_\mathrm D(s) = - Q'(s)$. Since $\eta_\mathrm D(s) \to 0$ as $\Re s \to + \infty$, we obtain a contradiction and $\eta_D(s)$ is not entire. The existence of a singularity of $\eta_D(s)$ implies the lower bound  (\ref{eq:B4})  (see Appendix B) and we obtain (\ref{eq:1.10o}) and (\ref{eq:1.11e}).  Notice that this argument works as soon as the entire functions $\zeta_{2}$ and $\zeta_\mathrm N$ have {\it finite order}. The recent work of Bonthonneau--J\'ez\'equel \cite{bonthonneau2020fbi} about Anosov flows suggests that this should be satisfied for obstacles with \textit{Gevrey} regular boundary $\partial D$.  In particular, the (MLPC) should be true for such obstacles. However in this paper  we are not going to study this generalization.

The paper is organised as follows. In \S\ref{sec:intro} one introduces the geometric setting of the billiard flow $\varphi_t$ and its smooth model. We define the Grassmannian extension $G$ and the bundles $\cal E, \cal F, \: \cal E_{k, l} = \Lambda^k \cal E^{\star} \otimes \Lambda^{\ell} \cal F $ over $G$. Next, we discuss the setting for which we apply the Dyatlov-Guillarmou theory \cite{dyatlov2016pollicott} for {some first order operator ${\bf Q}_{k, \ell}$} leading to a meromorphic continuation of the cut-off resolvent
${\bf R}_{k, \ell} (s) = \tilde{\chi} ({\bf Q}_{k, \ell} + s)^{-1}\tilde{\chi}.$
In \S\ref{sec:neumann} we treat the flat trace of the resolvent ${\bf R}_{\varepsilon}^{k, \ell}(s) = \e^{-\varepsilon ({\bf Q}_{k, \ell} + s) } {\bf R}_{k, \ell}(s), \: \varepsilon > 0,$ and we obtain a meromorphic continuation of $\eta_N$. In \S\ref{sec:dir}  we study the dynamical zeta functions $\eta_q(s)$ for particular rays $\gamma$ having number of reflections $m(\gamma) \in q \N,\: q \geqslant 2$. Applying the result for $\eta_2(s)$, we deduce the meromorphic continuation of $\eta_D.$ Finally, in \S\ref{sec:real}  we treat the modified Lax-Phillips conjecture for obstacles with real analytic boundary and we prove that the function $\eta_{\mathrm D}$ is not entire. In Appendix {A} we present a proof for $d \geqslant 2$ of the uniform hyperbolicity of the flow {$\phi_t$} in the Euclidean metric in $\R^d,$ {while in Appendix B we discuss the modifications of the proof of Theorem 2.1 in \cite{ikawa1990poles} for even dimensions} and we finish the proof of  Theorem 3.

\subsection*{Acknowledgements.} We would like to thank Fr\'ed\'eric Faure, Benjamin Delarue, St\'ephane Nonnenmacher, Tobias Weich and Luchezar Stoyanov for very interesting discussions. Thanks are  also due to Colin Guillarmou for pointing out to us that we could use Fried's result for the proof of Theorem \ref{thm:mlpc}. Finally, we would like to thank the anonymous referee for his comments and suggestions that helped to improve the manuscript. The first author is supported from the European Research Council (ERC) under the European Unions Horizon 2020 research and innovation programme with agreement No. 725967.

\section{Geometrical setting}\label{sec:intro}
\subsection{The billiard flow}\label{subsec:defflow}

Let $D_1, \dots, D_r \subset \R^d$ be pairwise disjoint compact convex obstacles, satisfying the condition (\ref{eq:1.1}), where {$r \geqslant 3$.} We denote by $S\R^d$ the unit tangent bundle of $\R^d$ and by $\pi : S\R^d \to \R^d$ the natural projection. For $x \in \partial D_j$, we denote by $n_j(x)$ the {\it inward unit normal vector} to $\partial D_j$ at the point $x$ pointing into $ D_j.$ Set $D = \bigcup_{j=1}^rD_j$ and  
\[\cal D= \{(x,v) \in S \R^d~:~x \in \partial D\}.\]
We will say that $(x,v) \in T_{\partial D_j}\R^d$ is incoming (resp. outgoing) if $\langle v, n_j(x)\rangle > 0$ (resp. $\langle v, n_j(x) \rangle < 0$), and introduce
\[
\begin{aligned}
\cal D_\mathrm{in} &= \{ (x, v) \in \cal D~:~(x,v) \text{ is } {\rm incoming}\}, \\
\cal D_\mathrm{out}&= \{ (x, v) \in \cal D~:~ (x,v) \text{ is } {\rm outgoing}\}.
\end{aligned}
\]
We define the grazing set $\cal D_\mathrm{g}= T(\partial D) \cap \cal D$ and one gets
\[\cal D = \cal D_\mathrm{g} \sqcup \cal D_\mathrm{in} \sqcup \cal D_\mathrm{out}.\]
 The billiard flow $(\phi_t)_{t \in \R} $ is the complete flow acting on $S\R^d \setminus \pi^{-1}(\mathring{D})$ which is defined as follows. For $(x,v) \in S\R^d\setminus \pi^{-1} (\mathring{D})$ we set
\[
\tau_\pm(x,v) = \pm \inf\{t \geqslant 0: x \pm tv \in \partial D\}
\]
and for $(x, v) \in \cal D_\mathrm{in/out/g}$ we denote by $v' \in \cal D_\mathrm{out/in/g}$  the image of $v$ by the reflexion with respect to $T_x(\partial D)$ at $x \in \partial D$,  that is
\[
v' = v - 2\langle v, n_j(x) \rangle n_j(x), \quad v \in S_x\R^d, \quad x \in \partial D_j
\]
{(see Figure \ref{fig:notations}).}
\begin{figure}[h]
\includegraphics[scale=1]{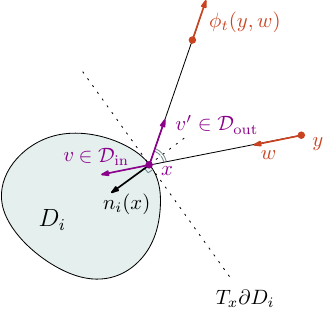}
\caption{The billiard flow $\phi_t$}
\label{fig:notations}
\end{figure}
By convention, we have $\tau_\pm( x, v) = \pm \infty,$ if the ray $ x \pm t v$ has no common point with $\partial D$ for $\pm t > 0.$
Then  for $(x,v) \in (S\R^d \setminus \pi^{-1}(D)) \cup \cal D_{\mathrm g}$ we define
\[
\phi_t(x,v) = (x + tv, v), \quad t \in [\tau_-(x,v), \tau_+(x,v)],
\]
while for $(x,v) \in \cal D_{\mathrm{in/out}}$, we set
\[
\phi_t(x,v) = (x+tv, v) \quad \text{ if } \quad \left \lbrace \begin{matrix} (x, v) \in \cal D_\mathrm{out},~t \in \left[0, \tau_+(x,v)\right], \vspace{0.2cm} \\ \text{or } (x, v) \in \cal D_{\mathrm{in}},~t \in \left[\tau_{-}(x, v) , 0\right], \end{matrix} \right.
\]
and 
\[
\phi_t(x,v) = (x+tv', v') \quad \text{ if } \quad \left \lbrace \begin{matrix} (x, v') \in \cal D_\mathrm{out},~t \in \left]0, \tau_+(x,v)\right], \vspace{0.2cm} \\ \text{or } (x, v) \in \cal D_{\mathrm{in}},~t \in \left[\tau_{-}(x, v') , 0\right[. \end{matrix} \right.
\]
Next we extend $\phi_t$ to a complete flow (which we still denote by $\phi_t$) characterized by the property
\[
\phi_{t+s}(x,v) = \phi_t ( \phi_s(x,v)), \quad t,s \in \R, \quad (x,v) \in S\R^d \setminus \pi^{-1}(D).
\]
Strictly speaking, $\phi_t$ is not a flow, since the above flow property does not hold in full generality for $(x,v) \in \cal D_{\mathrm{in/out}}$. However, we can deal with this problem by considering an appropriate quotient space (see \S\ref{subsec:smooth} below).

\subsection{A smooth model for the non-grazing billiard flow}\label{subsec:smooth}

In this subsection, we briefly recall the construction of \cite[Section 3]{delarue2022resonances} which allows to obtain a smooth model for the non-grazing billiard flow. First, we define the non-grazing billiard table $M$ as 
\[
M = B / \sim, \quad B = S\R^d \setminus \left(\pi^{-1}(\mathring{D}) \cup \cal D_\mathrm{g}\right),
\]
where $(x,v) \sim (y,w)$ if and only if $(x,v) = (y,w)$ or
\[
x = y \in \partial D \quad \text{ and } \quad w = v'.
\]
The set $M$ is endowed with the quotient topology. We will change the notation and pass from $\phi_t$ to the non-grazing flow {$\varphi_t$}, which is defined on $M$ as follows. For $(x,v) \in (S\R^d \setminus \pi^{-1}(D)) \cup \cal D_\mathrm{in}$ we define
\[
{\varphi_t([(x,v)])} = [\phi_t(x,v)], \quad t \in \left]\tau^{\mathrm{g}} _-(x,v), \tau^{\mathrm{g}} _+(x,v)\right[,
\]
where $[z]$ denotes the equivalence class of the vector $z \in B$ for the relation $\sim$, and 
\[
\tau^{\mathrm{g}} _\pm(x,v) = \pm \inf\{t > 0:\phi_{\pm t}(x,v) \in \cal D_\mathrm{g}\}.
\]
{Clearly,  we may have  $\tau_{\pm}^{\mathrm{g}}(x, v)  = \pm\infty.$} On other hand, it is important to note that $\tau^{\mathrm{g}}_{\pm}(x, v) \neq 0$ for $(x, v) \in {\mathcal D}_{\mathrm{in}}.$
Note that this formula indeed defines a flow on $M$ since each $(x,v) \in B$ has a unique representative in $(S\R^d \setminus \pi^{-1}(\mathring{D})) \cup \cal D_\mathrm{in}$.  {Thus $\varphi_t$}  is continuous but the flow trajectory of the point $(x, v)$  for times $t \notin \left]\tau^{\mathrm{g}} _-(x,v), \tau^{\mathrm{g}} _+(x,v)\right[$  is not defined. 

Following \cite[Section 3]{delarue2022resonances}, we define smooth charts on $M = B/\sim$ as follows. 
Introduce the surjective map $\pi_M : B \to M$ by
$\pi_M(x, v) = [(x, v)]$ and note that
\begin{equation}\label{eq:flows}
\varphi_t \circ \pi_M = \pi_M \circ \phi_t.
\end{equation}
Set $\mathring{B}:  = S \R^d \setminus \pi^{-1} (D)$. Then $\pi_M : \mathring{B} \rightarrow M$ is a homeomorphism onto its image ${\mathcal O}.$  Let $\cal G = \pi_M(\cal D_\mathrm{in})$ be the gluing region.  We consider the map $\pi_M^{-1} : {\mathcal O} \to \mathring{B}$. Then we may pull back the smooth structure of $\mathring{B}$ to ${\mathcal O}$ and define the charts on ${\mathcal O}$ by using those of ${\mathring B}.$  Next we wish to define charts in an open neighborhood of $\mathcal G$. For every point $z_\star = (x_\star, v_\star) \in \cal D_{\mathrm{in}}$ let  
\[F_{z_ \star}: U_{z_\star} \times U_{z_\star} \to \cal D_{\mathrm{in}}\] 
be a local smooth parametrization of $\cal D_{\mathrm{in}}$, where $U_{z_\star}$ is an open small neighborhood of $0$ in $\R^{d -1}.$ For small $\varepsilon_{z_\star} > 0$, we may define the map
\[
{\Psi_{z_\star}: \left]-\varepsilon_{z_\star}, \varepsilon_{z_\star} \right[ \times U_{z_\star} \times U_{z_\star} \to M}
\]
by
\begin{equation}\label{eq:defpsiz}
{\Psi_{z_\star}(t, y, w) = (\pi_M \circ \phi_t \circ F_{z_\star})(y, w).}
\end{equation}
 Up to shrinking $U_{z_\star}$ and taking $\varepsilon_{z_\star}$ smaller, {$\Psi_{z_\star}$} is a homeomorphism onto its image $\mathcal O_{z_\star} \subset M$,  (see Corollary 4.3 in \cite{delarue2022resonances}). 
Indeed, repeating the argument of \cite{delarue2022resonances}, to see that $\Psi_{z_\star}$ is injective, let $F_{z_\star}(y_k, w_k) = (x_k, v_k) \in \cal D_{\mathrm{in}},\:  k = 1, 2$, and assume that $\pi_M\phi_{t_1} (x_1, v_1) = \pi_M\phi_{t_2} (x_2, v_2)$. Since the vectors in $\cal D_\mathrm{in}$ are transversal to $\partial D$, we see that for each $z \in \cal O_{z_\star},$ there is a unique $t \in \left]-\varepsilon_{z_\star}, \varepsilon_{z_\star}\right[$ such that {$\varphi_t(z) \in \cal G$}. In particular, we have $t_1 = 0$ if and only if $t_2 = 0$. In this case, $(x_1, v_1) = (x_2, v_2)$ since $\pi_M : \cal D_\mathrm{in} \to \cal G$ is injective. If $t_1 \neq 0, t_2 \neq 0,$ then $t_1$ and $t_2$ have the same sign and by the injectivity of $\pi_M : \mathring{B} \to M$ and the definition of $\phi_t$, we have
\[
\left\lbrace \begin{array}{lll}
(x_1 + t_1 v_1, v_1) = (x_2 + t_2 v_2, v_2) &\text{ if } &t_1, t_2 > 0, \\
(x_1 + t_1 v_1', v_1') = (x_2 + t_2 v_2', v_2') &\text{ if } &t_1, t_2 < 0,
\end{array}
\right.
\]
where $v_k'$ is the reflection of $v_k$ with respect to $T_{x_k}\partial D$ for $k = 1,2.$
Thus one concludes that $(t_1, x_1, v_1)  = (t_2, x_2, v_2).$ As mentioned above, the directions in $\cal D_{\mathrm{in}}$ are transversal to the boundary $\partial D$. This implies that the maps {$\Psi_{z_\star}$} are open. In particular, {$\Psi_{z_\star}$} realises a homeomorphism onto its image ${\mathcal O}_{z_\star}$ and we declare the map {$\Psi_{z_\star}^{-1} : \cal O_{z_\star} \to \left]-\varepsilon_{z_\star}, \varepsilon_{z_\star} \right[ \times U_{z_\star} \times U_{z_\star}$} as a chart.  Hence we obtain an open covering
\[
\mathcal G  \subset \bigcup_{z_\star \in \cal D_{\mathrm{in}}} \mathcal O_{z_\star}.
\]
Note that if $ \mathcal O \cap {\mathcal O}_{z_\star}\neq \emptyset$ for any $z_\star$, clearly the map
{\[(t,x,v) \mapsto (\pi_M^{-1} \circ \Psi_{z_\star})(t,x,v) = (\phi_t \circ F_{z_\star})(x, v) \]}is smooth on {$\Psi_{z_\star}^{-1}(\cal O \cap \cal O_{z_\star})$}. On the other hand, assume that $\mathcal O_{z_\star} \cap \mathcal O_{z_\star'} \neq \emptyset$ for some $z_\star, z_\star' \in \cal D_\mathrm{in}$. If $\pi_M(\phi_{t}(F_{z_\star}(x,v))) = \pi_M(\phi_{s}(F_{z_\star'}(y, w))) \in {\mathcal O}_{z_\star}\cap {\mathcal O}_{z_\star'}$, then as above this yields $t = s,\: F_{z_\star}(x, v) = F_{z_\star'}(y, w),$  and we conclude that 
\begin{equation}\label{eq:changecoord}
\begin{aligned}
(\Psi_{z_\star}^{-1} \circ \Psi_{z_\star'})(t, y, w)  &= \left(\Psi_{z_\star}^{-1} \circ \pi_M \circ \phi_t \circ F_{z_\star'}\right)(y, w) \\
&= \left(\Psi_{z_\star}^{-1} \circ \pi_M \circ \phi_t \circ F_{z_\star}\right)\Bigl((F_{z_\star}^{-1}  \circ F_{z_\star'}) (y, w)\Bigr) \\
& = \left(t, (F_{z_\star}^{-1} \circ F_{z_\star'} )(y, w)\right).
\end{aligned}
\end{equation}
This shows that the change of coordinates {$\Psi_{z_\star}^{-1} \circ \Psi_{z_\star'}$} is smooth on the set {$\Psi_{z_\star'}^{-1}(\cal O_{z_\star} \cap \cal O_{z_\star'})$}, and these charts endow $M$ with a smooth structure. It is easy to see that with this differential structure the flow $(\varphi_t)$ is smooth on $M$. Indeed, this is obvious far from the gluing region $\mathcal G$. Now let $z \in \cal G$ and $z_\star \in \cal D_\mathrm{in}$ be such that $\pi_M(z_\star) = z$. Then for $s,t \in \R$, with $|t| + |s|$ small, and $(y,w) \in U_{z_\star} \times U_{z_\star}$, we have
\[
\begin{aligned}
\left(\Psi_{z_\star}^{-1} \circ \varphi_s \circ \Psi_{z_\star}\right)(t, y, w) &= \left(\Psi_{z_\star}^{-1} \circ \varphi_s \circ \pi_M \circ \phi_t \circ F_{z_\star}\right)(y,w) \\
&= \left(\Psi_{z_\star}^{-1}  \circ \pi_M \circ \phi_{t + s} \circ F_{z_\star}\right)(y, w) \\
&= ( s + t, y, w).
\end{aligned}
\]
Consequently, the flow $(\varphi_t)$ is also smooth near $\cal G$ and we obtain a smooth non-complete flow on $M$. 

\subsection{Oriented periodic rays}\label{subsec:orientation}
A periodic point of the billiard flow is a pair {$(x,v)$ lying in $S\R^d$}, together with a number $\tau > 0$, such that $\phi_\tau(x,v) = (x,v)$. The point $(x, v)$ will be called $\tau$- periodic. A \textit{periodic trajectory} of $\phi_t$, or equivalently an \textit{oriented periodic ray}, is by definition an equivalence class of periodic points, where we identify two periodic {points} $(x,v)$ and $(y,w)$, if they are $\tau$-periodic with the same $\tau$ and if there are $\tau_1, \tau_2 \in \R$ such that $\phi_{\tau_1}(x,v) = \phi_{\tau_2}(y,w)$. Of course, the map $\pi_M$ induces a bijection between oriented periodic rays and periodic orbits of the non-grazing flow $\varphi_t$.  For each periodic orbit $\gamma$, we will denote by $\tau(\gamma)$ its period. Also, we will often identify a periodic orbit with a parametrization {$\gamma : [0, \tau(\gamma)] \to S\R^d$.} 

Note that every oriented periodic ray is determined by a sequence
\[
\alpha_\gamma = (i_1, \dots, i_k),
\]
where $i_j \in \{1,\dots,r\}$, with $i_k \neq i_1$ and $i_j \neq i_{j+1}$ for $j = 1,...,k-1$,
such that $\gamma$ has {\it successive reflections}  on $\partial D_{i_1}, \dots, \partial D_{i_k}$. The sequence $\alpha_{\gamma}$ is well defined modulo cyclic permutation, and we say that the ray $\gamma$ {has type $\alpha_{\gamma}.$}  The non-eclipse condition (\ref{eq:1.1}) implies that the reciprocal is true. {More precisely},  for any sequence $\alpha = {(i_1, \dots, i_k)}$ with  $i_j \neq i_{j+1}$ for $j =1, \dots, k-1$ and $i_k \neq i_1$, there exists a unique periodic ray $\gamma$ such that $\alpha_\gamma = \alpha$ {(see \cite[Proposition 2.2.2 and Corollary 2.2.4]{petkov2017geometry}).}

We conclude this paragraph by some remark on the oriented rays. For every oriented periodic ray $\gamma$ generated by a periodic point $(x,v) \in {\mathring B}$ {and period $\tau$}, one may consider the reversed ray $\bar \gamma$, generated by $(x, -v) \in {\mathring B}$ and {$\tau.$} There are two possibilities. For most rays, $\gamma$ and $\bar \gamma$ {give} rise to {different oriented periodic rays}, even {if their projections in $\R^d$ are the same}. However it might happen that $\bar \gamma$ coincides with $\gamma$. This is the case, for example, if the ray {$\gamma$ has type $\alpha = (1,2)$} (modulo permutation).

\subsection{Uniform hyperbolicity of the flow $\varphi_t$}\label{subsec:anosov}
From now on, we will work exclusively with the flow $\varphi_t$ defined on $M = B/\sim$ by the smooth model described in $\S\ref{subsec:smooth}.$ Let $X$ be the generator of $\varphi_t$. The trapped set $K$ of $\varphi_t$ is defined as the set of points $z \in M$ which satisfy $-\tau_-^{\mathrm{g}}(z) = \tau_+^{\mathrm{g}} (z) = +\infty$ and
\[
\sup A(z) = - \inf A(z) = +\infty, \quad \text{where} \quad A(z) = \{t \in \R~:~\pi(\varphi_t(z)) \in \partial D\}.
\]
By definition, $\varphi_t(z)$ is defined for all $t \in \R$ whenever $z \in K$. The flow $\varphi_t$ is called uniformly hyperbolic on $K$, if for each $z \in K$ there exists a decomposition
\begin{equation} \label{eq:decomp}
T_zM= \R X(z) \oplus E_u(z) \oplus E_s(z),
\end{equation}
which is $\dd \varphi_t$-invariant (in the sense that $\dd \varphi_t(E_\bullet(z)) = E_\bullet(\varphi_t(z))$ for $\bullet = u, s$), with $\dim E_s(z) = \dim E_u(z)  = d - 1$, {such that for some constants $C > 0,\: \nu > 0$, independent of $z \in K$}, and some smooth norm $\|\cdot\|$ on $TM$, we have
\begin{equation}\label{eq:2.6}
\left\|\dd \varphi_t(z) \cdot v\right\| \leqslant \left \lbrace \begin{matrix} C \e^{-\nu t} \|v\|,~ &v \in E_s(z),~&t\geqslant 0, \vspace{0.2cm}  \\  C \e^{-\nu |t|} \|v\|,~ &v \in E_u(z),~&t\leqslant 0. \end{matrix} \right.
\end{equation}
The spaces $E_s(z)$ and $E_u(z)$ depend continuously on $z$ (see \cite[Section 2]{hasselblat2002}).

We may define the trapped set $K_e$ for the flow $\phi_t$ in the Euclidean metric and note that  $K = \pi_M(K_e).$ (Here we use the notation $\phi_t$ for the flow in the Euclidean metric to distinguish it with the flow $\varphi_t$ definite on the smooth model). 
The uniform hyperbolicity on $K_e$ of the flow $\phi_t$  in the Euclidean metric for $z \in \mathring B \cap K_e$ can be defined by the splitting of the tangent space $T_z(\mathring B \cap K_e)$ 
(see Definition 2.11 in \cite{delarue2022resonances} and {Appendix A}). Following this definition, one avoids the  points $(x, v) \in  K_e \cap \mathcal D_{in}.$ To treat these points, 
{denote $\overline{\cal{D}}_\mathrm{in} = \{(x, v): \: x \in \partial D, |v| = 1, \langle v, n(x) \rangle \geqslant 0\}$ and define the {\it billiard ball map}  
\[
{\bf B} : \overline{\cal{D}}_\mathrm{in} \ni (x, v) \longmapsto (y, R_{y}w)   \in \overline{\cal{D}}_\mathrm{in},
\]
where $R_y : S_y\R^d \to S_y\R^d$ is the reflection with respect to $T_y(\partial D)$ and
\[
(y, w) = \phi_{\tau_{+}(x, v)}(x, v),\: \tau_{+}(x, v) = \inf\{t > 0~:~\pi(\phi_t (x, v)) \in \partial D\}.\]
To see that ${\bf B}(x,v)$ is well defined  we need $\tau_{+}(x, v) < \infty$ and this condition is satisfied for $(x, v) \in K_e \cap \overline{\mathcal D}_{\mathrm{in}}.$
The map {\bf B} is called  {\it collision map} in \cite{chernov2006chaotic}, 
and it is smooth (see for instance, \cite{kovachev1988billard}).} 
{For $(x, v) \in  K_e \cap \cal D_{\mathrm{in}}$ we can define $\dd {\bf B}(x, v)$ 
 and this is useful for the estimates of $\|\dd \phi_t(x, v)\|$ for $(x, v) \in \mathring B \cap K_e$} (see \cite[\S4.4]{chernov2006chaotic} and Appendix A).

The uniform hyperbolicity of $\phi_t$ in the Euclidean metric on $\mathring B \cap K_e$ implies  the uniform hyperbolicity of $\varphi_t$ in the smooth model (see \cite[Proposition 3.7]{delarue2022resonances}). Thus, to obtain  \eqref{eq:2.6}, we may apply the uniform hyperbolicity of $\phi_t$ in the Euclidean metric on $\mathring{B} \cap K_e$ established for $d = 2$ in \cite{morita1991symbolic} and \cite[\S4.4]{chernov2006chaotic}. For $d \geqslant 3$, the same could perhaps be obtained by applying the results in  \cite[\S4]{chernov2003multi}. {The hyperbolicity at the points $z= (x, v) \in K_e$ which are not periodic must be justified and the stable/unstable spaces $E_s(z)/ E_u(z)$ must be well determined; for $ d \geqslant 3$ this seems to be not sufficiently detailed in the literature. Since the hyperbolicity of $\varphi_t$ is crucial for our exposition, and for the sake of completeness, we present in Appendix A a proof of the uniform hyperbolicity as well as a construction of $E_s(z)$ and  $E_u(z)$ for all $z \in \mathring{B} \cap K_e.$}

\subsection{The Grassmannian extension}\label{subsec:lifting}
In what follows, we assume that the flow $\varphi_t$ is hyperbolic on $K$ and we will take a small neighborhood $V$ of $K$ in $M$, with smooth boundary. We embed $V$ into a compact manifold without boundary $N$. {For example, we may take the double manifold $N$ of the closure of $V$. This means that $N = \bar{V} \times \{0, 1\}/\sim$ and $(x, 0) \sim (x, 1)$ for all $x \in \partial V$.} We arbitrarily extend $X$ to obtain a smooth vector field on $N$, which we still denote by $X$. The associated flow is still denoted by $\varphi_t$ (however note that this new flow $\varphi_t$ is now complete). 

{For our exposition it is important to introduce} the $(d-1)$-Grassmannian bundle 
\[\pi_G : G \to N\] over $N$. {More precisely}, for every $z \in N$, the set $\pi_G^{-1}(z)$ consists of all $(d-1)$-dimensional planes of $T_zN$.   Moreover, $\pi_G^{-1}(z)$ can be identified with the Grasmannian $G_{d-1}(\R^{2d- 1})$ which is isomorphic to $O(2d - 1) / (O(d-1) \times O(d)),$ $O(k)$ being the space of $(k \times k)$ orthogonal matrices with entries in $\R$. 
The dimension of $O(k)$ is $k(k-1)/2$, hence the dimension of $\pi_G^{-1}(z) $ is $d(d-1).$
 Note that $G$ is a smooth compact manifold. We may lift the flow $\varphi_t$ to a flow $\wt \varphi_t : G \to G$ which is simply defined by
\begin{equation}\label{eq:defflow}
\wt \varphi_t(z,E) = (\varphi_t(z), \dd \varphi_t(z)(E)), \quad  z \in N,\quad E \subset T_zN, \quad \dd \varphi_t(z)(E) \subset T_{\varphi_t(z)}N.
\end{equation}
Introduce the set 
\[
\widetilde K_u = \{(z,E_u(z))~:~z \in K\} \subset G.
\]
Clearly, $\wt K_u$ is invariant under the action of $\wt\varphi_t$, since $\dd\varphi_t(z) (E_u(z)) = E_u(\varphi_t(z))$. The set $\wt K_u$ will be seen as the trapped set of the restriction of $\widetilde \varphi_t$ to a neighborhood of $\wt K_u$.
As $K$ is a hyperbolic set, it follows from \cite[Lemma A.3]{bowen1975ergodic} that the set
$\widetilde K_u $
is  hyperbolic for $\wt \varphi_t$ and we have a decomposition
\[
T_{\omega} G = \R \widetilde X(\omega) \oplus \widetilde E_u(\omega) \oplus \widetilde E_s(\omega), \quad \omega \in \widetilde K_u.
\]
Here $\wt X$ is the generator of the flow $(\wt \varphi_t)$ and the spaces $\widetilde E_s(\omega)$ and $\widetilde E_u(\omega)$ are defined as follows. For  small $\varepsilon > 0$, let
\[W_s(z, \varepsilon) = \{ z' \in M: {\rm dist}\: (\varphi_{t}(z), \varphi_{t}(z')) \leqslant \varepsilon \text{ for every } t \geqslant 0\}\]
and
 \[W_u(z, \varepsilon) = \{ z' \in M: {\rm dist}\: (\varphi_{-t}(z), \varphi_{-t}(z')) \leqslant \varepsilon \text{ for every } t\geqslant 0\}\]
be the local stable and unstable manifolds at $z$ {of size $\varepsilon$}, where ${\rm dist}$  is any smooth distance on $M$. 
It is known that the local stable and unstable manifolds are smooth (see for instance,  \cite[Section 2]{hasselblat2002}). Moreover for $b = s,u$ we have
\[ T_z(W_b(z, \varepsilon)) = E_b(z)\]
and
\[\varphi_t(W_s(z, \varepsilon)) \subset W_s(\varphi_t(z), \varepsilon),\quad \varphi_{-t}(W_u(z, \varepsilon)) \subset W_u(\varphi_{-t}(z), \varepsilon), \quad t\geqslant t_0 > 0.\]
 For $b = s, u,$ we define
\[\widetilde W_b(z) =  TW_b (z, \varepsilon) = \{(z', E_b(z'))~:~z' \in W_b(z, \varepsilon)\} \subset G.\]
 Finally, for $\omega = (z, E_u(z)) \in \widetilde K_u$, set
\[
 \widetilde E_u(\omega) = T_{\omega}(\widetilde W_u(z)),
\]
and define $\wt E_s(\omega)$ as the tangent space at $\omega$ of the manifold 
\[
\wt W_{s, \mathrm{tot}}(z) = \displaystyle{\left\{E  \in \pi_G^{-1}(W_{s} (z,\varepsilon)) :\mathrm{dist} (E_u(z), E) < \varepsilon \right\}},
\]
where $\mathrm{dist}$ is any smooth distance {on the fibres of $TN$.} 

\begin{lemm}\label{lem:ident}
For any $\omega = (z, E) \in G$ we have isomorphisms
\[
\wt E_u(\omega) \simeq E_u(z), \quad \wt E_s(\omega) \simeq E_s(z) \oplus \ker \dd \pi_G(\omega).
\]
Under these identifications, we have
\[
\dd \wt \varphi_t|_{\wt E_u(\omega)} \simeq \dd \varphi_t|_{E_u(z)}, \quad \dd \wt \varphi_t|_{\wt E_s(\omega)} \simeq \dd \varphi_t|_{E_s(z)} \oplus \dd \wt \varphi_t|_{\ker \dd \pi_G(\omega)}.
\]
\end{lemm}
\begin{proof}
Note that if $\omega = (z, E) \in G$, by (\ref{eq:defflow}) one has
\begin{equation}\label{eq:commute}
\dd \pi_G(\omega) \circ \dd \wt\varphi_t (\omega) = \dd (\pi_G \circ \wt \varphi_t)(\omega) = \dd (\varphi_t \circ \pi_G)(\omega) = \dd \varphi_t(z) \circ \dd \pi_G (\omega).
\end{equation}
This  equality shows that $\dd \wt\varphi_t$ preserves $\ker \dd \pi_G.$ Looking at the definitions of $\wt W_u(z)$ and $W_u(z,\varepsilon)$, we see that
\[
\dd \pi_G(\omega)|_{\wt E_u(z)} : \wt E_u(z) \to E_u(z)
\]
realises an isomorphism. Then by \eqref{eq:commute}, it is clear that $\dd \pi_G(\omega)|_{T_\omega \wt W_u(z)}$ realises a conjugation between $\dd \wt \varphi_t(\omega)|_{\wt E_u(\omega)}$ and $\dd \varphi_t(z)|_{E_u(z)}$. Similarly, $\dd \pi_G|_{T_\omega \wt W_s(\omega)}$ realises an isomorphism $T_\omega \wt W_s(\omega) \simeq E_s(z)$, which conjugates $\dd \wt \varphi_t|_{\wt E_s(\omega)}$ and $\dd \varphi_t|_{E_s(z)}$. Thus the lemma will be proven if we show that we have the direct sum
\[
\wt E_s(z) = T_\omega \wt W_{s, \mathrm{tot}}(z) = T_\omega \wt W_s(z) \oplus \ker \dd \pi_G(\omega).
\]
To see this, take a local trivialization $\wt W_{s, \mathrm{tot}}(z) \to W_{s}(z, \varepsilon) \times G_{d-1}(\R^{2d-1})$ sending $\omega \in G$ on $(z, E_0)$ for some $E_0 \in G_{d-1}(\R^{2d-1})$ and such that $\wt W_s(z)$ is sent to $W_s(z,\varepsilon) \times \{E_0\}$. In these {local} coordinates one has the identifications 
\[T_\omega \wt W_s(z) \simeq E_s(z) \oplus \{0\} \quad \text{and} \quad \ker \dd \pi_G(\omega) \simeq \{0\} \oplus T_{E_0} G_{d-1}(\R^{2d-1}).\] As $T_\omega \wt W_{s, \mathrm{tot}}(z)$ is identified with $E_s(z) \oplus T_{E_0} G_{d-1}(\R^{2d-1})$, the proof is complete.
\end{proof}
We conclude this paragraph by noting that for any $\omega = (z,E) \in \wt K_u$ we have
\[
\begin{aligned}
\dim \wt {E}_u(\omega) +\dim \wt {E}_s(\omega)  &=  \dim E_u(z) + \dim E_s(z) + \dim \ker \dd\pi_G(\omega) \\
& = \dim M - 1 + \dim \pi_G^{-1}(z) \\
&= \dim G - 1,
\end{aligned}
\]
since $\dim G = \dim M + \dim \pi_G^{-1}(z).$

\subsection{Vector bundles}\label{subsec:vectorbundle}
We define the tautological vector bundle $\mathcal{E} \to G$ by
\[
\mathcal{E} = \{(\omega, u) \in \pi_{G}^*(TN)~:~\omega \in G,~u \in [\omega]\},
\]
where $[\omega] = E$ denotes the $(d-1)$ dimensional subspace of {$T_{\pi_G(\omega)}N$} represented by $\omega = (z,E)$ and $\pi_G^*(TN)$ is the pullback bundle of $TN.$
Also, we define the {\og vertical  bundle\fg}  $\mathcal{F} \to G$ by
\[
\cal{F} = \{(\omega,W) \in TG~:~\dd \pi_G(\omega) \cdot W = 0 \}.
\]
It is a subbundle of the bundle $TG \to G$. The dimensions of the fibres $\cal E_{\omega} $ and $\cal F_{\omega} $ of $\cal E$ and $\cal F$ over $\omega$ are given by
\[
\dim \cal E_{\omega} = d-1, \quad \dim \cal F_\omega = \dim \ker \dd \pi_G(\omega) = \dim \pi_G^{-1}(z) = d^2 - d
\]
for any $\omega \in G$ with $\pi_G(\omega) = z$.
Finally, set 
\[
\cal{E}_{k, \ell} = \wedge^k \cal{E}^* \otimes \wedge^\ell \cal{F}, \quad 0 \leqslant k \leqslant d - 1, \quad 0 \leqslant \ell \leqslant d^2- d,
\]
where $\cal E^*$ is the dual bundle of $\cal E$, that is, we replace the fibre $\cal E_{\omega}$ by its dual space $\cal E_{\omega}^*$.  We consider $\cal E^*$ and not $\cal E$ since the map $\dd \varphi_t(\pi_G(\omega)) : \cal E_{\omega} \to \cal E_{\wt \varphi_t(\omega)}$ is expanding for $\omega \in \wt K_u$ and $t \to +\infty$, whereas $\dd \varphi_t(\pi_G(\omega))^{-\top} : \cal E^*_{\omega} \to \cal E^*_{\wt \varphi_t(\omega)}$ is contracting. Here $^{-\top}$ denotes the inverse transpose. Indeed, for $\omega = (z, E_u(z)) \in \wt K_u$ and $u \in E_u(z)^*$ {(here $E_u(z)^*$ is the dual vector space of $E_u(z)$)} one has
\[
\langle \dd \varphi_t(z)^{-\top}u, v \rangle = \langle u, \dd \varphi_{-t}(\varphi_t(z))v \rangle, \quad v \in \dd \varphi_t(z)E_u(z)= E_u(\varphi_t(z)) \in \cal E_{\wt \varphi_t (\omega)},
\]
$\langle \bullet , \bullet \rangle$ being the pairing on $\cal E^*_{\wt \varphi_t(\omega)}$ and $\cal E_{\wt \varphi_t(\omega)}$.
Consequently,  {the map} $\dd \varphi_t(\pi_G(\omega))^{-\top}$ is contracting on $\cal E^*_\omega$ when $\omega \in \wt K_u,$ since $\dd \varphi_{-t}(\varphi_t(z))$ is contracting on $E_u(\varphi_t(z)).$ This fact will be convenient later for the proof of Lemma \ref{lem:alternatedsum} below. 

In what follows we use the notation $\omega = (z, \eta) \in G$ and $u \otimes v \in \cal E_{k, \ell}|_{\omega}$.
By using the flow $\wt \varphi_t$, we introduce a flow
$
\Phi^{k, \ell}_t : \cal E_{k, \ell} \to \cal E_{k, \ell}
$
by setting 
\begin{equation}\label{eq:transfertbundle}
\Phi_t^{k, \ell}(\omega, u \otimes v) = \Bigl(\wt \varphi_t(\omega), ~b_t(\omega) \cdot \left[ \left(\dd \varphi_{t}(\pi_G(\omega))^{-\top}\right)^{\wedge k}(u) \otimes \dd \wt \varphi_{t}(\omega)^{\wedge \ell}(v)\right]\Bigr),
\end{equation}
where we set
\[
b_t(\omega) = |\det \dd \varphi_t(\pi_G(\omega))|_{[\omega]}|^{1/2} \cdot |\det\left( \dd \wt \varphi_t(\omega)|_{\ker \dd \pi_G}\right)|^{-1}.
\]
Here the determinants are taken with respect to any choice of smooth metrics {$g_N$ on $N$ and the induced metrics $g_G$ on $G$}, as follows.  If $\omega = (z, E) \in G$ and $t \in \R$, then the number
$
|\det \dd \varphi_t(z)|_{[\omega]}|
$
is defined as the absolute value of the ratio 
\[
\frac{
(\dd \varphi_t(z)|_{[\omega]})^{\wedge^{d-1}} \cdot \mu_{[\omega]}
}{\mu_{[\wt \varphi_t(\omega)]}},
\]
where {$\mu_{[\omega]} = e_{1, [\omega]} \wedge \cdots \wedge e_{d-1, [\omega]} \in \wedge^{d-1} [\omega]$} (resp. {$\mu_{[\wt \varphi_t(\omega))]} \in \wedge^{d-1}[\wt \varphi_t(\omega)]$}) is a volume element  given by any basis {$e_{1,[\omega]},\dots,e_{d-1, [\omega]}$} of $[\omega]$ (resp. $[\wt \varphi_t(\omega)]$) which is orthonormal with respect to the scalar product induced by $g_N|_{[\omega]}$ (resp. $g_N|_{[\wt \varphi_t(\omega)]}$). The number $|\det\left( \dd \wt \varphi_t(\omega)|_{\ker \dd \pi_G}\right)|$ is defined similarly. {If we pass from one orthonormal basis to another one, we multiply the terms by the determinant of a unitary matrix and the absolute value of the above ratio is the same.}
{On the other hand, for a periodic point $\omega_{\wt \gamma} = \wt \varphi_{ \tau(\gamma)} (\omega_{\wt \gamma} )$ this number is simply 
$|\det \dd \varphi_{\tau(\gamma)}(\pi_G(\omega_{\wt \gamma}))|_{[\omega_{\wt \gamma}]}|$.}
Taking local trivializations of $\cal E^*$ and $\cal F$, we see that the action of $\Phi_t^{k, \ell}$ is smooth.  Thus we have the following diagram:
\[\begin{CD}  \cal E_{k, \ell} @>\Phi_t^{k, \ell} >> \cal E_{k, \ell} \\
@VVV   @VVV\\
G  @> \tilde{\varphi}_t >>  G\\
@VV\pi_GV          @VV\pi_GV \\
N @>\varphi_t >>  N \end{CD}
\vspace{0.2cm}
\]
Now, consider the transfer operator
\[
\Phi^{k, \ell, *}_{-t} : C^\infty(G, \cal E_{k, \ell}) \to C^\infty(G, \cal E_{k, \ell})
\]
defined by
\begin{equation}\label{eq:transfert}
\Phi^{k, \ell, *}_{-t}\ubf(\omega) = \Phi^{k, \ell}_t \bigl[ {\bf u} (\wt \varphi_{-t}(\omega)) \bigr], \quad {\bf u} \in C^\infty(G, \cal E_{k, \ell}).
\end{equation}
Let $\mathbf{P}_{k, \ell} :  C^\infty(G, \cal E_{k, \ell}) \to C^\infty(G, \cal E_{k, \ell})$ be the generator of $\Phi^{k, \ell, *}_{-t}$, which is defined by
\[
\Pbf_{k,\ell} {\bf u}  = \left.\frac{\dd}{\dd t} \Bigl(\Phi^{k, \ell,*}_{-t}{\bf u} \Bigr)\right|_{t=0}, \quad {\bf u} \in C^\infty(G, \cal E_{k, \ell}).
\]
Then we have the equality
\begin{equation}\label{eq:op}
\Pbf_{k,\ell}(f\ubf) = (\wt X f)\ubf + f (\Pbf_{k,\ell}\ubf), \quad f \in C^\infty(G), \quad \ubf \in C^\infty(G, \cal E_{k,\ell}).
\end{equation}
{Fix any norm on $\cal E_{k, \ell}$; this fixes a scalar product on $L^2(G, \cal E_{k, \ell}).$ We also consider the transfer operator $\Phi_{-t}^{k, \ell, *}$ as a strongly continuous semigroup $e^{- t \Pbf_{k, \ell}},\: t \geqslant 0$ with generator $\Pbf_{k, \ell}$} with domain in $L^2(G, \cal E_{k, \ell}).$ {The exponential bound of the derivatives of $\varphi_{-t}$ implies an estimate
\[\| \e^{- t \Pbf_{k, \ell}} \|_{L^2(G, \cal E_{k, \ell}) \to L^2(G, \cal E_{k, \ell})} \leqslant C \e^{\beta t}, \quad t \geqslant C_0 > 0,\]}for some constants $\beta > 0, C_0 > 0$.  
Next, we want to study the spectral properties of the operator $\Pbf_{k,\ell}$ applying the work of Dyatlov--Guillarmou \cite{dyatlov2016pollicott}. For this purpose, one needs to see $\wt K_u$ as the trapped set of the restriction of $\wt \varphi_t$ to some neighborhood {$\wt V_u$} of $\wt K_u$ in $G$, so that $\partial \wt V_u$  has convexity properties with respect to $\wt X$ (see the condition (\ref{eq:localconvex}) below with $\wt Y$ replaced by $\wt X$). These conditions are necessary if we wish to apply the results in \cite{dyatlov2016pollicott}. However, it is not clear that such a neighborhood exists, and one needs to modify $\wt X$ slightly outside a neighborhood of $\wt K_u$ to obtain the desired properties. This is done in \S\ref{subsec:isolatingblocks} below.

\subsection{Isolating blocks}\label{subsec:isolatingblocks}
By \cite[Theorem 1.5]{conley1971isolated}, there exists an arbitrarily small {open} neighborhood $\wt V_u$ of $\wt K_u$ in $G$ such that the following holds.

\begin{enumerate}[label=(\roman*)]
\item The boundary $\partial \wt V_u$ of $\wt V_u$ is smooth,
\item The set $\partial_0 \wt V_u = \{z \in \partial \wt V_u:\wt X(z) \in T_z (\partial \wt V_u)\}$ is a smooth submanifold of codimension $1$ of $\partial \wt V_u$,
\item There is $\varepsilon > 0$ such that for any $z \in \partial \wt V_u$ one has 
\[
\wt X(z) \in T_z(\partial \wt V_u) \quad \implies \quad \wt\varphi_t(z) \notin \mathrm{clos} ~\wt V_u, {\quad t \in \left]-\varepsilon, \varepsilon \right[\setminus \{0\},}
\]
where $\clos~A$ denotes the closure of a set $A$.
\end{enumerate}
In what follows we denote
\[
\Gamma_\pm(\wt X) = \{z \in \wt V_u:\wt \varphi_t(z) \in \wt V_u,~ \mp t > 0\}.
\]
A function $\tilde \rho \in C^\infty(\clos~\wt V_u, \R_{\geqslant 0})$ will be called a boundary defining function for $\wt V_u$ if we have $\partial \wt V = \{z \in \clos~ \wt V_u: \tilde \rho(z) = 0\}$ and $\dd \tilde \rho(z) \neq 0$ for any $z \in \partial \wt V_u$. 

By \cite[Lemma 2.3]{guillarmou2021boundary} (see also \cite[Lemma 5.2]{delarue2022resonances}), we have the following result.

\begin{lemm}\label{lem:convexity}
For any small neighborhood $\wt W_0$ of $\partial \wt V_u$ in $\clos~ \wt V_u$, we may find a vector field $\wt Y$ on $\clos~ \wt V_u$ which is arbitrarily close to $\wt X$ in the $C^\infty$-topology, such that the following holds.
\begin{enumerate}[label=\normalfont(\arabic*)]
\item $\supp(\wt Y - \wt X) \subset \wt W_0$,
\item
$
\Gamma_\pm(\wt X) = \Gamma_\pm(\wt Y),
$
where $\Gamma_\pm(\wt Y)$ is defined as $\Gamma_\pm(\wt X)$ by replacing the flow $(\wt \varphi_t)$ by the flow generated by $\wt Y$,
\item For any defining function $\tilde \rho$ of $\wt V_u$ and  any $\omega \in \partial \wt V_u$ we have
\begin{equation}\label{eq:localconvex}
\wt Y \wt\rho(\omega) = 0 \quad \implies \quad \wt Y^2 \wt\rho(\omega) < 0.
\end{equation}
\end{enumerate}
\end{lemm}

From now on, we will fix $\wt V_u, \wt W_0$ and $\wt Y$ as above.  By {\cite[Lemma 1.1]{dyatlov2016pollicott}} we may find a smooth extension of $\wt Y$ on $G$ (still denoted by $\wt Y$) so that for every $\omega \in G$ and $t \geqslant 0$, we have
\begin{equation}\label{eq:globalconvex}
\omega, \wt\varphi_t(\omega) \in \clos~ \wt V_u \quad \implies \quad \wt\varphi_\tau(\omega) \in \clos~ \wt V_u \text{ for every } \tau \in [0, t].
\end{equation}
Let {$(\wt\psi_t)_{t \in \R} $} be the flow generated by $\wt Y$. Set $\wt \Gamma_\pm = \Gamma_\pm(\wt Y)$ for simplicity. The extended unstable/stable bundles $\wt E_\pm^* \subset T^*\wt V_u$ over $\wt \Gamma_\pm$ are defined by
\[
\wt E_\pm^*(\omega) = \{\eta \in T^*_\omega \wt V_u~:~{\Psi_t(\eta)} \to_{t \to \pm\infty} 0\},
\]
where $\Psi_t$ is the symplectic lift of {$\wt \psi_t$}, that is
\[
{\Psi_t(\omega, \eta) = \left(\wt \psi_t (\omega), \dd \wt \psi_t(\omega)^{-\top} \cdot \eta \right)}, \quad (\omega, \eta) \in T^*G, \quad t \in \R,
\]
and $^{-\top}$ denotes the inverse transpose. Then by {\cite[Lemma 1.10]{dyatlov2016pollicott}}, the bundles $\wt E^*_\pm(\omega) $ depend continuously on $\omega \in \wt \Gamma_\pm$, and for any smooth norm $|\cdot|$ on $T^*G$ with some constants {$C > 0, \beta > 0$} {independent of $\omega, \eta$ for $t  \to \mp \infty$} we have
\[
|{\Psi_{\pm t}(\omega, \eta)}| \leqslant C \e^{-\beta |t|}|\eta|, \quad \eta \in E^*_\pm(\omega).
\]

\subsection{Dyatlov--Guillarmou theory}\label{subsec:dyatlovguillarmou}

Let $\nabla^{k, \ell}$ be any smooth connection on $\cal E_{k, \ell}$. Then by (\ref{eq:op}) we have
\[
\Pbf_{k,\ell} = \nabla^{k, \ell}_{\wt X} + \Abf_{k, \ell}
\]
for some $\Abf_{k, \ell} \in C^\infty(G,\End (\cal E_{k, \ell}))$. We define a new operator $\Qbf_{k, \ell}$ by setting
\[
\Qbf_{k,\ell} = \nabla^{k, \ell}_{\wt Y} + \Abf_{k, \ell} : C^\infty(G, \cal E_{k, \ell}) \to C^\infty(G,\cal E_{k, \ell}).
\]
Note that $\Qbf_{k, \ell}$ coincides with $\Pbf_{k,\ell}$ near $\wt K_u$ since $\wt Y$ coincides with $\wt X$ near $\wt K_u$. Clearly, we have
\begin{equation}\label{eq:opy}
\Qbf_{k,\ell}(f\ubf) = (\wt Y f)\ubf + f (\Qbf_{k,\ell} \ubf), \quad f \in C^\infty(G), \quad \ubf \in C^\infty(G, \cal E_{k,\ell}).
\end{equation}
{Next, consider the transfer operator}
$\e^{-t\Qbf_{k,\ell}} : C^\infty(G, \cal E_{k,\ell}) \to C^\infty(G, \cal E_{k,\ell})$ 
with generator $\Qbf_{k,\ell}$, that is,
\[
\partial_t \e^{-t \Qbf_{k,\ell}} \ubf = -\Qbf_{k,\ell} \e^{-t\Qbf_{k,\ell}} \ubf, \quad \ubf \in C^\infty(G, \cal E_{k,\ell}), \quad t \geqslant 0.
\]
{As above, for some constant $C > 0$}, we have
\[
\|\e^{-t\Qbf_{k,\ell}}\|_{L^2(G, \cal E_{k, \ell}) \to L^2(G, \cal E_{k, \ell})} \leqslant C\e^{C|t|}, \quad t \geqslant 0.
\]
{Then for for $\Re(s) > C$}, the resolvent $(\Qbf_{k,\ell}  + s)^{-1}$ on $L^2(G, \cal E_{k, \ell})$ is given by
\begin{equation}\label{eq:inverse}
(\Qbf_{k,\ell}  + s)^{-1} = \int_0^\infty \e^{-t(\Qbf_{k,\ell} + s)}\dd t : L^2(G, \cal E_{k, \ell}) \to L^2(G, \cal E_{k, \ell}).
\end{equation}
Consider the operator
\[
\Rbf_{k, \ell}(s) = \mathbf{1}_{\wt V_u} (\Qbf_{k,\ell} + s)^{-1} \mathbf{1}_{\wt V_u}, \quad \Re(s) \gg 1,
\]
from $C^\infty_c(\wt V_u, \cal E_{k, \ell})$  to $ \mathcal{D}'(\wt V_u, \cal E_{k, \ell}),$
where ${\mathcal{D}}'(\wt V_u, \cal E_{k, \ell})$ denotes the space of $\cal E_{k, \ell}$-valued distributions.
{Recall that $\wt K_u$ is the trapped set of $\wt\varphi_t$ when restricted to $\wt V_u$. Taking into account (\ref{eq:localconvex}), (\ref{eq:globalconvex}) and \eqref{eq:opy}, we see that the assumptions (A1)--(A5) in \cite[\S0]{dyatlov2016pollicott} are satisfied.} We are in position to apply \cite[Theorem 1]{dyatlov2016pollicott} in order to obtain a meromorphic extension of  $\Rbf_{k, \ell}(s)$ to the whole plane $\C.$ 
Moreover, according to \cite[Theorem 2]{dyatlov2016pollicott}, for every {pole} $s_0 \in \C$ in a small neighborhood of $s_0$ one has the representation
\begin{equation}\label{eq:laurent}
\Rbf_{k, \ell}(s) = \Rbf_{H, k, \ell}(s) + \sum_{j = 1}^{J(s_0)} \frac{ (-1)^{j-1} (\Qbf_{k,\ell}  + s_0)^{j-1} \Pi_{s_0}^{k, \ell}} {(s - s_0) ^j},
\end{equation}
where $ \Rbf_{H, k, \ell}(s): C^\infty_c(\wt V_u, \cal E_{k, \ell}) \rightarrow \mathcal{D}'(\wt V_u, \cal E_{k, \ell})$ is a holomorphic family of operators near $s = s_0$ {and $\Pi_{s_0}^{k, \ell}: \: C_c^{\infty} (\wt V_u, \cal E_{k, \ell}) \to \mathcal{D}'( \wt V_u, \cal E_{k, \ell})$ is a finite rank projector.} Denote by $K_{\Rbf_{H, k, \ell}(s)}$ and $ K_{\Pi^{k, \ell}_{s_0}}$ the Schwartz kernels of the operators $\Rbf_{H, k, \ell}(s)$ and $\Pi^{k, \ell}_{s_0}$, respectively.  Recall the definition of the twisted wavefront set 
\[\WF'(A) = \{ (x, \xi, y , -\eta) : (x, \xi, y , \eta) \in \WF(K_A)\},\]
$K_A$ being the distributional kernel of the operator $A$.
By {\cite[Lemma 3.5] {dyatlov2016pollicott}}, we have
\begin{equation}\label{eq:wfset}
\WF'(K_{\Rbf_{H, k, \ell}}(s)) \subset \Delta(T^*\wt V_u) \cup \Upsilon_+ \cup (\wt E_+^* \times \wt E_-^*).
\end{equation}
Here $\Delta(T^* \wt V_u)$ is the diagonal in $T^*(\wt V_u \times \wt V_u)$,
\[
\Upsilon_+ = \{(\Psi_t(\omega, \Omega), \omega, \Omega)~: ~(\omega, \Omega) \in  T^*\wt V_u,~t \geqslant 0,~\langle\wt Y(\omega), \Omega \rangle = 0\},
\]
while the bundles $\wt E_\pm^*$ and flow $\Psi_t$ are defined in \S\ref{subsec:isolatingblocks}. Finally, we have
\begin{equation}\label{eq:suppK}
\supp(K_{\Pi^{k, \ell}_{s_0}}) \subset \Gamma_+ \times \Gamma_- \quad  \text{ and } \quad \WF'(K_{\Pi^{k, \ell}_{s_0}}(s)) \subset \wt E_{+}^* \times \wt E_-^*.
\end{equation}

\section{Dynamical zeta function for the Neumann problem}\label{sec:neumann}
In this section we prove that the function $\eta_\mathrm N$ admits a meromorphic continuation to the whole complex plane, by relating $\eta_\mathrm N(s)$ to the flat trace of the cut-off resolvent $\Rbf_{k, \ell}(s)$.

\subsection{Flat trace}

First, we recall the definition of the flat trace for operators acting on vector bundles. Consider a manifold $V$,  a vector bundle {$\cal E$ }over $V$ and a continuous operator {$\mathbf{T} : C^\infty_c(V, \cal E) \to \mathcal{D}'(V, \cal E)$}. Fix a smooth density $\mu$ on $V$; this defines a pairing $\langle \bullet , \bullet \rangle $ on {$C^\infty_c(V, \cal E) \times C^\infty_c(V, \cal E^*)$}. Let
\[
{K_\mathbf{T} \in \mathcal{D}'(V \times V, \cal E \boxtimes \cal E^*)}
\] be the Schwartz kernel of $\Tbf$ with respect to this pairing, which is defined by
\[
\langle K_\mathbf T, \pi_1^*\mathbf u \otimes \pi_2^* \mathbf v\rangle = \langle \mathbf T \mathbf u, \mathbf v \rangle, \quad \mathbf u \in C^\infty_c(V, \cal E), \quad \mathbf v \in C^\infty_c(V, \cal E^*),
\]
where the pairing on {$\mathcal{D}'(V \times V, \cal E \boxtimes \cal E^*) \times C^\infty_c(V \times V, \cal E \boxtimes \cal E^*)$} is taken with respect to $\mu\times \mu$. Here, the bundle $\cal E \boxtimes \cal E^* = \pi_1^*\cal E \otimes \pi_2^*\cal E^* \to V$ is given by the tensor product of the pullbacks $\pi_1^*\cal E$, and $\pi_2^*\cal E^*$, where $\pi_1, \pi_2 : V\times V \to V$ denote the projections on the first and the second factor, respectively.

  Denote by $\Delta = \{(x,x)~:~x \in V\} \subset V \times V$ the diagonal in $V \times V$ and consider the inclusion map $\iota_{\Delta} : \Delta \to V\times V,~ (x, x)  \mapsto (x,x)$. Assume that
\begin{equation} \label{eq:wfcondition} 
\WF'(K_\Tbf) \cap \Delta(T^*V\setminus \{0\}) = \emptyset,
\end{equation}
where $\Delta(T^*V \setminus \{0\})$ is the diagonal in $(T^*(V)\setminus \{0\}) \times (T^*(V) \setminus \{0\})$. Then by \cite[Theorem 8.2.4]{horma1983vol}, the pull-back 
\[
{{\iota_\Delta^*K_{\Tbf}\in \mathcal{D}'(V, \ {\mathrm{End}} (\mathcal{E}))}}
\]
is well defined, where we used the identification
$
\iota_\Delta^*(\cal E \boxtimes \cal E^*) \simeq \cal E \otimes \cal E^* \simeq \mathrm{End}(\cal E).
$
If $K_\mathbf T$ is compactly supported,  we define the \textit{flat trace} of $\Tbf$ by
\[
{\tr^\flat \Tbf = \langle \tr_{\mathrm{End}(\cal E)}(\iota_\Delta^* K_\mathbf T), 1\rangle},
\]
where again the pairing is taken with respect to $\mu$. It is not hard to see that the flat trace does not depend on the choice of the density $\mu$.

\subsection{Flat trace of the cut-off resolvent}\label{subsec:flattraceresolv}
We introduce a cut-off function $\wt \chi \in C^\infty_c(\wt V_u)$ such that $\wt \chi \equiv 1$ on $\wt K_u$. For $\varrho \in C^\infty_c(\R^+ \setminus \{0\})$  define 
\[
\Tbf^{k, \ell}_{\varrho} \ubf = \left(\int_0^\infty  \varrho(t)\wt \chi(\e^{-t\Qbf_{k, \ell}} \ubf) \wt \chi \dd t\right), \quad \ubf \in C^\infty(G, \cal E_{k, \ell}).
\]
 As in \cite{Weich2023}, we need to introduce some notations. For simplicity let $\cal H = \cal E \otimes \cal F ,$ where $\cal E \rightarrow G,\: \cal F \rightarrow G$ are bundles over $G$. Let $e^{- t {\bf X}}$ be a transport operator. For $\omega \in G$ and $t > 0$ introduce the {\it parallel transport} map
\[\alpha_{\omega, t} = \alpha_{1, \omega, t} \otimes \alpha_{2, \omega, t} :\: \cal E_{\omega} \otimes \cal F_ {\omega}  \longrightarrow \cal E_{{\tilde{\varphi}}_t(\omega)} \otimes \cal F_{{\tilde{\varphi}}_t(\omega)}\]
given by
\[ \ubf \otimes {\bf v}  \longmapsto (e^{- t {\bf X}}(\ubf \otimes {\bf v} )) (\wt \varphi (t)),\]
where $\ubf, {\bf v}$ are some sections of $\cal E_{\omega}$ and $\cal F_{\omega}$ over $\omega,$ respectively. The definition does not depend on the choice of $\ubf$ and ${\bf v}$ (see \cite[Eq. (0.8)]{dyatlov2016pollicott}). Now if $\wt \gamma(t) = \{ \wt \varphi_t (\omega_0): \: 0 \leqslant t \leqslant \tau(\wt\gamma)\}$ is a periodic orbit, we have
\[\alpha_{j, \omega_0, \tau(\wt \gamma)} = \alpha_{j, \omega_0, t}^{-1}\circ \alpha_{j, \wt \gamma(t), \tau(\wt \gamma) } \circ \alpha_{j, \omega_0, t},\quad j = 1, 2,\]
and therefore the trace
\begin{equation}\label{eq:product}
\tr (\alpha_{\wt\gamma}) = \tr (\alpha_{1, \wt \gamma(t), \tau(\wt \gamma)})\tr (\alpha_{2, \wt \gamma(t), \tau(\wt \gamma)})
\end{equation} 
is independent of $t$. (Here we use the flow $\wt \varphi_t$ instead of $\wt \psi_t$ since for periodic orbits the action of both flows is the same.)} {Returning to the bundle $\cal E_{k, \ell}$, the flow $\Phi^{k, \ell}_t$ and the operator $e^{- t \Qbf_{k, \ell}}$, introduced in $\S2.6,$ the corresponding parallel transport map will be denoted by $\alpha^{k, \ell}_{\omega, t}$ and the trace $\tr (\alpha^{k, \ell}_{\wt\gamma})$ is well defined.\

In the same way we define the linearized Poincar\'e map for $\omega \in \tilde{\gamma}$ and $\tau(\tilde{\gamma})$ by
\[P_{\omega, \tau(\wt\gamma)} = \dd \wt \varphi_{-\tau(\wt\gamma)} (\omega)\big\vert_{E_s(\omega) \oplus E_u(\omega)}.\]
As above, given a periodic orbit $\wt \gamma$, the map $P_{\omega, \tau(\wt \gamma)} $ is conjugated to $ P_{\omega', \tau(\wt \gamma)}$ if $\omega$ and $\omega'$ lie on $\wt\gamma$  and we define $\wt P_{\gamma}$ as $P_{\omega, \tau(\wt \gamma)}.$ 
{To define the flat trace ,we must check the condition (\ref{eq:wfcondition}) concerning the intersection of the wave front $\WF'(K_{\Tbf^{k, \ell}_\varrho})$ of  the kernel $K_{\Tbf^{k, l}_\varrho}$ of $\Tbf^{k, \ell}_\varrho$ and the conormal bundle $N^*(\Delta_{\wt V_u})$, $\Delta_{\wt V_u}$ being the diagonal in $(\wt V_u  \times \wt V_u).$  This is down in Section 3.1 of \cite{Weich2023}. We omit the repetition and refer to this paper for a detailed exposition}.

We may apply the Guillemin trace formula \cite[\S2 of Lecture 2]{guillemin1977lectures} (we refer to \cite[Lemma 3.1]{Weich2023} for a detailed presentation based on the argument of \cite[Appendix B]{dyatlov2016zetafunction}), which implies that the flat trace of $\mathbf T_\varrho^{k, \ell}$ is well defined, and
\begin{equation}\label{eq:atiyahbott}
{\tr^\flat(\Tbf^{k, \ell}_\varrho) = \sum_{\wt \gamma} \frac{\varrho(\tau_{\gamma})\tau^\sharp(\gamma)\tr (\alpha^{k, \ell}_{\wt \gamma})} {|\det(\id - \wt P_{\gamma})|},}
\end{equation}
where the sum runs over all periodic orbits $\tilde{\gamma}$ of $\tilde{\varphi}_t$. Here, 
\[
\wt P_\gamma = \left.\dd \wt \varphi_{-\tau(\gamma)}(\omega_{\wt \gamma})\right|_{\wt E_u(\omega_{\wt \gamma}) \oplus \wt E_s(\omega_{\wt \gamma})}\]
is the linearized Poincar\'e map of the closed orbit 
\[
t \mapsto \tilde \gamma(t) = \left(\gamma(t), E_u(\gamma(t))\right)
\]
of the flow $\wt \varphi_t$ and $\omega_{\tilde \gamma} \in \mathrm{Im}(\tilde \gamma)$ is any reference point taken in the image of $\tilde \gamma$. Note that if we take another point $\omega'_{\tilde \gamma} \in \mathrm{Im}(\tilde \gamma)$, then the map $\dd \wt \varphi_{-\tau(\gamma)}(\omega'_{\wt \gamma})$ is conjugated to $\dd \wt \varphi_{-\tau(\gamma)}(\omega_{\wt \gamma})$ by $\dd \wt \varphi_{t_1}(\omega_{\wt \gamma})$, where $t_1 \in \R$ is chosen so that $\wt \varphi_{t_1}(\omega'_{\wt \gamma}) = \omega_{\wt \gamma}$. Hence the determinant $\det(\id - \tilde{P}_\gamma)$ does not depend on the reference point $\omega_{\wt \gamma}$ and is well defined.  The number $\tr (\alpha^{k, \ell}_{\wt \gamma})$ is the trace of the linear map
\[
\alpha^{k, \ell}_{\omega_{\wt \gamma}, \tau(\gamma)} : \cal E_{k, \ell}|_{\omega_{\tilde \gamma}} \to \cal E_{k, \ell}|_{\omega_{\tilde \gamma}},
\]
where for $t \in \R$ and $\omega \in G$, we denote by 
\[
\alpha^{k,\ell}_{\omega, t} : \cal E_{k, \ell}|_{\omega} \to \cal E_{k, \ell}|_{\wt \varphi_t(\omega)}
\]
the restriction of the map $\Phi_t^{k, \ell} : \cal E_{k, \ell} \to \cal E_{k, \ell}$ to the fiber $\cal E_{k, \ell}|_{\omega}$. Again, if we take another reference point $\omega'_{\wt \gamma}$, the map $\alpha^{k, \ell}_{\omega'_{\wt \gamma}, \tau(\gamma)}$ is conjugated to $\alpha^{k, \ell}_{\omega_{\wt \gamma}, \tau(\gamma)}$, hence its trace depends only on $\wt \gamma$, and this justifies the notation $\tr (\alpha^{k, \ell}_{\wt \gamma})$. 

Next, we follow the strategy of \cite[\S3.1]{Weich2023} which is based on that used in \cite[\S4]{dyatlov2016zetafunction} for {Anosov flows on closed manifolds} to compute the flat trace of the (shifted) resolvent defined below.  We  may apply formula \eqref{eq:atiyahbott} with the functions
$
\varrho_{s,T}(t) = \e^{-st} \varrho_T(t),
$
where {$\varrho_T \in C^\infty_c(\R^+)$ satisfies $\supp \varrho_T \subset [\varepsilon/2, T + 1]$ for $0 < \varepsilon < d_0 = \min_{\gamma \in \mathcal P}  \tau(\gamma)$ small and $\varrho_T \equiv 1$ on $[\varepsilon, T]$.}  Then taking the limit $T \to \infty$, we obtain, with (\ref{eq:inverse}) in mind,
\begin{equation}\label{eq:traceresolv}
{\tr^\flat \Rbf^{k, \ell}_{\varepsilon}(s) = \sum_{\tilde{\gamma}} \frac{\e^{-s\tau(\gamma)} \tau^\sharp(\gamma) \tr (\alpha_{\wt \gamma}^{k, \ell})}{|\det(\id-\wt P_\gamma)|}, \quad \Re(s) \gg 1.}
\end{equation}
Here for $\Re(s)$ large enough and $\varepsilon >0$ small, we set
\[
 \Rbf^{k, \ell}_\varepsilon(s) = \wt \chi \e^{-\varepsilon( s + \Qbf_{k, \ell})} (\Qbf_{k, \ell} + s)^{-1} \wt \chi,
\]
and $\varepsilon$ is chosen so that $\e^{-\varepsilon \Qbf_{k,\ell}} \supp(\wt \chi) \subset \wt V_u$, so that $ \Rbf^{k, \ell}_\varepsilon(s)$ is well defined.
The equality \eqref{eq:traceresolv} is exactly the equation concerning $\lim_{T \to \infty} \:\tr^\flat( B_T)$ with $f \equiv 1$ on page 668  in \cite{Weich2023}, and we refer to this work for a detailed proof. Note that the flat trace $\tr^\flat \Rbf_\varepsilon^{k, \ell}(s)$ is well defined
 thanks to the information on the wavefront set  $\WF'(K_{\Rbf_\varepsilon^{k, \ell}(s)})$ obtained from \eqref{eq:wfset}, together with the multiplication properties satisfied by wavefront sets (see \cite[Theorem 8.2.14]{horma1983vol}). 

Next, one states the following result, similar to that in \cite[Section 2]{faure2017semiclassical}.  {This crucial lemma explains the reason to introduce the bundles $\cal E_{k, \ell}.$ For the sake of completeness, we present a detailed proof.} 

\begin{lemm}\label{lem:alternatedsum}
 For any periodic orbit $\tilde{\gamma}$ related to a periodic orbit $\gamma$, we have
\[
\frac{1}{|\det(\id-\wt P_{\gamma})|}\sum_{k = 0}^{d-1}\sum_{\ell = 0}^{d^2 - d} (-1)^{k + \ell} \tr (\alpha_{\wt \gamma}^{k, \ell}) = |\det(\id-P_\gamma)|^{-1/2}.
\]
\end{lemm}

\begin{proof}
Let $\gamma(t)$ be a periodic orbit and let $\tilde{\gamma}(t) = ( \gamma(t), E_u(\gamma(t)),\: \omega_{\tilde{\gamma}} \in \tilde{\gamma}, \: z \in \gamma.$  Set 
\[P_{\gamma,u}  = \dd \varphi_{-\tau(\gamma)}(z)|_{E_u(z)},\quad P_{\gamma,s}  = \dd \varphi_{-\tau(\gamma)}(z)|_{E_s(z)},\]
\[ P_{\gamma, \perp} = \dd \wt \varphi_{-\tau(\gamma)}(\omega_{\tilde \gamma})|_{\ker \dd \pi_G(\omega)}, \quad P_{\gamma, \perp}^{-1} = \dd \wt \varphi_{-\tau(\gamma)}(\omega_{\tilde \gamma})^{-1} \vert_{\ker \dd \pi_G(\omega)} .\]
  The linearized Poincar\'e map $\wt P_\gamma$ of the closed orbit $\wt \gamma$ satisfies

\begin{equation}\label{eq:ptilde}
\begin{aligned}
\det(\id - \wt P_\gamma) &= \det\left(\id - \dd \wt \varphi_{-\tau(\gamma)}|_{\wt E_s(\omega) \oplus \wt E_u(\omega)}\right) \\
&= \det\left(\id - P_\gamma\right) \det\left(\id - P_{\gamma, \perp}\right)
\end{aligned}
\end{equation}
since $\wt E_s(\omega) \simeq E_s(z) \oplus \ker \dd \pi_G(\omega)$ and $\wt E_u(\omega) \simeq E_u(z)$ by Lemma \ref{lem:ident}.
Recall the well known formula
\[
\det(\id- A) = \sum_{j=0}^{k} {(-1)^j} \tr \wedge^jA
\]
for any endomorphism $A$ of a $k$-dimensional vector space. Moreover, notice that {\[\tr (\alpha_{\wt \gamma}^{k, \ell}) = b_{\tau(\gamma)}(\omega_{\tilde{\gamma}})\tr \wedge^k P_{\gamma, u} \tr \wedge^{\ell} P_{\gamma, \perp}^{-1},\]
since by \eqref{eq:transfertbundle}, $\alpha_{\wt \gamma}^{k, \ell}$ coincides with the map
\[
b_{\tau(\gamma)}(\omega_{\tilde{\gamma}})\wedge^k \left[\dd \varphi_{\tau(\gamma)}(\pi_G(\omega_{\wt \gamma}))^{-\top}\right] \otimes \wedge^\ell \left[\dd \varphi_{\tau(\gamma)}(\omega_{\wt \gamma})\right] :  \wedge^k \cal E^*|_{\omega_{\wt \gamma}} \otimes \wedge^\ell \cal F|_{\omega_{\wt \gamma}} \to \wedge^k \cal E^*|_{\omega_{\wt \gamma}} \otimes \wedge^\ell \cal F|_{\omega_{\wt \gamma}}.
\]}Therefore, one gets
\begin{equation}\label{eq:altsum}
\begin{aligned}
&{\sum_{\ell = 0}^{d^2 - d} \sum_{k= 0}^{d- 1} (-1)^{k + \ell}}{\tr (\alpha_{\wt \gamma}^{k, \ell})} \\
&\hspace{1cm}= b_{\tau(\gamma)}(\omega_{\tilde{\gamma}})\left(\sum_{k=0}^{d-1}(-1)^k \tr \wedge^k P_{\gamma,u }\right)\left(\sum_{\ell = 0}^{d^2-d} (-1)^\ell \tr \wedge^\ell P_{\gamma, \perp}^{-1}\right)
\\ &\hspace{1cm}= |\det(P_{\gamma, u})|^{-1/2}|\det(P_{\gamma, \perp})|
\det(\mathrm{Id}-P_{\gamma,u}) \det(\mathrm{Id}-P_{\gamma, \perp}^{-1}).
\end{aligned}
\end{equation}
Here we have used the equality
\[
\begin{aligned}
b_{\tau(\gamma)}(\omega_{\wt \gamma}) &= |\det \dd \varphi_{\tau(\gamma)}(\pi_G(\omega_{\wt \gamma}))|_{[\omega_{\wt \gamma}]}|^{1/2} \cdot |\det\left( \dd \wt \varphi_{\tau(\gamma)}(\omega_{\wt \gamma})|_{\ker \dd \pi_G}\right)|^{-1} \\
&=  |\det(P_{\gamma, u})|^{-1/2}|\det(P_{\gamma, \perp})|
\end{aligned}
\]
which holds because $P_{\gamma, u}$ and $P_{\gamma, \perp}$ are defined with $\dd\varphi_{-t}$ and $\dd\tilde{\varphi}_{-t}$, respectively.
Therefore \eqref{eq:ptilde} yields
\begin{equation}\label{eq:altsum}
\sum_{k, \ell}(-1)^{k + \ell}  \frac{{\tr (\alpha_{\wt \gamma}^{k, \ell})}}{|\det(\id-\wt P_\gamma)|} = \frac{ \det(\id-P_{\gamma,u}) \det(\id-P_{\gamma, \perp}^{-1})|\det(P_{\gamma, u})|^{-1/2}}{|\det(\id-P_\gamma)| |\det(\id-P_{\gamma, \perp})||\det(P_{\gamma, \perp})|^{-1}}.
\end{equation}
 Since $P_{\gamma}$ is a linear symplectic map, we have
\[
\det(\id-P_{\gamma,s}^{-1}) = \det(\id-P_{\gamma, u}),\: \det(P_{\gamma,s}) = \det(P_{\gamma,u}^{-1} ),\] 
 and  one deduces
\[
\begin{aligned}
|\det(\id-P_\gamma)| &= |\det(\id-P_{\gamma,u})||\det(\id-P_{\gamma,s})| \\
&=  |\det (P_{\gamma, s})| |\det(\id - P_{\gamma, u})| |\det( \id -  P_{\gamma, s}^{-1})|\\
 &=   |\det (P_{\gamma, u})|^{-1} |\det(\id-P_{\gamma, u})|^2. 
\end{aligned}
\]
For $ t > 0$ the map $\dd \wt \varphi_{t} = (\dd \wt \varphi_{-t})^{-1} $ is contracting on $\ker \dd \pi_G \subset \tilde{E}_s(\omega_{\tilde{\gamma}})$ (resp. $\dd \varphi_{-t}$ is contracting on $E_u(z)$) and these contractions yield
 $\det(\id-P_{\gamma, \perp}^{-1}) > 0$ (resp. $\det(\id-P_{\gamma,u}) > 0$). Thus the terms involving $P_{\gamma, \perp}$ in (\ref{eq:altsum}) cancel and since
\[|\det(\id-P_{\gamma})|^{-1/2} = |\det(P_{\gamma,u})|^{1/2} \det(\id-P_{\gamma, u})^{-1}, \]
 the right hand side of (\ref{eq:altsum}) is equal to $|\det(\id- P_{\gamma})|^{-1/2}.$ \end{proof}

\subsection{Meromorphic continuation of $\eta_\mathrm N$}
\label{subsec:meroetan}

From Lemma \ref{lem:alternatedsum} and \eqref{eq:traceresolv}, we deduce that for  $\Re(s) \gg 1$, we have
\[
\eta_\mathrm{N}(s) = \sum_{k = 0}^{d-1} \sum_{\ell = 0}^{d^2 - d} (-1)^{k+\ell}  \tr^\flat \Rbf_{\varepsilon}^{k, \ell}(s),
\]
where $\eta_{\mathrm N}(s)$ is defined by
\[
\eta_\mathrm{N}(s) = \sum_{\gamma} \frac{\tau^\sharp(\gamma) \e^{-\tau(\gamma) s}}{|\det(1-P_\gamma)|^{1/2}}.
\]
Since for every $k, \ell$ the family $s \mapsto \Rbf_\varepsilon^{k, \ell}(s)$ extends to a meromorphic family on the whole complex plane, so does $s \mapsto \eta_{\mathrm N}(s).$
Indeed, it follows from the proof of {\cite[Lemma 3.2]{dyatlov2016pollicott}} that $s \mapsto K_{\Rbf_\varepsilon^{k, \ell}(s)}$ is continuous as a map\footnote{This follows from the fact that the estimates on the wavefront set of $\Rbf_{\varepsilon}^{k, \ell}(s)$ given in \cite[Lemma 3.5]{dyatlov2016pollicott} are locally uniform with respect to $s \in \mathbb{C}$.}
\[
\C \setminus \mathrm{Res}(\Rbf_{\varepsilon}^{k,\ell}) \to \mathcal{D}'_\Gamma(G \times G, \cal E_{k, \ell} \boxtimes \cal E_{k, \ell}^*).
\]
Here for $s \notin \mathrm{Res}(\Rbf_{\varepsilon}^{k,\ell})$ the distribution $K_{\Rbf_\varepsilon^{k, \ell}(s)}$ is the Schwartz kernel of $\Rbf_\varepsilon^{k, \ell}(s)$ and (see (\ref{eq:wfset}) for the notation)
\[
\Gamma = \Delta_\varepsilon \cup \Upsilon_{+, \varepsilon} \cup \wt E_+^* \times \wt E_-^*,
\]
where 
$
\Delta_\varepsilon = \{(\Psi_\varepsilon(\omega, \Omega), \omega, \Omega): (\omega, \Omega) \in T^*(\wt V_u) \setminus \{0\}\}$
and 
\[
\Upsilon_{+,\varepsilon} = \{(\Psi_t(\omega, \Omega), \omega, \Omega): (\omega, \Omega) \in  T^*(\wt V_u) \setminus \{0\},~t\geqslant \varepsilon,~\langle \wt Y(\omega), \Omega \rangle = 0\},
\] 
while $\mathcal{D}'_\Gamma(G \times G, \cal E_{k, \ell} \boxtimes \cal E_{k, \ell}^*)$ is the space of distributions valued in $\cal E_{k, \ell} \boxtimes \cal E_{k, \ell}^*$ whose wavefront set is contained in $\Gamma$. This space is endowed with its usual topology (see \cite[\S8.2]{horma1983vol}). {Thus, outside the set of poles $\mathrm {Res}(\Rbf_{\varepsilon}^{k,\ell})$, we apply the procedure with a flat trace}. In particular, $s \mapsto \tr^\flat \Rbf_{\varepsilon}^{k, \ell}(s)$ is continuous on $\C \setminus \mathrm{Res}(\Rbf_{\varepsilon}^{k, \ell})$ by \cite[Theorem 8.2.4]{horma1983vol}. Finally, Cauchy's formula implies that this map is meromorphic on $\C$ and this completes the proof that the Dirichlet series $ \eta_{\mathrm N}(s) $ admits a meromorphic continuation in $\C$.  

Next, we establish that $\eta_\mathrm N(s)$ has simple poles with integer residues. To do this, we may proceed as in {\cite[\S4]{dyatlov2016pollicott}}. For the sake of completeness we reproduce the argument.  Let $s_0 \in \mathrm{Res}(\Rbf_{k, \ell})$ for some $k, \ell$. Recalling the development \eqref{eq:laurent}, it is enough to show that 
\begin{equation}\label{eq:tracezero}
{\tr^\flat\left(\wt \chi \e^{-\varepsilon (s_0 + \Qbf_{k, \ell}) }\left[(\Qbf_{k, \ell} + s_0)^{j-1} \Pi^{k, \ell}_{s_0} \right]\wt \chi\right) = 0, \quad j \geqslant 2,}
\end{equation}
and 
\begin{equation}\label{eq:tracerank}
{\tr^\flat \left(  \wt \chi  \e^{- \varepsilon(s_0 + \Qbf_{k, \ell})}\Pi^{k, \ell}_{s_0} \wt \chi \right) = \mathrm{rank}~\Pi_{s_0}^{k, \ell}.}
\end{equation}
 In the following we fix $k$ and $\ell$.  We may write
\[
\Pi^{k, \ell}_{s_0} = \sum_{i=1}^m \ubf_i \otimes \mathbf{v}_i,
\]
where $\otimes$ denotes the tensor product and by \eqref{eq:suppK} 
 for $i = 1, \dots, m$ we have
\begin{equation}\label{eq:uivi}
\begin{aligned}
\mathbf{u}_{i} \in \mathcal{D}'(\tilde V_u, \cal E_{k, \ell}), \quad \supp(\mathbf{u}_{i}) \subset \Gamma_+, \quad \WF'(\mathbf{u}_{i}) \subset \tilde{E}^*_+, \\
\mathbf{v}_{i} \in \mathcal{D}'(\tilde V_u, \cal E_{k, \ell}^{*}), \quad \supp(\mathbf{v}_{i}) \subset \Gamma_-, \quad \WF'(\mathbf{v}_{i}) \subset \tilde{E}^*_-.
\end{aligned}
\end{equation}
The relations
\[
\wt E^*_+ \cap  \wt E^*_- \cap (T^*(\wt V) \setminus \{0\}) = \emptyset, 
\] 
make possible to define the pairing $\langle \mathbf{u}_i, \mathbf{v}_p \rangle$ on $\cal E_{k, \ell} \times \cal E_{k, \ell}^*$ for $i, p = 1,\dots,m$ which yields a distribution on $\tilde V_u$. This distribution is compactly supported since
\[
\supp \mathbf{u}_i \cap \supp \mathbf{v}_p \subset \Gamma_+ \cap \Gamma_- = \wt K_u.
\] 
The family $(\mathbf{u}_{i})$ is a basis of the range of $\Pi^{k, \ell}_{s_0}.$
By definition of the flat trace using the information on the wavefront sets and the supports of ${\bf u}_i$ and ${\bf v}_j$, we can write
\begin{equation}\label{eq:traceintegral}
\begin{aligned}
\tr^\flat\bigl(&\wt \chi \e^{- \varepsilon ( s_0 +\Qbf_{k, \ell})}\left[(\Qbf_{k, \ell} + s_0)^{j-1} \Pi^{k, \ell}_{s_0} \right] \wt \chi\bigr)\\
& {\hspace{0.4cm} = \sum_{i = 1}^m \int_{\wt V_u}  \langle\wt \chi \e^{-\varepsilon ( s_0 +\Qbf)} (\Qbf_{k, \ell} + s_0)^{j-1} \mathbf{u}_{i}, \wt \chi \mathbf{v}_{i}\rangle.}
\end{aligned}
\end{equation}
Here the integrals make sense taking into account the estimates of the supports and the wavefront sets of $\mathbf{u}_{i}$ and {$\mathbf{v}_p$} mentioned above. Since $\Pi^{k, \ell}_{s_0}\circ\Pi^{k, \ell}_{s_0} = \Pi^{k, \ell}_{s_0},$ the family {$(\mathbf{v}_{p})$} is dual to the basis $(\mathbf{u}_{i})$ in the sense that 
\begin{equation}\label{eq:kronecker}
\int_{\wt V_u} \langle \mathbf{u}_{i} , \mathbf{v}_{p}\rangle = \delta_{ip}, \quad 1 \leqslant i, p \leqslant m,
\end{equation}
where $\delta_{ip}$ are the Kronecker symbols. Introduce
\[
C^{(j)}_{s_0, k, \ell} = \bigl\{\mathbf{u} \in \mathcal{D}'(\tilde V_u, \cal E_{k, \ell}):\:  {\rm supp}\: u \subset \Gamma_{+}, \: \WF(u) \subset \tilde{E}^*_{+},\: (\Qbf_{k, \ell} + s_0)^j \mathbf{u} = 0\bigr\}.
\]
Then, since $\wt \chi = 1$ near $\wt K_u,$ by applying (\ref{eq:traceintegral}), one deduces that {$\wt \chi \e^{-\varepsilon (s_0 + \Qbf_{k, \ell})} (\Qbf_{k, \ell} + s_0) \Pi^{k, \ell}_{s_0} \wt \chi$} maps
\[
\begin{aligned}
\wt \chi C_{s_0, k, \ell}^{(j + 1)} \to \wt \chi C_{s_0, k, \ell}^{(j)}, \quad \text{and} \quad \wt \chi C_{s_0, k, \ell}^{(1)} \to \{0\}, \quad j \geqslant 1.
\end{aligned}
\]
This fact and \eqref{eq:kronecker} show that \eqref{eq:tracezero} holds. To prove \eqref{eq:tracerank}, we write
\[
\begin{aligned}
&\sum_i\int_{\wt V_u}  \langle \wt \chi \e^{-\varepsilon (s_0 + \Qbf_{k, \ell})} \mathbf{u}_{i} , \wt \chi \mathbf{v}_{i} \rangle \\
& \hspace{2cm}= \sum_i \int_{\wt V_u} \langle  \mathbf{u}_{i},  \mathbf{v}_{i}\rangle \
{ - \sum_i \int_{0}^{\varepsilon} \dd t \int_{\wt V_u} \left\langle \wt \chi \e^{-t(s_0 + \Qbf_{k, \ell})} (\Qbf_{k, \ell}+s_0) \mathbf{u}_i, \wt \chi \mathbf{v}_i \right \rangle.}
\end{aligned}
\]
Now, we replace $\varepsilon$ by $t$ in \eqref{eq:tracezero} for any $t \in [0, \varepsilon]$, and we obtain that the last sum in the right hand side of the above equation vanishes. Finally, applying \eqref{eq:kronecker},  we obtain \eqref{eq:tracerank}.

\section{Dynamical zeta function for particular rays}\label{sec:dir}
In this section we adapt the above construction to prove the following result. 
\begin{theo}
\label{thm:dirichletq}
Let $q \in \mathbb N_{\geqslant 1}$. The function $\eta_q(s)$ defined by
\[
\eta_q(s) = \sum_{\gamma \in {\mathcal P},\:m(\gamma) \in q \mathbb N} \frac{\tau^\sharp(\gamma) \e^{-s\tau(\gamma)}}{|\det(\mathrm{Id}-P_\gamma)|^{1/2}}, \quad \Re(s) \gg 1,
\]
where the sum runs over all periodic rays $\gamma$ with $m(\gamma) \in q \N$, admits a meromorphic continuation to the whole complex plane with simple poles and residues valued in $\mathbb Z / q$.
\end{theo}

Note that for large $\Re(s)$ we have the formula
\begin{equation}\label{eq:4.1}
\eta_\mathrm D(s) = 2 \eta_2(s) - \eta_\mathrm N(s).
\end{equation}
In particular, Theorem  \ref{thm:dirichletq} implies that $\eta_\mathrm D(s)$ also extends meromorphically to the whole complex plane, since $\eta_\mathrm N(s)$ does by the preceding section. In particular, we obtain Theorem \ref{thm:main} since $2 \eta_2(s)$ has simple poles with residues in $\Z.$ 

\subsection{The $q$-reflection bundle}\label{subsec:qreflexion}
For $q \geqslant 2$ define the \textit{$q$-reflection bundle} $\cal R_q \to M$ by
\begin{equation}\label{eq:bundle}
\cal R_q = \left. \left(\left[S\R^d \setminus \left(\pi^{-1}(\mathring{D}) \cup \cal D_\mathrm g \right)\right] \times \R^q\right) \right/ \approx,
\end{equation}
where the equivalence classes of the relation $\approx$ are defined as follows. For $(x,v) \in S\R^d \setminus \left(\pi^{-1}(\mathring{D}) \cup \cal D_\mathrm g \right)$ and $\xi \in \R^q$, we set
\[
[(x,v,\xi)] = \left\{(x,v, \xi) , (x, v', A(q) \cdot \xi)\right\} \quad \text{ if } (x,v) \in \cal D_\mathrm{in},\: (x, v') \in \cal D_\mathrm{out},
\]
where $A(q)$ is the $q \times q$ matrix with entries in $\{0, 1\}$ given by
\[A(q) = 
\begin{pmatrix} 
			    0 &   &  & 1\\
                              1 & 0 & & \\
                               & \ddots & \ddots & \\              
                               & & 1  & 0
                              \end{pmatrix}.
\]
{Clearly, the matrix $A(q)$ yields a shift permutation 
\[A(q) (\xi_1, \xi_2, ..., \xi_q) = (\xi_q, \xi_1,...,\xi_{q-1}).\]}This indeed defines an equivalence relation since $(x,v') \in \cal D_\mathrm{out}$ whenever $(x,v) \in \cal D_\mathrm{in}$. Note that
\begin{equation}\label{eq:Aq}
A(q)^q = \id, \quad \tr A(q)^j = 0, \quad j = 1, \dots, q-1.
\end{equation}
Let us describe the smooth structure of $\cal R_q$, using the charts of $M$ and  the notations of \S\ref{subsec:smooth}. {For $z_\star \in \cal D_\mathrm{in}$, let $U_{z_\star} = B(0, \delta) = \{x \in \R^{d-1}: \: |x| < \delta\}$ be a neighborhood of $0$ used for the definition of $F_{z_\star}$ (see \S\ref{subsec:smooth}) and let 
\[\Psi_{z_\star}^{-1} : {\mathcal O}_{z_\star} \rightarrow ]-\epsilon, \epsilon[ \times B(0, \delta) \times B(0, \delta)= W_{z_\star}\] be a chart.} Then the bundle $\cal R_q \to M$ can be defined by defining its transition maps, as follows. {Let $W = \Psi^{-1} (B \setminus \pi^{-1}(\partial D))$ be a chart.} In the smooth coordinates introduced in \S\ref{subsec:smooth}, we have {$W_{z_\star} \cap W = W_+ \sqcup W_-$}, where
\[
{W_+ = \left]0, \varepsilon\right[  \times B(0, \delta) \times B(0, \delta)\quad \text{ and }  \quad W_- = \left]-\varepsilon, 0\right[  \times B(0, \delta) \times B(0, \delta).}
\]
Then we define the transition map {$\alpha_{z_\star} : W_{z_\star} \cap W \to \mathrm{GL}(\R^q)$} of the bundle $\cal R_q$ with respect to the pair of charts {$(\Psi_{z_\star}, \Psi)$} to be the locally constant map defined by 
\[
\alpha_{z_\star}(z) = \left\{ \begin{matrix} \id &\text{ if } z \in W_-, \vspace{0.2cm} \\ A(q) &\text{ if } z \in W_+.\end{matrix} \right.
\]
For $z_\star, z_\star' \in \cal D_\mathrm{in}$, the transition map of $\cal R_q$ for the pair of charts {$(\Psi_{z_\star},\Psi_{z_\star'})$} is declared to be constant and equal to $\id$ on {$W_{z_\star} \cap W_{z_\star'}$}. In this way we obtain a smooth bundle $\cal R_q$ over $M$, which is clearly homeomorphic to the quotient space (\ref{eq:bundle}). Since the transition maps of $\cal R_q$ are locally constant, there is a natural flat connection $\dd^{q}$ on $\cal R_q$ which is given in the charts by the trivial connection on $\R^q$. 

Consider a small smooth neighborhood $V$ of $K$. As in \S\ref{subsec:anosov}, we embed $V$ into a smooth compact manifold without boundary $N$, and we fix an extension of $\cal R_q$ to $N$ (this is always possible if we choose $N$ to be the {double} manifold of $V$). Consider any connection $\nabla^q$ on the extension of $\cal R_q$ which coincides with $\dd^q$ near $K$, and denote by 
\[
P_{q,t}(z) : \cal R_q(z) \to \cal R_q(\varphi_t(z))
\]
the parallel transport of $\nabla^q$ along the curve $\{\varphi_\tau(z)~:~0 \leqslant \tau \leqslant t\}$.
We have a smooth action of {$\varphi_t^q$ on $\cal R_q$ which is given by the horizontal lift of $\varphi_t$}
\[
\varphi_t^q(z, \xi) =(\varphi_t(z), P_{q,t}(z) \cdot \xi), \quad (z, \xi) \in \cal R_q.
\]
{As in (\ref{subsec:flattraceresolv}), we see that for a periodic orbit $ \gamma$ we define $P_{q, \gamma}$ as an endomorphism on $\R^q$}. From \eqref{eq:Aq},  and the fact that $\nabla^q$ coincides with $\dd^q$ near $K$, we easily deduce that for any periodic orbit $\gamma = (\varphi_\tau(z))_{\tau \in [0, \tau(\gamma)]}$, we have

\begin{equation}\label{eq:holonomy}
{\tr (P_{q, \gamma})} = 
\left\{
\begin{matrix} q &\text{ if } & m(\gamma) = 0 \mod q, \vspace{0.2cm}\\ 
0 &\text{ if } & m(\gamma) \neq 0 \mod q.
\end{matrix}
\right.
\end{equation}

\subsection{Transfer operators acting on $G$}\label{subsec:actionG}
Now, consider the bundle 
\[\cal E_{k, \ell}^q = \cal E_{k, \ell} \otimes \pi_G^*\cal R_q,\] where $\pi_G^*\cal R_q$ is the pullback of $\cal R_q$ by $\pi_G$ and $\cal E_{k, \ell}$ is defined in \S\ref{subsec:vectorbundle}, so that $\pi_G^* \cal R_q \rightarrow G$ is a vector bundle over $G$. We may lift the flow $\varphi_t^q$ to a flow $\Phi^{k, \ell, q}_t$ on $\cal E_{k, \ell}^q$ which is defined locally near $\wt K_u$ by
\[
\begin{aligned}
&\Phi^{k, \ell, q}_t(\omega, u \otimes v \otimes \xi) \\
 & \quad \quad = \Bigl(\wt \varphi_t(\omega), ~b_t(\omega) \cdot \left[ \left(\dd \varphi_{t}(\pi_G(\omega))^{-\top}\right)^{\wedge k}(u) \otimes (\dd \wt \varphi_t(\omega))^{\wedge \ell}(v) \otimes P_{q, t}(z) \cdot \xi \right]\Bigr)
 \end{aligned}
\]
for any $\omega = (z, E) \in G, \: u \otimes v \otimes \xi \in \cal E_{k, \ell}^q(\omega)$ and $t \in \R$. Here $b_t(\omega)$ is defined in \S\ref{subsec:vectorbundle}. As in \S\ref{subsec:dyatlovguillarmou}, we consider a smooth connection $\nabla^{k, \ell, q} = \nabla^{k, \ell} \otimes \pi_G^* \nabla^q$ on $\cal E_{k, \ell}^q$. Define the transfer operator
\[\Phi^{k, \ell, q, *}_{-t}: C^{\infty} (G, \cal E_{k, \ell}^q) \rightarrow C^{\infty} (G, \cal E_{k, \ell}^q)\]
by
\[\Phi^{k, \ell, q, *}_{-t}{\bf u} (\omega) = \Phi^{k, \ell, q}_{t}[{\bf u} (\tilde{\varphi}_{-t}(\omega)],\: {\bf u} \in C^{\infty}(G, \cal E_{k, \ell}^q).\]
Then  the operator
\[
\Pbf_{k, \ell, q} = \left.\frac{\dd}{\dd t} \Bigl(\Phi^{k, \ell, q, *}_{-t} \Bigr)\right|_{t=0},\:{\bf u} \in C^{\infty}(G, \cal E_{k, \ell}^q)\]
which is defined near $\wt K_u$, can be written locally as $\nabla^{k, \ell, q}_{\wt X} + \mathbf{A}_{k, \ell, q}$ for some $\mathbf{A}_{k, \ell, q} \in C^\infty(\wt U_u, \End\: \cal E_{k, \ell}^q)$ defined in some small neighborhood $\wt U_u$ of $\wt K_u$. Next, we choose some $\mathbf{B}_{k, \ell, q} \in C^\infty(G, \End\: \cal E_{k, \ell}^q)$ which coincides $\mathbf{A}_{k, \ell, q}$ near $\wt K_u$. We consider $\wt V_u$ and $\wt Y$ as in \S\ref{subsec:isolatingblocks}, and set 
\[
\Qbf_{k, \ell, q} = \nabla^{k, \ell, q}_{\wt Y}  + \mathbf{B}_{k, \ell, q} : C^\infty(G, \cal E_{k, \ell}^q) \to  C^\infty(G, \cal E_{k, \ell}^q).
\]

\subsection{Meromorphic continuation of $\eta_q(s)$}
For $\wt \chi \in C^\infty_c(\wt V_u)$ such that $\wt \chi \equiv 1$ near $\wt K_u$, define
\[
\Rbf_{\varepsilon}^{k, \ell, q}(s) = \wt \chi\e^{-\varepsilon(\Qbf_{k, \ell, q}+s)}  (\Qbf_{k, \ell, q}  + s)^{-1} \wt \chi.
\]
Repeating the argument of the preceding section, one can obtain an analog of (\ref{eq:traceresolv}), where the factor $\tr (\alpha^{k, \ell}_{\wt \gamma})$ 
must be replaced by $\tr (\alpha^{k, \ell}_{\wt \gamma})\tr (P_{q, \gamma}).$
This leads to a meromorphic continuation of $\Rbf_{\varepsilon}^{k, \ell, q}(s).$ 

On the other hand, by (\ref{eq:holonomy}) one gets ${\tr (P_{q, \gamma})} = {\bf 1}_{q\N}(m(\gamma))$.
In particular,  proceeding exactly as in the the preceding section, we obtain that for $\Re(s)$ large enough we have
\begin{equation}\label{eq:traceetaq}
\sum_{k = 0}^{d -1}\sum_{\ell= 0}^{d^2 - d} (-1)^{k + \ell}\tr^\flat \Rbf_{\varepsilon}^{k, \ell, q}(s) = q{\sum_{\substack{\gamma \in {\mathcal P} \\ {m(\gamma) \in q \N}}}\frac{\tau^\sharp(\gamma)\e^{-s \tau(\gamma)} }{|\det(\mathrm{Id}-P_\gamma)|^{1/2}}.}
\end{equation}
Therefore, repeating the argument of \S\ref{sec:neumann}, we establish  a meromorphic continuation of the function $s \mapsto \eta_q(s).$ Finally, by using \eqref{eq:traceetaq}, we may proceed exactly as in \S\ref{subsec:meroetan} to show that $q \eta_q(s)$ has integer residues. This completes the proof of Theorem \ref{thm:dirichletq}.

\section{Modified Lax--Phillips conjecture for real analytic obstacles}\label{sec:real}
\def\ec{{\mathcal E}}
In this section, we assume that the obstacles $D_1, \dots, D_r$ have real analytic boundary. Then the smooth structure on $M$ defined in \S\ref{subsec:smooth} induces an analytic structure on $M$. Indeed, with notations of \S\ref{subsec:smooth}, the local parametrizations $F_{z_\star}$ of $\cal D_\mathrm{in}$ can be chosen to be real analytic, as $\mathcal D_\mathrm{in}$ is a real analytic submanifold of $S\R^{d-1}$. This makes the transition maps \eqref{eq:changecoord} real analytic, and thus we obtain a real analytic structure on $M$. In the charts defined by $\Psi_{z_\star}$ and {$\Psi$ (see \S\ref{subsec:qreflexion})}, the billiard flow $\varphi_t$ is a translation and it defines a real analytic flow. Of course, the Grassmannian bundle $G \to M$ also becomes real analytic. Consequently, the lifted flow $\wt \varphi_t$ on $G$, which is defined by \eqref{eq:defflow}, is real analytic as well. 

Consider the bundles $\mathcal{E}_{k, \ell}^q \to G$ defined in \S\ref{subsec:actionG} for $q \geqslant 2$, $1 \leqslant k \leqslant d-1$ and $1 \leqslant \ell \leqslant d^2 - d$. In the case $q = 1$ the bundles $\ec_{k, \ell}^1 \to G$ are isomorphic to $\ec_{k, \ell}$, $\ec_{k, \ell}$ being the bundles defined in \S\ref{subsec:vectorbundle}. As before, we naturally extend the flow $\wt \varphi_t$ to a flow $\Phi^{k, \ell, q}_t$ (which is non-complete) on $\ec_{k, \ell}^q$. We set
\[
\ec^{+}_q = \bigoplus_{k + \ell \text{ even}} \mathcal{E}_{k, \ell}^q \quad \text{ and } \quad \ec^{-}_q = \bigoplus_{k + \ell \text{ odd}} \mathcal{E}_{k, \ell}^q.\]
Define the flows $\Phi_{t,q}^+$ and $\Phi_{t,q}^-$, acting respectively on the bundles $\cal E_q^+$ and $\cal E_q^-$, by
\[\Phi_{t, q} ^+ = \bigoplus_{k + \ell \text{ even}} \Phi_t^{k, \ell,q} \quad \text{ and } \quad \Phi_{t, q}^{-} = \bigoplus_{k + \ell \text{ odd}} \Phi^{k, \ell, q}_t.\]
 Then $\Phi_{t, q}^\pm$ is a  virtual lift of $\wt \varphi_t$ to the virtual bundle $\ec^{\mathrm{virtual}}_q= \ec_q^{+} - \ec_q^{-}$, in the sense of \cite[p. 176]{fried1995meromorphic}.  Next,
 given a periodic ray $\gamma$, a point $\omega = (z, E) \in G,\: z \in \gamma,$ and a bundle $\xi \rightarrow G$ over $G$, one considers the transformation $\Phi_{\tau(\gamma)}: \xi_{\omega} \rightarrow \xi_{\omega},$ where $\xi_{\omega} $ is the fibre over $\omega$ and $\Phi_t$ is the lift of the flow $\tilde{\varphi}_t$ to $\xi.$ Then we set $\chi_{\gamma}(\xi) = {\rm tr} \: \Phi_{\tau(\gamma)} .$ Following \cite[p. 176]{fried1995meromorphic}, one defines $\chi_{\gamma} (\ec^{+}_q- \ec^{-}_q) = \chi_{\gamma}(\ec^{+}_q) - \chi_{\gamma}( \ec^{-}_q).$ For a periodic ray $\gamma$ related to a primitive periodic ray $\gamma^{\sharp}$ one defines $\mu(\gamma) \in \N$ determined by the equality  $\tau(\gamma) = \mu(\gamma) \tau( {\gamma}^{\sharp}).$ 
 
 After this preparation introduce the zeta function
\[
\zeta_q(s) := \exp\Bigl( - \frac{1}{q} \sum_{\tilde{\gamma}} \frac{\chi_{\gamma}(\ec^{+}_q- \ec^{-}_q) }{\mu(\gamma) |\det(\mathrm{Id}-\tilde{P}_{\gamma})|}\e^{-s \tau(\gamma)}\Bigr), \quad \Re(s) \gg 1.
\]
This function corresponds exactly to the {\it flat-trace function} $s \mapsto T^{\flat}(s)$ introduced by Fried \cite[p. 177]{fried1995meromorphic}. On the other hand, one has
\[\chi_{\gamma}(\ec^{+}_q- \ec^{-}_q) = {\rm tr} \: \Phi_{\tau(\gamma), q} ^{+} - {\rm tr}\: \Phi_{\tau(\gamma), q} ^{-} = \sum_{k, \ell} (-1)^{k + \ell} {\rm tr} \:\Phi^{k, \ell , q}_{\tau(\gamma)}(\omega_{\tilde{\gamma}}).\]
According to the analysis of \S\ref{sec:neumann} for the function $\zeta_\mathrm N(s)$, one deduces 
\[\frac{\dd}{\dd s} \log \zeta_1(s) = \sum_{\gamma \in {\mathcal P}} \frac{\tau(\gamma^{\sharp}) \e^{-s \tau(\gamma)}}{|\det(\mathrm{Id}- P_{\gamma})|^{1/2}}= \eta_{\mathrm N}(s),\quad \Re s \gg 1.\]
Similarly,  the argument of \S\ref{sec:dir} implies
\[ \frac{\dd}{\dd s} \log (\zeta_2(s)^2) = 2 \sum_{\substack{\gamma \in {\mathcal P}\\ m(\gamma) \in 2 \N}}  \frac{\tau(\gamma^{\sharp}) \e^{-s \tau(\gamma)}}{|\det(\mathrm{Id}- P_{\gamma})|^{1/2}} = 2 \eta_2(s),\quad \Re s \gg 1.\]
Consequently, the representation (\ref{eq:4.1}) yields
\begin{equation} \label{eq:5.1} 
\eta_{\mathrm D}(s) = \frac{\dd}{\dd s} \log \Bigl(\frac{ \zeta_2(s)^2 }{\zeta_1(s)}\Bigr),\quad \Re s \gg 1.
\end{equation}
For obstacles with real analytic boundary the flow $\wt \varphi_t$ is real analytic and the bundles $\ec^{\pm}_q$ are real analytic, too.
 
 For convenience of the reader, we recall the definition of the order of a function $f$ meromorphic on the complex plane (see for instance \cite{Hayman1964functions}). For $r \geqslant 0$, denote by  $n(r, f)$ the number of poles of $f$ in the disk $\{|z|\leqslant r\}$ counted with their multiplicity. Introduce the 
(Nevalinna) counting function
\[N(r, f) = \int_0^r \frac{ n(t, f) - n(0, f)}{t} \dd t + n(0, f) \log r.\]
 Let $\log^+: \R \rightarrow \R^+$ be the function defined by
\[\log^+ x = \left\{\begin{matrix} \log x & \text{if} & x \geqslant 1,\\
0 & \text{if} & x \leqslant 1. \end{matrix}\right.\]
The proximity function $m(r, f)$ is defined by
\[m(r, f) = \frac{1}{2\pi}\int_0^{2\pi} \log^+ |f(r \e^{i \theta})| \dd \theta,\]
assuming that $f(z)$ has no poles for $|z| = r.$
Then $T(r, f) = N(r, f) + m(r, f)$ is called the (Nevalinna) characteristic of $f$. Finally, the order $\rho(f)$ of $f$ is defined by
\[\rho(f) = \limsup_{r \to \infty} \frac{ \log^+ T(r, f)}{\log r}.\]

We are now in position to apply the principal result of Fried \cite[Theorem on p. 180, see also pp. 177--178]{fried1995meromorphic} saying that the zeta functions $s \mapsto \zeta_{k}(s),\: k = 1, 2,$ are entire functions with finite orders $\rho(\zeta_k)$.
Thus $\zeta_2^2 / \zeta_1$ is a meromorphic function with order {$\max\{\rho(\zeta_1), \rho(\zeta_2)\}.$}

\begin{proof}[Proof that $\eta_D(s)$ is not an entire function]
We will show that the Dirichlet series $\eta_\mathrm D(s)$ cannot be continued as an entire function to $\mathbb C$, that is, $\eta_\mathrm D(s)$ has at least one pole. We proceed by contradiction and assume that $\eta_\mathrm D(s)$ is an entire function.  
Applying the representation (\ref{eq:5.1}), this means that $\zeta_2(s)^2/\zeta_1(s)$ has neither poles nor zeros.  As we have mentioned above, this function has finite order, so by the Hadamard factorisation theorem we deduce that $ \zeta_2(s)^2/ \zeta_1(s)= \exp(Q(s))$ for some polynomial $Q(s)$. This implies that $\eta_{\mathrm D}(s) = Q'(s)$ is a polynomial, which is impossible. Indeed, since $\eta_{\mathrm D}(s) \to 0$ as $\Re(s) \to +\infty$, this implies that $Q'(s)$ must be the zero polynomial.  By uniqueness of the development of  Dirichlet series of the form $\sum_n a_n \e^{-\lambda_n s}$ \cite{perron1908theorie}  absolutely convergent for $\Re s \geq s_0$, this leads to a contradiction.
\end{proof}

\appendix

\def\ts{\tilde{\sigma}}
\def\p{{\mathcal P}}
\def\h{{\mathcal H}}
\def\pa{\partial}
\def\om{\omega}
\def\ep{\varepsilon} 
\def\lr{\longrightarrow}

\section{Hyperbolicity of the billiard flow}

In this appendix we show that the non-grazing flow $\phi_t$ defined in \S\ref{subsec:defflow} in Euclidean metric is uniformly hyperbolic on the trapped set $K_e.$ Throughout this section  we work with the Euclidean metric. As it was mentioned in $\S\ref{subsec:anosov}$, we can obtain the uniform hyperbolicity of the flow $\varphi_t$ on $K$ in the smooth model from that for $\phi_t$ on $\mathring B \cap K_e.$ The flow $\phi_t$ is hyperbolic on $\mathring B \cap K_e$ if for every $z = (x, v) \in \mathring {B} \cap K_e$ we have a splitting
\[T_z \R^d = \R X(z) \oplus E_s(z) \oplus E_u(z),\]
where $X(z) = v$ and $E_s(z)/E_u(z)$ are stable/unstable spaces such that $\dd \phi_t(z)$ maps $E_{s/u} (z)$ onto $E_{s/u}(\phi_t(z))$ whenever $\phi_t(z) \in \mathring{B} \cap K_e$,  and 
if for some constants $C > 0, \nu > 0$ independent of $z \in K_e$, we have
\begin{equation}\label{eq:anosovest}
\|\dd \phi_t(z) \cdot v\| \leqslant  \begin{cases} C \e^{-\nu t} \|v\|,~ &v \in E_s(z),~t\geqslant 0, \vspace{0.2cm}  \\  C \e^{-\nu |t|} \|v\|,~ &v \in E_u(z),~t\leqslant 0. \end{cases} 
\end{equation}

First, we consider the case of periodic points. Our purpose is to define the unstable and stable manifolds $E_u(z)$ and $E_s(z)$ at a periodic point $z \in \mathring B \cap K_e$, and to estimate the norm of $\dd \phi_t(z)\vert_{E_b(z)}$ for $b = u, s.$ Consider a periodic ray $\gamma$ with reflection points $z_i =(q_i, \omega_i),\: q_i \in \partial D, \: \omega_i \in {\mathbb S}^{d-1},\: i = 0,\dots,m(\gamma) = m$ with period $T > 0.$
Let $\pi_x$ be the projection $\pi_x: (t, x, \tau, \xi) \ni T^*(\R \times \Omega) \rightarrow x \in \Omega$ and let $\tilde{\gamma} \subset T^*(\R \times \Omega)$ be the generalized bicharacteristic of the wave operator $\partial_t^2 - \Delta_x$ for which $\pi_x(\tilde{\gamma}) = \gamma$ (see Section 1.2 in \cite{petkov2017geometry} for the definition of generalized bicharacteristics).
Let $\rho = (x, \xi) \in \tilde{\gamma}\cap \Bigl(T^*(\mathring \Omega) \setminus \{0\}\Bigr)$ be such that $\pi_x(\rho) \neq q_i,\:i = 0,\dots, m.$ Then the flow $\phi_T$ maps a small conic neighborhood ${\mathcal V} \subset T^*(\mathring \Omega) \setminus \{0\}$ of $\rho$ to a conic neighborhood ${\mathcal W} \subset T^*(\mathring \Omega)\setminus \{0\}$ of $\rho$ and 
\[ \dd \phi_T(\rho): T_{\rho}(T^*(\mathring\Omega)) \rightarrow  T_{\rho}(T^*(\mathring\Omega)). \]
The tangent vector $\zeta$ to $\tilde{\gamma}$ at $\rho$ and the cone axis $\eta = \{t\xi, \: t > 0\}$ are invariant with respect to $\dd \phi_T(\rho)$ and we define
the quotient $\Sigma_{\rho} = T_{\rho}(T^*(\mathring\Omega))/E_{\rho},\: E_{\rho}$ being the two dimensional space spanned by $\zeta$ and $\eta$. Then 
\[ P_{\gamma}(\rho)  = \dd\phi_T(\rho) \vert_{\Sigma_{\rho}}\]
is the linear Poincaré map corresponding to $\gamma$ at $\rho$. It is easy to see that if $\mu \in \tilde{\gamma} \cap \Bigl(T^*(\mathring \Omega) \setminus \{0\}\Bigr)$ is another periodic point, the maps $P_{\gamma}(\rho)$ and  $P_{\gamma}(\mu)$ are conjugated and the eigenvalues of $P_{\gamma}(\rho)$ are independent of the choice of $\rho.$

Recall the billiard ball map ${\bf B}$ introduced in \S\ref{subsec:anosov}. The advantage is that ${\bf B}$ is smooth (see \cite{kovachev1988billard}). We will apply the representation of the Poincar\'e map for billiard ball map ${\bf B}$ established in Theorem 2.3.1 and Proposition 2.3.2 in \cite{petkov2017geometry}. To do this, we recall some notations given in Section 2 of \cite{petkov2017geometry}. Let $\Pi_i \subset \R^d$ be the plane passing through $q_i$ and orthogonal to $\omega_i$ and let $\Pi_i'$ be the plane passing through $q_i$ and orthogonal to $\omega_{i-1}.$ For $j = i \: ({\rm mod}\: m)$ we set $\Pi_j = \Pi_i,\: q_j = q_i.$ Set $\lambda_i = \|q_{i-1} - q_i\|$ and let $\sigma_i$ be the symmetry with respect to the tangent plane $\alpha_i = T_{q_i}(\partial D)$.  (If $u = u_t + u_{n} \in {\mathbb S}^{d-1}$ with $u_t \in \alpha_i, u_n \perp \alpha_i$, then $\sigma_i(u) = u_t - u_{n}.$) Clearly,
 \[\sigma_i(\omega_i) = \omega_{i+1},\quad \sigma_i(\Pi_i') = \Pi_i, \quad \Pi_0 = \Pi_m.\]
 We identify  $\Pi_{i-1 }$ and $\Pi_i'$ by using a translation along the line determined by the segment $[q_{i-1}, q_i]$ and we will write $\sigma_i(\Pi_{i-1}) = \Pi_i.$ 
   
Given $(u, v) \in \Pi_{i -1}\times \Pi_{i-1}$ sufficiently close to $(0, 0)$, consider the line $\ell(u, v)$ passing through $q_{i-1} +  u$ and having direction $\omega_{i- 1} + v$ (the point $v$ is identified with the vector $v$). Then $\ell(u, v)$ intersects $\partial D$ at a point $p= p(u, v)$ close to $q_i.$ Let $\ell'(u, v)$ be the line symmetric to $\ell(u, v)$ with respect to the tangent plane to $\partial D$ at $p$ and let $q_i + u' \in \Pi_{i}$ be the intersection point of $\ell'(u, v)$ with $\Pi_{i}$. There exists an unique $v'\in \Pi_{i}$ for which $\omega_{i} + v'$ has the direction of $\ell'(u, v).$ Thus we get a map
\[\Psi_{i} :  \Pi_{i- 1} \times \Pi_{i -1}\ni (u, v) \longmapsto (u', v') \in \Pi_{i} \times \Pi_{i}\]
defined for $(u, v)$ in a small neighborhood of $(0, 0)$ (see Figure \ref{fig:Psi}).
\begin{figure} 
\includegraphics{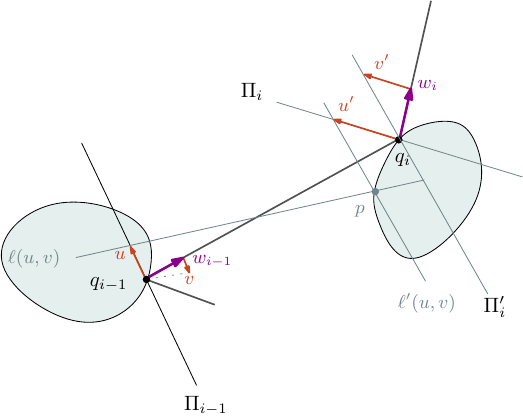}
\caption{The map $\Psi_i : (u,v) \mapsto (u',v')$}
\label{fig:Psi}
\end{figure}
The smoothness of the billiard ball map $\bf B$  implies the smoothness of $\Psi_i.$ 
Next consider the second fundamental form $S(\xi, \eta) = \langle G_i(\xi), \eta\rangle$ for $D$ at $q_i$, where
\[G_i = \dd n_j(q_i) : \alpha_i \longrightarrow \alpha_i\]
is the Gauss map.  Here recall that $n_j(q)$ is the inward unit normal vector to $\partial D_j$ at $q$ pointing into $D_j$. Introduce a symmetric linear map $\wt \psi_i$ on $\Pi_i$ defined for $\xi, \eta \in \Pi_i'$ by
\[\langle \wt \psi_i \sigma_i(\xi), \sigma_i(\eta)\rangle  = - 2 \langle\omega_{i-1}, n_j(q_i)\rangle \langle G_i(\pi_i(\xi)), \pi_i(\eta)\rangle,\]
where $\langle .  , . \rangle$ denotes the scalar product in $\R^d$, $\pi_i: \Pi_i' \longrightarrow \alpha_i$ is the projection on $\alpha_i$ along $\R \omega_{i-1}$.

Notice that the non-eclipse condition \eqref{eq:1.1} implies that there exists  $\beta_0 \in \left]0, \pi/2\right[$ depending only on $D$ such that for all incoming directions $\omega_{i- 1}$ and all reflection points $q_i \in \partial D_j$, one has
\[-\langle \omega_{i-1}, n_j(q_i)\rangle  = \langle \omega_i, n_j(q_i)\rangle \geqslant \cos  \beta_0 > 0.\]
Consequently, the symmetric map $\wt \psi_i$ has spectrum included in $[\mu_1, \mu_2]$ with $0 < \mu_1 < \mu_2$ depending only on $\kappa = \cos \beta_0$ and the sectional curvatures of $\partial D.$ Finally, define the symmetric map
\[\psi_i = s_i^{-1} \wt \psi_i s_i: \: \Pi_m \longrightarrow \Pi_m\]
with $s_i = \sigma_i \circ \sigma_{i-1} \circ \cdots \circ \sigma_1.$ 
By Theorem 2.3.1 in \cite{petkov2017geometry}, the map $\dd \Psi_i(0, 0)$ has the form
\[\dd\Psi_i(0, 0) = \begin{pmatrix} I  & \lambda_i I\\
\tilde{\psi}_i & I + \lambda_i \tilde{\psi}_i \end{pmatrix} \begin{pmatrix} \sigma_i & 0\\ 0 & \sigma_i \end{pmatrix},\]
and the linearized Poincar\'e map $P_{\gamma}$ related to $\gamma$ is given by
\[
P_{\gamma} = \dd (\Psi_m \circ \cdots \circ \Psi_1)(0,0):  \Pi_0 \times \Pi_0 \longrightarrow \Pi_0 \times \Pi_0,
\] 
which implies 
\[P_{\gamma} =  \begin{pmatrix} s_m  & 0 \\0 & s_m\end{pmatrix}  \begin{pmatrix} I && \lambda_m I\\ \psi_m && I + \lambda_m \psi_m \end{pmatrix}  \cdots \begin{pmatrix} I && \lambda_1 I\\ \psi_1 && I + \lambda_1 \psi_1 \end{pmatrix}.\]
Here the space $\Pi_0 \times \Pi_0$ is identified with the space $T_{\rho}(T^*(\Omega))/E_{\rho},$ where $\pi_x(\rho) = (q_0, \omega_0).$

Now we repeat without changes the argument of Proposition 2.3.2 in \cite{petkov2017geometry}.  For $k = 0, 1,\dots,m,$ consider the space  ${\mathcal M}_k$  of linear symmetric non-negative definite maps $M: \: \Pi_k \longrightarrow \Pi_k.$ Next, let ${\mathcal M}_k(\varepsilon) \subset {\mathcal M}_k$ be the space of maps such that $M \geqslant \varepsilon I$ with $\varepsilon > 0.$ To study the spectrum of $P_{\gamma}$, consider the subspace
\[L_0 = \{(u, M_0 u):\: u \in \Pi_0\}, \quad M_0 \in {\mathcal M}_0,\] 
which is Lagrangian with respect to the symplectic structure on $\Pi_0 \times \Pi_0$ induced from the symplectic structure on $T_{\rho}(T^*(\Omega))$ by the factorisation with $E_{\rho}.$
 The action of the map $\dd\Psi_1(0,0)$, transforms the Lagrangian space $L_0$ into 
 \[L_1 = \{ \sigma_1(I + \lambda_1 M_0) u , \sigma_1( ( I + \lambda_1 \psi_1) M_0 + \psi_1)u: \: u \in \Pi_0\} \subset \Pi_1 \times \Pi_1.\]
 Introducing the operator
 \[{\mathcal A}_i : {\mathcal M}_{i-1} \longrightarrow {\mathcal M}_i\]
 defined by 
 \[{\mathcal A}_i (M) = \sigma_iM ( I + \lambda_i M)^{-1}\sigma_i^{-1}  + \tilde{\psi}_i,\]
 we write $L_1 = \{(u, M_1 u):\: u \in \Pi_1\}$ with $M_1 = {\mathcal A}_1(M_0).$ By recurrence, define
 \[L_k = \{(u, M_k u):\: u \in \Pi_k\},\quad M_k = {\mathcal A}_k(M_{k-1}), \quad k = 1, 2,\dots,m.\]
 The maps ${\mathcal A}_k$ are contractions from ${\mathcal M}_{k-1}(\varepsilon)$ to ${\mathcal M}_k(\varepsilon)$, hence
 \[{\mathcal A} = {\mathcal A}_m \circ \cdots \circ {\mathcal A}_1\]
becomes  also a contraction from ${\mathcal M}_0(\varepsilon)$ to ${\mathcal M}_0(\varepsilon).$ We choose $M_0 \in {\mathcal M}_0(\varepsilon)$ as a {\it fixed point} of ${\mathcal A}$ and thus we fix $L_0$. Notice that $\varepsilon > 0$ can be chosen uniformly for all periodic rays. 
Thus we deduce
\[P_{\gamma} \begin{pmatrix} u \\ M_0 u \end{pmatrix} = \begin{pmatrix} Su \\ M_0 Su \end{pmatrix}\]
with a map $S: \Pi_0 \longrightarrow \Pi_0$ having the form
\[S = \sigma_m(I + \lambda_m {\mathcal A}'_{m-1} (M_0)) \circ \sigma_{m-1} ( I + \lambda_{m-1} {\mathcal A}'_{m-2} (M_0)) \circ \cdots\circ \sigma_1( I + \lambda_1 M_0),\]
where ${\mathcal A}'_k = {\mathcal A}_k\circ {\mathcal A}_{k-1} \circ \cdots \circ {\mathcal A}_1.$ Setting 
\[d_0 = \min_{i \neq j} {\rm dist}\: (D_i , D_j) > 0,\quad d_1 = \max_{i \neq j} {\rm dist}\: (D_i , D_j),\]
 and $\beta = \log ( 1 + \varepsilon d_0),$  one obtains
\[\|Su\| \geqslant ( 1 + d_0 \varepsilon)^m \|u\| = \e^{\beta m} \|u\|.\]
Obviously, the eigenvalues of $S$ are eigenvalues of $P_{\gamma}$ and we conclude that $P_{\gamma}$ has $(d- 1)$ eigenvalues $\nu_1,\dots,\nu_{d-1}$ satisfying
\[|\nu_j| \geqslant \e^{\beta m},\quad j = 1, \dots , d-1.\]

For  $0 < \tau< \lambda_1$, consider a point $\rho = \phi_{\tau}(z) \in \mathring{B} \cap \gamma,$  where $z = (x, v) \in \cal D_{\mathrm{in}}.$ The map $\phi_{\tau}: \cal D_{\mathrm{in}} \rightarrow \mathring{B}$ is smooth near $z$ and  moreover $\dd\phi_{\tau}(z): \Sigma_{z} \to \Sigma_{\rho}$.
We identify $\Pi_0 \times \Pi_0$ with $\Sigma_z$ and $\Sigma_{\phi_{\tau}(z)}$ with the image 
\[\dd \phi_{\tau}(z) \Sigma_z = \begin{pmatrix}  I & \tau I\\0 & I\end{pmatrix} (\Pi_0 \times \Pi_0).\] 
Next we  define the unstable subspace of $\Sigma_{\rho}$  as 
\[E_u(\rho) =\dd \phi_{\tau} (z)(L_0) = \begin{pmatrix}  I & \tau I\\0 & I\end{pmatrix} (L_0).\] 
Let  $0 < \sigma < \lambda_{p+1}$  with $p \geqslant 1$ and  set $t = -\tau + \sum_{j= 1}^p \lambda_j  + \sigma$. 
Then $\phi_{-\tau}$ is smooth near $\rho,$  the map ${\bf B}^p$  is smooth and
\[
\dd \phi_t(\rho)|_{\Sigma_\rho} = \dd \phi_\sigma(\mathbf{B}^p(z)) \circ \dd \mathbf{B}^p(z) \circ \dd \phi_{- \tau}(\rho) : \Sigma_{\rho} \to \Sigma_{\phi_t(\rho)}.
\]
This is illustrated by the diagram
\[ \begin{CD} 
@. E_u(\rho) @> \dd \phi_t(\rho) >>  E_u( \phi_t(\rho))\\
@ . @VV \dd\phi_{-\tau}(\rho)V         @AA \dd \phi_{\sigma} ({\bf B}^p(z))A @. \\
\Pi_0 @ > \chi_0 >>  L_0 @> \dd {\bf B}^p (z)>> L_p @< \chi_p<< \Pi_p,
\end{CD}\]
where $\chi_0: \Pi_0 \ni u \mapsto (u,M_0 u) \in L_0 \subset \Pi_0 \times \Pi_0$ and $\chi_p: \Pi_p \ni u \mapsto (u, M_p u) \in L_p \subset \Pi_p \times \Pi_p.$
It is easy to obtain an estimate of the action of $\dd \phi_t(\rho)\vert_{E_u(\rho)} $ for 
 $\rho = \phi_{\tau}(z),\:v = \dd \phi_{\tau}(z) (u, M_0 u) \in E_u(\rho)$. Clearly,
\[\dd \phi_{t}(\rho) \cdot v= (\dd \phi_{\sigma} ({\bf B}^p(z)) \circ \dd {\bf B}^p (z)) (u, M_0 u).\]
By the above argument we deduce
\[\dd {\bf B}^p (z) (u, M_0 u ) =  ( S_p  u, M_p S_p u) \in L_p\]
with
\[S_p =  \sigma_p(I + \lambda_p {\mathcal A}'_{p-1} (M_0)) \circ \sigma_{p-1} ( I + \lambda_{p-1} {\mathcal A}'_{p-2} (M_0)) \circ \cdots \circ \sigma_1( I + \lambda_1 M_0).\]
Setting $\beta_0 = \beta / d_1$ and $w = (u, M_0 u) = \dd \phi_{-\tau} (\rho)  \cdot v,$ we have
 \begin{equation}
 \label{eq:A1}
\|\dd {\bf B}^p(z) \cdot w\| = \|(S_p u, M_p S_pu)\| \geqslant \|S_p u\| \geqslant \e^{\frac{\beta}{d_1} p d_1} \|u\| \geqslant \e^{\beta_0 (t + \tau - \sigma)} \|u\|,
\end{equation}
 and
\begin{equation} \label{eq:A1bis}
\|\dd {\bf B}^p(z) \cdot w\| \leqslant C_0 \e^{-\beta_0 d_1}e^{\beta_0 t} \|w\| = C_0 \e^{-\beta_0 d_1}e^{\beta_0 t} \|\dd \phi_{-\tau}(\rho) v\|.
\end{equation}
Here we used the estimate
 \[\|w\| = \Bigl( \|u\|^2 + \|M_0 u\|^2\Bigr)^{1/2} \leqslant (1 + B_0^2)^{ 1/2}\|u\|\]
 with $\|M_0\|_{\Pi_0 \to \Pi_0} \leqslant B_0$ and we set $C_0 = (1 + B_0^2)^{-1/2}.$ The constant $B_0$ can be chosen uniformly for all $M_k$ and all periodic points since for every non-negative symmetric map $M$ one has
\[\|M(I + \lambda_k M)^{-1} \| \leqslant\frac{1}{\lambda_k} \leqslant \frac{1}{d_0},\]
while the norms $\|\wt \psi_k\|$ are uniformly bounded by a constant depending of the sectional curvatures of $D$ and $\kappa > 0$. Consequently,
\begin{equation} \label{eq:A2} 
\|{\mathcal A}_k(M) \| \leqslant B_0,
\end{equation} 
 the same is true for the norm of the fixed point $M_0 = {\mathcal A}_m(M_{m-1})$
and the  estimate (\ref{eq:A2}) is uniform for all periodic points. Finally, estimating the norm of $\dd \phi_{-\sigma}({\bf B}^p(z))  = \begin{pmatrix} I & -\sigma I\\0 & I \end{pmatrix}$, we obtain 
$\| \dd \phi_{-\sigma} ({\bf B}^p (z) ) \zeta\| \geqslant ( 1 + d_1)^{-1} \|\zeta\|$ and 
 \begin{equation}\label{eq:A3}
 \begin{aligned}
 \|\dd \phi_{t}(\rho) v\| &\geqslant  ( 1 + d_1)^{-1}  C_0e^{- \beta_ 0 d_1} \e^{\beta_0 t} \|\dd \phi_{-\tau}(\rho) v\| \nonumber \\ &\geqslant (1 + d_1) ^{-2} C_0e^{-\beta_0 d_1} \e^{\beta_0 t}\|v\|.
 \end{aligned}
 \end{equation}  
It remains to treat the case $\rho = \phi_{\tau}(z),\: z \in \cal D_{in},\: 0 < t = \tau + \sigma < \lambda_1.$ Then  $\phi_t(\rho) = \phi_{\tau + \sigma}(z) \in \mathring{B} \cap \gamma$ and we obtain easily an estimate for $\|(\dd \phi_{\tau + \sigma}(z)) \cdot v\|$. 

 Our case is a partial one of a more general setting (see \cite{Liverani1994sympl}) concerning Lagrangian spaces $\{(u, Mu)\}$ with positive definite linear maps $M$. Such spaces are called positive Lagrangian. A linear symplectic map $L$ is called monotone if it maps positive Lagrangian onto positive Lagrangian. In \cite{Liverani1994sympl} it is proved that any monotone symplectic map is a contraction on the manifold of positive Lagrangian spaces. After a suitable conjugation the map $L$ has the representation (see Proposition 3 in \cite{Liverani1994sympl}) 
 \[L = \begin{pmatrix} A^{-1} & 0 \\ 0 & A^*\end{pmatrix}  \begin{pmatrix} I & & R\\ P & & I + PR\end{pmatrix}\]
 with positive definite matrices $P, \: R$. In our situation we have $A = I, R = \lambda_i I, \: P = \psi_i.$
 
 To determine the stable space $E_s(z)$ at $z$,  we will study the flow $\phi_t$ for $t < 0$ and repeat the above argument leading to a fixed point. For completeness we present some details.
      The linear map $ P_{\gamma}^{-1} $ for a periodic ray $\gamma$ with $m$ reflections has the representation 
      \[P_{\gamma}^{-1} = (\dd \Psi_1)^{-1} \circ \cdots \circ (\dd \Psi_m)^{-1} : \: \Pi_0 \times \Pi_0 \longrightarrow \Pi_0 \times \Pi_0,\]
where 
\[(\dd \Psi_k)^{-1}= \begin{pmatrix} \sigma_k^{-1} & 0\\
0 & \sigma_k^{-1} \end{pmatrix} \begin{pmatrix} I + \lambda_k \psi_k & & - \lambda_k I\\
- \psi_k & & I \end{pmatrix}.\]     
 Recall that $\Pi_0 = \Pi_m.$ Consider a Lagrangian $Q_0 = Q_m  = \{(u, - N_m u):\: u \in \Pi_0\}$ with a symmetric non-negative definite map $N_m \in {\mathcal M}_0.$  Then
 \[
 \begin{aligned}
( \dd \Psi_m)^{-1} Q_{m} &= \left\{ \left(\sigma_m^{-1} (I + \lambda_m (\psi_m + N_m) ) u, - \sigma_m^{-1} (\psi_m + N_m) u\right)~:~ u \in \Pi_{0}\right\} \\
&= \{( u, -N_{m-1} u): \: u \in \Pi_{m-1}\},
\end{aligned}
\]
 where
\[ N_{m-1} = \sigma_m^{-1} (\psi_m + N_m) \Bigl( I + \lambda_m(\psi_m + N_m)\Bigr)^{-1}\sigma_m:\: \Pi_{m-1} \longrightarrow \Pi_{m-1}.\]     
By recurrence, introduce the Lagrangian spaces 
\[Q_k = \{(u, - N_k u):  u \in \Pi_k\}, \quad N_k =  \bc_k(N_{k+1}),\quad k = 0, \dots,m-1,\] 
where
\[\bc_k(M) = \sigma_{k + 1}^{-1} (\psi_{k+1} + M) \Bigl( I + \lambda_{k+1} (\psi_{k+1} + M) \Bigr)^{-1}\sigma_{k+1}:\: \Pi_k \longrightarrow \Pi_k.\]
It is easy to see that $\bc_{k}$ are contractions from ${\mathcal M}_{k+1}(\ep)$ to ${\mathcal M}_k(\ep)$ since
\[
\begin{aligned}
&\sigma_{k+1} \Bigl(\bc_k(M_1) - \bc_k(M_2) \Bigr) \sigma_{k+1}^{-1} \\
&\qquad \qquad = (I + \lambda_{k+1} (\psi_{k+1} + M_1))^{-1} (M_ 1- M_2)(I + \lambda_{k+1} (\psi_{k+1} + M_2))^{-1}.
\end{aligned}
\]
Therefore, $\bc = \bc_0 \circ \cdots \circ \bc_{m-1}$ will be contraction from ${\mathcal M}_0(\ep)$ to ${\mathcal M}_0(\ep)$ and there exists a {\it fixed point} $N_m \in {\mathcal M}_0(\ep)$ of $\bc.$ Moreover,
\[P_{\gamma}^{-1} \begin{pmatrix} u \\ -N_m u\end{pmatrix} = \begin{pmatrix} \wt{S} u\\ - N_m \wt{S} u\end{pmatrix},\quad u \in \Pi_0,\]
where
\[
\begin{aligned}
\wt{S} = &\sigma_1^{-1} ( I + \lambda_1 (\psi_1 + \bc_1'(N_m)) )\circ \sigma_2^{-1} ( I + \lambda_2(\psi_2 +\bc_2'(N_m))) \\
& \hspace{7cm} \circ \cdots \circ \sigma_m^{-1} (I + \lambda_m (\psi_m + N_m) )
\end{aligned}
\]
and $\bc_k' = \bc_k\circ \cdots \circ \bc_{m-1},\: k = 1,\dots,m-1.$ Clearly,
\[\|\wt{S} u \|\geqslant  (1 + d_0 \ep)^m \|u\|,\: u \in \Pi_0,\]
where $\ep > 0$ depends of the sectional curvatures of $D.$ Thus the stable manifold at $\phi_{\sigma} (z), -\lambda_{m-1} < \sigma < 0$ can be defined as $E_s(\phi_\sigma (z)) = \dd \phi_{\sigma} (z) (Q_m)$ and we may repeat the above argument for the estimate of $\dd\phi_{t}(\phi_{\sigma}(z))$ acting on $E_s(\phi_{\sigma}(z))$ for $t < 0.$ 

The intersection of the unstable and stable manifolds at $y = \phi_t(z),\: 0 < t < \lambda_p$ is $ (0, 0)$. Indeed, we have
\[E_u(y ) = \dd \phi_t (z)(L_{p-1}),\quad E_s(y) = \dd \phi_{t - \lambda_p}(\phi_{\lambda_p}(z)) (Q_{p}),\]
where 
$
L_{p-1} = \{(u, M_{p-1}u):\: u \in \Pi_{p-1} \times \Pi_{p-1}\}$ and  $Q_p = \{ (- u, - N_{p} u):\: u \in \Pi_p \times \Pi_p\}.
$
Assume that $E_u(y) \cap E_s(y) \neq (0, 0).$ Then there exists $0 \neq v \in L_{p-1} \cap \dd \phi_{-\lambda_p} (\phi_{\lambda_p}(z))(Q_p).$
 By the above argument
$\dd \phi_{-\lambda_p} (\phi_{\lambda_p}(z))(Q_p) = \{(u, - N_{p-1} u) : u \in \Pi_{p-1} \times \Pi_{p-1}\}.$ This implies the existence of $u \neq 0$ for which
$(M_{p-1} + N_{p-1}) u = 0$ which is impossible since $M_{p-1} + N_{p-1}$ is a definite positive map.
Consequently, $E_u(y)$ and $E_s(y)$ are transversal subspaces of dimension $d-1$ of $\Sigma_{y}$ and we have a direct sum
$\Sigma_{y} = E_u(y) \oplus E_s(y).$
 
     Now we pass to the  estimates of $\dd \phi_t(z)\vert_{E_u(z)}$, where $z \in \mathring{B} \cap K_e$ is not a periodic point. Since $z \in K_e $, the trajectory $\gamma = \{\phi_t (z): t \in \R\}$ has infinite number successive reflection points $q_k \in \partial D_{i_k}, \: k \in \Z,$ with an infinite sequence
     \[J_0 = (i_j)_{j \in \Z},\quad i_j \neq i_{j+ 1} .\]
     For every $p \geqslant p_0 \gg 1$ define the configuration
     \[
     \alpha_p = \left\lbrace \begin{array}{ll} (i_{-p}, \dots,i_0, \dots, i_p) &\text { if } i_p \neq i_{-p}, \\ (i_{-p}, \dots, i_0, \dots, i_{p+1}) &\text { if } i_p = i_{-p}.\end{array} \right.
     \]
 Repeating $\alpha_p$ infinite times, one obtains an infinite configuration. Following the arguments of the proof of Proposition 10.3.2 in \cite{petkov2017geometry}, there exists a periodic ray $\gamma_p$ following this configuration and we obtain a sequence of periodic rays $(\gamma_{p_0 + k})_{k \geqslant 0}$.  
  Let 
 $\{q_{p, k} \in \partial D_{i_k}\}$  be the reflexion points of $\gamma_p$. For the periodic ray $\gamma_p$ passing through $q_{p, 0} \in \partial D_{i_0}$ consider the linear space 
  \[L_{p, 0} = \{(u, M_{p, 0} u):\: u \in \Pi_{p,0}\} \subset \Pi_{p, 0} \times \Pi_{p, 0}.\]
    Our purpose is to show that the symmetric linear maps $M_{p, 0} \in {\mathcal M}_{p,0} (\varepsilon)$ composed by some unitary maps converge as $ p \to \infty$ to a symmetric linear map $\wt {M}_0 \in {\mathcal M}_{0}(\varepsilon)$ on $\Pi_0$. This composition is necessary since the maps $M_{p, 0}$,~ $p \geqslant p_0,$ are defined on different spaces.  To do this, we will use Lemmas 10.2.1, 10.4.1 and 10.4.2 in \cite{petkov2017geometry}. Consider the rays $\gamma_{p_0 + q},\: q \geqslant 1,$ and $\gamma$. These rays have reflection points passing successively through the obstacles
    \[L'= D_{i_{-p_0 - 1}}, D_{i_{-p_0}},\dots, D_{i_0},\dots,D_{i_{p_0}}, D_{i_{p_0 + 1}} = L''.\]
 According to Lemma 10.2.1 in \cite{petkov2017geometry}, there exist  
     uniform constants $C > 0$ and $\delta \in (0, 1)$ such that for any $|k|\leqslant p_0$ and $j = 1, \dots, q$, one has
    \[
    \begin{aligned}
    \|q_{p_0 + 1, k} - q_{p_0 + j,k}\| &\leqslant C(\delta^{p_0 + k}  + \delta^{p_0 - k}), \quad 
    \|q_{p_0 + j, k} - q_k\| &\leqslant C(\delta^{p_0 + k}  + \delta^{p_0 - k}).
    \end{aligned}
    \]
We need to introduce some notations from \cite[Section 10.4]{petkov2017geometry}. Let $x \in \partial D_i$ and $y \in \partial D_j$ with $i \neq j,$ and assume that the segment $[x, y]$ is transversal to both $\partial D_i$ and $\partial D_j$. Let $\Pi$ be the plane orthogonal to $[x, y]$, passing through $x$. Let $ e = (x-y)/\|x - y\|,$ and introduce the projection $\pi: \Pi \longrightarrow T_x(\partial D)$ along the vector $e.$ As above, we define the symmetric linear map $\wt \psi : \Pi \rightarrow \Pi$ by
\[\langle \wt \psi (u), u \rangle = 2 \langle e, n(x)\rangle \langle G_x (\pi (u)), \pi(u)\rangle,\quad u \in \Pi,\]
and notice that
\[{\rm spec}\: \wt \psi \subset [ \mu_1,  \mu_2 ],\quad 0 < \mu_1 < \mu_2.\]
Setting $D_0 = 2 C$, we have the estimates
\[ \|q_{p_0 + j, k} - q_k\| \leqslant D_0 \delta^{p_0 +k},\quad k = -p_0+ 1,\dots, 0, \quad j = 1,\dots,q.\]
Fix $1 \leqslant j \leqslant q$ and introduce the vectors
\[ e_k =\frac{q_{k+1} - q_k}{\|q_{k+1} - q_k\|} ,\quad e_k'= \frac{q_{p_0 + j, k + 1}- q_{p_0 + j, k}}{\|q_{p_0 + j, k + 1}- q_{p_0 + j, k}\|}.\]
Consider the maps $\wt \psi_k : \Pi_k \longrightarrow \Pi_k$ and $\wt \psi'_{ k}:\Pi_{k}' \longrightarrow \Pi_k'$
related to the segments $[q_{k-1}, q_k]$ and $[q_{p_0 + j, k-1}, q_{p_0 +j, k}],$ respectively. Let $M_{-p_0 + 1}: \Pi_{-p_0 + 1} \lr \Pi_{-p_0 + 1}$ and $M_{-p_0 + j}': \Pi_{-p_0 + j}' \lr \Pi_{-p_0 + j}'$ be symmetric non-negative definite linear operators. By recurrence, define
\[M_k = \sigma_k M_{k-1} ( I + \lambda_k M_{k-1})^{-1} \sigma_k + \wt \psi_k, \quad k = -p_0+2,\dots,0,\]
where $\lambda_k = \|q_{k-1} - q_k\|$ and $\sigma_k$ is the symmetry with respect to $T_{q_k}(\partial D).$ Similarly, we define $M_k', \: k =-p_0 + 2,\dots,0,$ replacing $\wt \psi_k, \lambda_k$ and $\sigma_k$ by $\wt \psi_{k}', \:\lambda_{p_0 +j, k} = \|q_{p_0 + j, k-1} - q_{p_0 + j, k}\|$ and $\sigma_k'$, respectively. Next, introduce the constants
\[b = ( 1 + 2 \mu_1 \kappa d_0)^{-1} < 1, \quad a_1 = \max\{\delta, b\} < 1,\]
where $d_0 > 0$ and $\kappa > 0$ were defined above. 
We choose $M_{-p_0 + 1}$ so that $\|M_{-p_0 + 1} \| \leqslant B_0$ and by induction one deduces $\|M_k\| \leqslant B_0.$ Here $B_0 > 0$ is the constant in  (\ref{eq:A2}). We have uniform estimates
 \begin{equation} \label{eq:A4}
\|M_k\| \leqslant B_0, \quad  \|M_k'\| \leqslant B_0, \quad k = -p_0 + 1,\dots,0.
\end{equation} 
 Applying \cite[Lemma 10.4.1]{petkov2017geometry}, there exists a linear isometry $A_k: \R^d \to \R^d$ such that $A_k(\Pi_k') = \Pi_k$, and $A_k$ satisfies the estimates
 \begin{equation} \label{eq:A5}
\|A_k - I\| \leqslant C_1 D_0( 1+ \delta) \delta^k, \quad
\|\wt \psi_k- A_k \wt \psi_k'A_k^{-1}\| \leqslant C_2 D_0( 1+ \delta)\delta^k,
\end{equation}
for any $k = -p_0 + 1,\dots,0$. Now we are in position to apply \cite[Lemma 10.4.2]{petkov2017geometry} saying that with some constant $E > 0,$ depending only of $D, \: \kappa,\: \delta$ and $b,$ for $k = -p_0 + 1,\dots, 0$ we have
\begin{equation} \label{eq:A6}
\|M_k - A_k M_{k}'A_k^{-1} \| \leqslant D_0 E a_1^{p_0 + k} + b^{2(k+ p_0 - 1)} \|M_{-p_0 + 1} - A_{-p_0 + 1}  M_{-p_0 + 1}'A_{-p_0 + 1}^{-1} \|.
\end{equation}
The norm of the second term on the right hand side is bounded by $2B_0 b^{2(k + p_0 -1)}$ and for $k = 0$ we obtain
\[\|M_0 - A_0 M_0'A_0^{-1} \| \leqslant D_0 E a_1^{p_0} + 2 B_0 b^{2(p_0 - 1)}.\]
Applying the above estimate for the rays $\gamma_{p_0 + q}$, the maps $M_0', A_0$ will depend of the ray $\gamma_{p_0 + q}$ and for this reason we denote them by $M_{q, 0}', A_{q, 0}.$
Now we use these estimates for the maps $M_{q, 0}', \: M_{q', 0}$ related to the rays $\gamma_{p_0 + q}$ and $\gamma_{p_0 + q'}$ and by the triangle inequality one deduces
\begin{equation} \label{eq:A7}
\big\|A_{q,0} M_{q,0}'A_{q,0}^{-1}  - A_{q', 0}M_{q', 0}' A_{q', 0}^{-1}\big\| \leqslant 2 D_0 E a_1^{p_0} + 4 B_0 b^{2(p_0 - 1)}.
\end{equation} 
Here $A_{q,0}(\Pi_{q,0}') = \Pi_0$
and $A_{q', 0}(\Pi_{q', 0}') = \Pi_0$ are some isometries satisfying the estimates (\ref{eq:A5}). Clearly, one obtain a Cauchy sequence
$(A_{q, 0} M_{q, 0}' A_{q, 0}^{-1}) _{q \geqslant 1} $ which converges to a symmetric non-negative linear map $\wt{M}_0$ in $\Pi_0.$ Moreover, if for every $q$ we have $M_{q, 0}' \geqslant \ep I$, then $\wt{M}_0 \geqslant \ep I.$

 After this preparation we define the unstable manifold at $\rho = \phi_{\tau} (z_0),\: z_0 = (q_0, v_0)$ for some $0 < \tau < \|q_1 - q_0\|$ as the subspace 
\[E_u(\rho) = \dd \phi_\tau(z_0)  \{(u, \wt{M}_0 u) \in \Pi_0 \times \Pi_0:\: u \in \Pi_0\} \subset \Sigma_{\rho}.\]
 It is important to note that the procedure leading to the estimate (\ref{eq:A6}) can be repeated starting with $\wt{M}_0$ instead of $M_{-p_0 + 1}.$  Then if $\wt{M}_{k}$ are the maps obtained from $\wt{M}_0$ after successive reflections, we obtain an estimate
\begin{equation*}
\|\wt{M} _k - A_k M_k'A_k^{-1} \| \leqslant D_0 E a_1^{p_0 + k} + b^{2(k+ p_0 - 1)} \|\wt{M}_0 - A_0\wt{M}_0'A_0^{-1} \|
\end{equation*}
for $k = 1,\dots, p_0/2.$

We can repeat the above argument for $v \in E_u(\rho)$ and $t = - \tau + \sum_{j= 1}^p \lambda_j + \sigma$, where $0 < \tau < \lambda_1$ and $0 < \sigma < \lambda_{p+1}$, to estimate  
\[\|\dd \phi_{t}(\rho) \cdot v\| .\]
We apply (\ref{eq:A1}) and (\ref{eq:A1bis}) with the expansion map $\wt{S}_p$ defined as the composition of the maps $( I + \lambda_{k} {\mathcal A}_{k- 1} '(\wt{M}_0))$ and we get an estimate for    $\|\dd \phi_{t}(\rho) \cdot v\| .$ Finally, the construction of the stable space $E_s(\phi_\sigma(z_0)),\: -\|q_{-1} - q_0\| < \sigma < 0$ can be obtained by a similar argument and we omit the details.
       
 \section{Ikawa's criterion and proof of Theorem 3}\label{appendix:b}
 \def\j{{\bf j}}   
In this appendix we prove Theorem 3 for all dimensions $d \geqslant 2.$  The result of Ikawa \cite[Theorem 2.1]{ikawa1990poles} was established for $d$ odd and it yields only an infinite number of resonances in a suitable band. To obtain a stronger result we apply the argument of \cite{zworski2017local}. The proof is based on Lemma 2.2, Proposition 2.3 and Theorem 2.4 in \cite{ikawa1990poles}.  Recall the notation
\[
\Lambda_{\omega} = \{ \mu_j \in \C \setminus e^{i\frac{\pi}{2} \overline{\R^{+}}}: \: 0 < \Im \mu_j \leqslant \omega |\Re \mu_j|, \: 0 < \arg \mu_j < \pi\}.
\] {For the modification covering all dimensions $d \geqslant 2$, it is necessary only to modify Lemma 2.2 since the other results are independent of  the dimension $d$.  Below we consider only the  resonances $\mu_j$ for which $0 <  {\rm arg}\: \mu_j < \pi$ and we omit this in the notation. 

Let $\rho \in C_c^{\infty}(\R; \R_+)$ be an even function with $\supp \rho \subset \left[-1, 1\right]$ such that
\[
\rho(t) > 1 \quad \text{if} \quad |t| \leqslant 1/2,
\]
and the property that its Fourier transform is non-negative,
\[
\hat{\rho}(k) = \int \e^{i t k} \rho(t) \dd t \geqslant 0, \quad k \in \R.
\]
(As in \cite{ikawa1990poles}, we use the above Fourier transform, since we deal with $\langle e^{i \mu_j t}, \rho\rangle$). It is easy to construct $\rho$ with the above properties. Let $\phi \in C_c^{\infty}(\R; [0, 1])$ be an even function with support in $[-1/2, 1/2]$ such that $\phi(x)\equiv 1$ for $|x| \leqslant 3/8.$ Define
\[\Phi(t): = (\phi \star \phi )(t) = \int_{-\infty}^{\infty} \phi(x) \phi(t- x) \dd x \geqslant 0.\]
Clearly, $\Phi(t)$ is even, has support in $[-1, 1]$ and $\hat{\Phi}(k) = (\hat{\phi}(k))^2.$  For $k \in \R$ the function $\hat{\phi}(k)$ is real valued and $\hat{\Phi}(k) \geqslant 0$ for $k \in \R.$ On the other hand, for $|t| \leqslant 1/2$ we have
\[\Phi(t) \geqslant \int_{-1/4}^{1/4} \phi(t- x) \dd x= \int_{t - 1/4}^{t + 1/4} \phi(s) \dd s \geqslant {\rm mes}\:([t - 1/4, t+ 1/4] \cap [-3/8, 3/8]) \geqslant \frac{1}{8}\]
 and we may take $\rho(t) = 9\Phi(t)$.

Let $(\ell_q)_{q \in \N}$ and $(m_q)_{q \in \N}$ be sequences of positive numbers such that $\ell_q \geqslant d_0 = \min_{k \neq j} {\rm dist}\: (D_k, D_j)> 0, \: m_q \geqslant \max\{1, \frac{1}{d_0}\}$ and let $\ell_q, m_q\to \infty$ as $q \to \infty$. Finally, set
\[
\rho_q(t) := \rho(m_q (t- \ell_q)), \quad t \in \R,
\]
and  $c_0 := \int\rho(t) \dd t \geqslant 1.$ The result \cite[Lemma 2.2]{ikawa1990poles} must be modified as follows.
\begin{lemm} Let $0 < \delta < 1$ be fixed. Assume that for $\alpha \geqslant 1$we have
\[N(\alpha) = \# \{\mu_j \in \Lambda_{\omega}: \: 0 < \Im \mu_j \leqslant \alpha,\: |\mu_j | \leqslant r \} \leqslant P(\alpha, \delta)r^{\delta}\]
with $P(\alpha, \delta) < \infty.$
Then we have
\begin{equation} \label{eq:B1}
 \sum_{ \mu_j \in \Lambda_{\omega}}|\hat{\rho_q}(\mu_j)| \leqslant C_0e^{\alpha}  m_q^{d + 1} \e^{-\alpha \ell_q} + C_1P(\alpha, \delta) m_q^{-(1- \delta)} 
\end{equation}
for some constants $C_0 > 0, C_ 1 > 0$ independent of $\alpha$ and $q$.
\end{lemm} 
\begin{proof} We write 
$$\sum_{\mu_j \in \Lambda_{\omega}} = \sum_{\substack{\mu_j \in \Lambda_{\omega} \\ \Im \mu_j > \alpha}} + \sum_{\substack{\mu_j \in \Lambda_{\omega} \\ \Im \mu_j \leqslant \alpha}} = (\textrm{I})+ (\textrm{II}).$$
For $(\textrm{I})$ one integrates by parts,
\begin{equation}\label{eq:Iipp}
\int_{\ell_q	-\frac{1}{m_q}}^{\ell_q + \frac{1}{m_q }} \rho(m_q(t- \ell_q)) \e^{it \mu_j} \dd t = \frac{(-1)^{d+2} m_q^{d+2} }{(i\mu_j)^{d + 2}}  \int_{\ell_q -\frac{1}{m_q}}^{\ell_q + \frac{1}{m_q }} \rho^{(d+ 2)} (m_q (t - \ell_q)) \e^{it \mu_j} \dd t.
\end{equation}
We have $|e^{it\mu_j} |\leqslant \e^{- t\Im\mu_j }$ and since $\supp \rho_q \subset [\ell_q - m_q^{-1}, \ell_q + m_q^{-1}]$ and $\Im \mu_j > \alpha$, we get
 \[|e^{it\mu_j} | \leqslant \e^{-\alpha(\ell_q - m_q^{-1})} \leqslant\e^{-\alpha (\ell_q - 1)}.\]
In particular, the right hand side of \eqref{eq:Iipp} is estimated by  
\[Ce^{\alpha }  \frac{e^{-\alpha \ell_q} m_q^{d+1}} {|\mu_j|^{d+2}} \|\rho\|_{C^{d+2}(\R)}\]
with a constant $C > 0$ independent of $j$ and $q$. On the other hand, for $d$ even by the results of Vodev \cite{vodev1994}, \cite{vodev1994even} we have the estimate
\[\# \{ \mu_j: 0 \leqslant \arg \mu_j \leqslant \pi, \:|\mu_j| \leqslant k \}\leqslant C_2 k^{d}.\]
and for $d$ odd we have the same bound (see {Section 4.3} in \cite{dyatlov2019mathematical}).
  Consequently, the series
\[\sum_{|\mu_j| \geqslant 1} \frac{1}{|\mu_j|^{d+2}} = \sum_{k= 1}^{\infty} \sum_{k \leqslant |\mu_j| < k+1}  \frac{1}{|\mu_j|^{d+2}} \leqslant C_2 \sum_{k = 1}^{\infty} \frac{(k+ 1)^d}{k^{d+2}}  \leqslant C_3\]
is convergent. This yields the first term on the right hand side of (\ref{eq:B1}). Passing to the estimate of $($II$),$ we apply the argument of the proof of Theorem 2 in \cite{zworski2017local}. First,
\[\int e^{i \zeta t} \rho_q(t) dt = m_q^{-1} e^{i \zeta \ell_q}\hat{\rho}(\zeta m_q^{-1}).\]
Applying the Paley-Winner theorem for $\Im \zeta \geqslant 0$ and $N \geqslant 2,$ one deduces
\[\big|\int e^{i \zeta t} \rho_q(t) dt\big| \leqslant C_N m_q^{-1} e^{- \Im \zeta (\ell_q - m_q^{-1})} (1 + |\zeta m_q^{-1}|)^{-N} \leqslant C_N m_q^{-1} ( 1 + |\zeta m_q^{-1})^{-N}.\]
Therefore
\[\big |\sum_{\Im \mu_j > \alpha} \langle e^{i \mu_j t}, \rho_q(t)\rangle\big | \leqslant C_N m_q^{-1} \int_0^{\infty} (1 + m_q^{-1}  r)^{-N} d N_{\alpha}(r) \]
\[\leqslant - C_N m_q^{-1} \int_{0}^{\infty} \frac{d}{dr} \Bigl( (1 + m_q^{-1}  r)^{-N}\Bigr) N_{\alpha}(r)dr \leqslant B_N P(\alpha, \delta) m_q^{-1 + \delta} \int_0^{\infty} ( 1 + y)^{-N- 1}y^{\delta} dy \]
\[= A_N P(\alpha, \delta) m_q^{-(1- \delta)}.\]

Notice that the other terms in the trace formula of Zworski \eqref{eq:tracezworski} are easily estimated. In fact, since $\lambda \mapsto \psi(\lambda)$ has compact support, one gets
\[
\begin{aligned}
\left|\int \Bigl(\int\psi(\lambda) \frac{\dd\sigma}{\dd\lambda}(\lambda) \cos(\lambda t) \dd\lambda\Bigr) \rho(m_q(t - \ell_q)) \dd t \right| 
&\leqslant C_{\psi}  \int\rho(m_q(t - \ell_q)) \dd t \\
&  \leqslant C_{\psi} c_0 m_q^{-1}. 
\end{aligned}
\]
Here we integrate by parts in the integral with respect to $\lambda$ and exploit the fact that $\sigma(\lambda)$ is bounded on the support of $\psi(\lambda)$ (see Section 3.10 in \cite{dyatlov2019mathematical} for the estimates of $\sigma(\lambda)$).
Similarly,
  \[\left| \int v_{\omega, \psi} (t) \rho(m_q(t - \ell_q)) \dd t \right| \leqslant C_{\omega, \psi}  \int \rho(m_q(t - \ell_q)) \dd t \leqslant c_0 C_{\omega, \psi}m_q^{-1}. \]
   We can put the estimates of these terms in $C_1 P(\alpha, \delta)m_q^{-(1- \delta)}$ increasing the constant $C_1$.
    This completes the proof.
 \end{proof}

Define the distribution $\hat F_\mathrm D \in  {\mathcal S}'(\R^+)$ by
\begin{equation} \label{eq:B2}
\hat{F}_\mathrm D(t) = \sum_{\gamma \in \mathcal P} \frac{(-1)^{m(\gamma)} \tau^{\sharp}(\gamma)  \delta(t - \tau(\gamma))}{|\det(I - P_{\gamma})|^{1/2}}.
\end{equation}. 
As we mentioned above, the following results are proved in \cite{ikawa1990poles}  and their proofs are independent of the dimension $d.$ For convenience of the reader we present the statements.

\begin{prop}[Prop. 2.3, \cite{ikawa1990poles}]\label{prop:2.3ika} Suppose that the function $s \mapsto \eta_\mathrm D(s)$ cannot be prolonged as an entire function of $s$. Then there exists $\alpha_0 > 0$ such that for any $\beta > \alpha_0$ we can find sequences  $(\ell_q), (m_q)$ with $\ell_q \to \infty$ as $q \to \infty$ and such that for all $q \geqslant 0$ one has
\[
e^{\beta \ell_q} \leqslant m_q \leqslant \e^{2 \beta \ell_q}
\quad \text{and} \quad
|\langle \hat{F}_\mathrm D, \rho_q \rangle | \geqslant \e^{- \alpha_0 \ell_q}.
\]
\end{prop} 
\begin{theo}[Theorem 2.4, \cite{ikawa1990poles}]\label{thm:5} There exist constants  $C > 0, \alpha_1 > 0$ such that for any sequences $(\ell_q)$ and $(m_q)$ with $\ell_q \to \infty$ as $q \to \infty$, it holds
\begin{equation} \label{eq:B4}
|\langle u, \rho_q\rangle | \geqslant |\langle \hat F_\mathrm D, \rho_q \rangle| - C \e^{\alpha_1 \ell_q} m_q^{-1}. 
\end{equation}
\end{theo}
\begin{rem} In \cite[Theorem 2.4]{ikawa1990poles}, on the right hand side of (\ref{eq:B4}), one has the term $m_q^{-\epsilon}$ for some $\epsilon >0$ instead of $m_q^{-1}$. In particular the above estimate holds, increasing $\beta >\alpha_0.$
\end{rem} 
The above theorem is given in \cite{ikawa1990poles} without proof. However its proof repeats that of Proposition 2.2 in \cite{ikawa1988analysis} following the procedure described in \cite[\S3]{ikawa1985trap} and exploiting the construction of asymptotic solutions in \cite{ikawa1988fourier}. The first term on the right hand side of (\ref{eq:B4}) is obtained by the leading term in (\ref{eq:sing}) applying the stationary phase argument to a trace of a global parametrix (see Chapter 4 in \cite{petkov2017geometry}) or to the trace of the asymptotic solutions given below.
For the second one we must estimate
a sum 
\[\sum_{\substack{\gamma \in {\mathcal P} \\ \tau(\gamma) \leqslant \ell_q + m_q^{-1}}}  \int_{\ell_q- m_q^{-1}}^{\ell_q + m_q^{-1} } \rho_q(t) r_{\gamma}(t) \dd t,\]
where $r_{\gamma}$ is a function in $L^1_{\mathrm{loc}}(\R)$, which is obtained from the lower order terms in the application of the stationary phase argument. Since $r_{\gamma}(t)$ could increase as $t \to \infty$, we need a precise analysis of the behavior of  $r_{\gamma}(t)$.

We discuss briefly the approach of Ikawa and refer to \cite{ikawa1985trap}, \cite{ikawa1988fourier} for more details. First one expresses the distribution $u(t)$ defined in Introduction by the kernels $E(t, x, y),\: E_0(t, x, y)$ of the operators $\cos (t\sqrt{- \Delta}) \oplus 0$ and $\cos(t\sqrt{-\Delta_0})$, respectively (recall that  $-\Delta$ is the Laplacian in $Q = \R^d \setminus D$ with Dirichlet boundary conditions on $\partial D$).  Consider
\[
\hat E(t,x,y) =
\left\{
\begin{array}{ccc}
E(t,x,y) & \text{if} & (x,y) \in Q \times Q, \\
0 & \text{if} & (x,y) \notin Q \times Q.
\end{array}
\right.
\]
If $D \subset \{x: \: |x| \leqslant a_0\},$ then
 \[\supp_{x, y} \Bigl(\hat{E}(t, x, y) - E_0(t, x, y) \Bigr) \subset \bigl\{(x, y) \in \R^d \times \R^d:\: |x| \leqslant a_0 + t, \: |y| \leqslant a_0 + t\bigr\}.\]
 For $t \in \supp \rho_q$ we must study the trace
 \[\int_{\Omega_q} \langle \hat{E}(t, x, x) - E_0(t, x, x), \rho_q \rangle \dd x\] 
 with $Q_q = \{x \in Q: |x|\leqslant a_0 + \ell_q + 1\}.$ For odd dimensions the kernel $E_0(t, x, x)$ vanishes for $t > 0$. For even dimensions, $x \mapsto E_0(t, x, x)$ is smooth for any $t > 0$ and we can easily estimate 
 \[\left|\int_{\Omega_q} \langle E_0(t, x, x), \rho_q \rangle \dd x\right| \leqslant A_0 m_q^{-1}\]
 with $A_0 > 0$ independent of $q$ by using the representation of the kernel $E_0(t, x, y)$ by oscillatory integrals with phases $\langle x- y, \eta \rangle \pm t$ (see for example, \cite[\S 3.1]{petkov2017geometry}).
 
 Now, choose  $g \in C_c^{\infty} (Q_q)$ and write the kernel $E(t, x, y) g(y)$ of $\cos(t \sqrt{-\Delta})g(y)$ as 
 \[E(t, x, y)g(y) = (2 \pi)^{-d} \int_{{\mathbb S}^{d-1}} d\eta \int_0^{\infty} k^{d-1} u(t, x; k ,\eta)e^{- i k \langle y, \eta \rangle} g(y) dk,\]
 where $u(t, x; k, \eta)$ is the solution of the problem
 \[
 \begin{cases}
 (\partial_t^2 - \Delta_x) u =  0 \quad  \text{in} \quad \R_t \times Q,\\
 u = 0 \quad \text{on} \quad \R_t \times \partial Q,\\
 u(0, x) = \wt g(x) \e^{i k \langle x, \eta \rangle}, \quad \partial_tu(0, x) = 0,
 \end{cases} 
 \] 
 with a function   $\wt g \in C_0^{\infty} (Q)$ equal to 1 on $\supp g.$ In the works  \cite{ikawa1985trap, ikawa1988fourier, ikawa1988analysis, burq1993obstacles}) of Ikawa and Burq,  asymptotic solutions $w^{(N)} = w_{q, +}^{(N)} + w_{q, -}^{(N)}$ of the above problem have been constructed. They have the form  
 \[
 w_{q, \pm}^{(N)}(t, x; k, \eta) =\sum_{|{\bf j}|d_0 \leqslant a_0 + \ell_q + 1} \e^{i k (\varphi_{\j}^{\pm} (x, \eta) \mp t)} \sum_{h = 0}^N v_{{\j}, h}^{\pm} (t, x, \eta) (i k)^{-h}.
 \]
Here $\j = \{j_1, j_2,\dots,j_n\},\: j_k \in (1, \dots ,r),\: j_k \neq j_{k + 1},\: k = 1,2,\dots,n-1,\: |\j| = n$ is a configuration related to the rays reflecting successively on $\partial D_{j_1}, \partial D_{j_2},..., \partial D_{j_n}$ (see \S\ref{subsec:orientation}). The phases $\varphi_{\j}^{\pm}$ are constructed successively starting from $\langle x, \eta \rangle$ and following the reflections on obstacles determined by the configuration $\j.$ The amplitudes $v_{{\j}, h}^{\pm} $ are determined by transport equations. The reader may consult \cite[\S3]{ikawa1985trap}, \cite[Equations (3.2) and (3.3)]{ikawa1988analysis}, \cite[\S4]{ikawa1988fourier} and \cite{burq1993obstacles}  for the construction of $v_{\j, h}^{\pm}$. The function $u - w^{(N)}$ is solution of the problem
\[
\begin{cases}
(\partial_t^2 - \Delta_x) (u - w^{(N)})=  k^{-N} F_N(t, x; k, \eta) \quad \text{in} \quad \R_t \times Q,\\
u - w^{(N)}= k^{-N} b_N(t, x; k, \eta)\quad \text{on} \quad \R_t \times \partial Q,\\
(u- w^{(N)})(0, x; k, \eta) = \partial_t(u - w^{(N)})(0, x; k, \eta)  = 0.\end{cases}
\]
Here $F_N$ is obtained as the action of $(\partial_t^2 - \Delta_x)$ to the amplitudes $v_{{\j}, N}^{\pm},$ while $b_N$ is obtained by the traces on $\partial Q$ of the amplitudes $v_{{\j}, N}^{\pm}$. 
It is important to note that the asymptotic solutions $w^{(N)}$ are independent of the sequence $(m_q)$. The integral 
involving $u - w^{(N)}$ is easily estimated and it yields a term ${\mathcal O}(m_q^{-1})$ (see \cite{ikawa1985trap}). For the integral involving $w_{q, \pm}^{(N)}$ one applies the stationary phase argument as $k \to \infty$ for the integration with respect to $x \in Q_q,\: \eta \in {\mathbb S}^{d-1}$, considering $t$ as a parameter. Next, in \cite{ikawa1988fourier}, estimates of the derivatives of order $p$ of $v_{\j, h}^{\pm}(x, t, \eta)$ with respect to $x \in Q_q, \: \eta \in {\mathbb S}^{d-1}$ with bound $C_pe^{-\alpha_2 \ell_q} (t+ 1)^h,\: \alpha_2 > 0$ have been established. Here $C_p > 0$ and $\alpha_2 > 0$ are independent of $\ell_q$.  By using a partition on unity $\sum_j \psi_j(x) = 1$ on  $Q_q,$ for large fixed $N$ one deduces  the estimate
\[\big|\langle u -  \hat F_\mathrm D, \rho_q \rangle \big| \leqslant A \e^{-\alpha_2 \ell_q} \# \big\{\j: |\j| \leqslant \frac{2 \ell_q}{d_0} \big \} \ell_q^{2N + 2} m_q^{-1}\]
with constant $A > 0$ independent of $q$. Finally, since 
  \[\# \big\{\j: |\j| \leqslant \frac{2 \ell_q}{d_0} \big \} \leqslant \e^{\alpha_3 \ell_q},\: \: \forall q \geqslant 1\]
  with  constant $\alpha_3 > 0$ independent of $q$,
  we obtain (\ref{eq:B4}).
     
     Combining Proposition \ref{prop:2.3ika} and the estimates (\ref{eq:B1}) and (\ref{eq:B4}), it is easy to obtain a contradiction with the assumption that $P(\alpha, \delta) < \infty$ for all $\alpha \geqslant 1.$ Indeed, let
     \[\alpha = \frac{(2d+3)}{1 - \delta} ( \alpha_0 + \alpha_1 + 1),\quad \beta = \frac{\alpha}{2d+3}.\]
Then 
 \[ m_q^{d+1} \e^{- \alpha \ell_q} \leqslant \e^{(d+1)2\beta \ell_q}  \e^{-\alpha \ell_q}= \e^{-\beta \ell_q} \leqslant e^{- \beta( 1 -\delta)\ell_q}\]
 and 
 \[\alpha_1 + \alpha_0 -\beta  = \alpha_0 + \alpha_1  -\frac{\alpha_0 + \alpha_1 +1}{1- \delta} = -1 - \frac{\delta(\alpha_0 + \alpha_1 + 1)}{1- \delta}.\]
  From (\ref{eq:B1}), (\ref{eq:B4}) and Proposition B.2 one deduces
\[\Bigl(C_0 \e^{\alpha}  + C_1 P(\alpha, \delta)\Bigr) \e^{-\beta (1-\delta)\ell_q} \geqslant C_0 e^{\alpha} m_q^{d+1}e^{-\alpha \ell_q} + C_1 P(\alpha, \delta)e^{-\beta(1 - \delta)\ell_q}\]
\[\geqslant |\langle u, \rho_q\rangle| \geq\e^{-\alpha_0 \ell_q} - C \e^{\alpha_1 \ell_q} \e^{-\beta \ell_q} \]
\[= \e^{-\alpha_0 \ell_q}\Bigl( 1 - C e^{-\ell_q - \frac{\delta(\alpha_0 + \alpha_1 + 1)}{1- \delta}\ell_q}\Bigr).\]   
 Since $\beta( 1- \delta)  > \alpha_0,$ letting $q \to \infty$ we obtain a contradiction. This completes the proof of Theorem 3.
   
\bibliographystyle{alpha}
\bibliography{bib.bib}

\end{document}